\documentclass[11pt,a4paper,fleqn]{article}
\usepackage[T1]{fontenc}
\usepackage{lmodern}
\usepackage{amssymb,amsthm,mathtools,mathrsfs}
\usepackage{enumitem,microtype}
\usepackage[margin=23mm]{geometry}
\usepackage{xcolor}
\usepackage[numbers,sort&compress]{natbib}
\usepackage{hyperref}
\usepackage[capitalise,nameinlink,noabbrev]{cleveref}
\hypersetup{colorlinks=true,linkcolor=blue,citecolor=blue,urlcolor=blue,
 pdftitle={Pure Type-One Generalized Hypercomplex Structures in Dimension Four},
 pdfauthor={Guangzhen Ren}}
\allowdisplaybreaks[1]
\numberwithin{equation}{section}
\newcommand{\RR}{\mathbb R}
\newcommand{\ZZ}{\mathbb Z}
\newcommand{\QQ}{\mathbb Q}
\newcommand{\ii}{\mathrm{i}}
\newcommand{\Id}{\operatorname{Id}}
\newcommand{\Span}{\operatorname{span}}
\newcommand{\rank}{\operatorname{rank}}
\newcommand{\vol}{\operatorname{vol}}
\newcommand{\im}{\operatorname{im}}
\newcommand{\Lie}{\mathcal L}
\newcommand{\Ric}{\operatorname{Ric}}
\newcommand{\diver}{\operatorname{div}}
\theoremstyle{plain}
\newtheorem{maintheorem}{Theorem}

\newtheorem{proposition}{Proposition}[section]
\newtheorem{lemma}[proposition]{Lemma}
\newtheorem{corollary}[proposition]{Corollary}
\theoremstyle{definition}
\newtheorem{example}[proposition]{Example}
\newtheorem{remark}[proposition]{Remark}
\crefname{maintheorem}{Theorem}{Theorems}
\Crefname{maintheorem}{Theorem}{Theorems}
\crefname{proposition}{Proposition}{Propositions}
\crefname{lemma}{Lemma}{Lemmas}
\crefname{corollary}{Corollary}{Corollaries}
\crefname{example}{Example}{Examples}
\crefname{remark}{Remark}{Remarks}

\title{Pure Type-One Generalized Hypercomplex\\Structures in Dimension Four}
\author{Guangzhen Ren\thanks{\raggedright Department of Mathematics,
Zhejiang International Studies University, Hangzhou, China.
Email: \texttt{gzren@zisu.edu.cn}.}}
\date{}

\begin{document}
\maketitle

\begin{abstract}
We classify untwisted generalized hypercomplex structures on four-manifolds whose entire sphere has type one. Their Poisson tensors determine a canonical line field, and their local geometry is described by Gauduchon--Tod coframes. On a closed manifold, the canonical foliation is automatically Riemannian: the non-Riemannian alternative would require a twisting three-form with a nonzero period. Without assuming a circle action, we classify the compact structures into four families, distinguished by their first Betti numbers \(b_1=1,2,3,4\): Berger quotients, flat axis quotients, nontrivial circle bundles over flat three-tori, and constant tori. The underlying manifolds are parallelizable and admit symplectic forms precisely when the intrinsic Gauduchon--Tod invariant vanishes. These results address the four-dimensional untwisted case of the intermediate-type problem posed by Fino and Grantcharov.
\end{abstract}



\section{Introduction}
Generalized hypercomplex geometry includes ordinary hypercomplex
and holomorphic symplectic structures in a common twistor framework
\cite{Stienon2009,FinoGrantcharov2025}. The type distinguishes these
classical geometries from spheres of intermediate type.

An orthogonal quaternionic triple
\[ \mathcal J_i\mathcal J_j=-\delta_{ij}\Id
                         +\varepsilon_{ijk}\mathcal J_k\]
on \(TM\oplus T^*M\) is generalized hypercomplex if every member
\(\mathcal J_a=\sum_i a_i\mathcal J_i\), \(a\in S^2\), is an
integrable generalized complex structure
\cite{Hitchin2003,Gualtieri2011,Stienon2009}.
It suffices that two anticommuting generators are integrable:
all six Nijenhuis concomitants then vanish
\cite[Proposition~12, (i)$\Rightarrow$(iii)]{Stienon2009}, and
bilinearity gives integrability of the sphere.
We study the \emph{odd} sector in real dimension four. Pure-spinor
parity is locally constant on $M\times S^2$, and type lies in
$\{0,1,2\}$ \cite[Definition~1.1 and Proposition~1.3]{Gualtieri2011}.
Thus, on connected $M$, type one at one point and one parameter
forces type one everywhere: every real Poisson tensor $\pi_a$ has
rank two. We call these structures \emph{pure type-one}.
Throughout, triples are globally ordered, all data are smooth,
and the initial Courant bracket is untwisted; compact manifolds
have no boundary. Equivalence means a diffeomorphism, a common
closed real $B$-field and a constant $\operatorname{SO}(3)$ rotation.

Fino and Grantcharov ask for a description when the twistor-space
type lies in $2\leq\operatorname{type}(\mathcal I_Z)\leq2n$ on a
$4n$-dimensional base \cite[p.~106, after Theorem~4.4]{FinoGrantcharov2025}.
For $n=1$ this interval consists of the single value $2$.
The twistor structure $\mathcal I_Z$ on $Z=M\times\mathbb{CP}^1$
restricts to $\mathcal J_a$ in the $M$-directions and to the ordinary
complex structure in the sphere directions
\cite[Section~4, Theorem~4.2]{FinoGrantcharov2025}.
The latter contributes no Poisson block, so
\[\operatorname{type}(\mathcal I_Z)(x,a)
 =\frac{6-\rank\pi_a(x)}2
 =1+\operatorname{type}(\mathcal J_a)(x).\]
Thus, on a connected four-manifold, intermediate twistor type at
one point forces twistor type two everywhere. Theorem~A describes
this sector for the untwisted Courant bracket; Theorem~B classifies
its compact members. Their untwisted Kodaira--Thurston example
\cite[Proposition~5.3]{FinoGrantcharov2025}, whose entire sphere
has type one, belongs to the circle-bundle family with $b_1=3$ below;
its classification parameters are computed in \cref{ex:FG-parameters}.

Pantilie classified generalized quaternionic vector spaces up to
linear \(B\)-field transformations
\cite[Theorem~4.4]{Pantilie2014} and constructed generalized
quaternionic structures from three-dimensional Einstein--Weyl
geometry \cite[Example~5.3]{Pantilie2014}.
Our pointwise normal form specializes that algebraic framework to
dimension four. We determine its realization by an integrable
ordered sphere, equivalence under closed $B$-fields, and the
restrictions imposed by compactness and the untwisted bracket.
The foliation criterion of Chen and Nie \cite{ChenNie2021} concerns
the existence of a single odd-type generalized complex structure,
allowing a closed twisting three-form; the simultaneous untwisted
quaternionic sphere imposes additional conditions.
Ren, Tang and Wu \cite{RenTangWu2026} construct generalized
hypercomplex structures from rank-three generalized Clifford data
and prove integrability of the associated twistor structure.

The essential local datum is a Gauduchon--Tod (GT) coframe
$\eta=(\eta^1,\eta^2,\eta^3)$ on a three-manifold $N$. Set
$h:=\sum_{i=1}^{3}\eta^i\otimes\eta^i$ and orient $N$ by the
coframe. The GT equations are
\begin{equation}\label{eq:GT}
 d\eta^i=\theta\wedge\eta^i-\kappa *_h\eta^i,
 \qquad i=1,2,3,
\end{equation}
where $\theta\in\Omega^1(N)$ is a real one-form and
$\kappa\in C^\infty(N,\mathbb R)$ is a real-valued function.
For a fixed coframe, \eqref{eq:GT} determines $\theta$ and
$\kappa$ uniquely: wedging the difference of two such systems
with $\eta^i$ gives equality of $\kappa$, and the remaining
three identities give equality of $\theta$.
These equations and their hyperCR interpretation are classical
\cite{GauduchonTod1998,CalderbankPedersen2000}.
For invariant circle-bundle models, the Courant isomorphism of
Cavalcanti and Gualtieri \cite{CavalcantiGualtieri2010} also transports
generalized structures by T-duality. Here the canonical foliation
is recovered intrinsically from the Poisson tensors; no circle
action or invariance is assumed in the compact classification.

The local normal form supplies a transverse Weyl connection and
an ordered transverse projective frame. Rigal's rigidity theorem
reduces the non-Riemannian case to a positive scalar-affine structure;
the GT equations then force $\kappa=0$.
The resulting Hopf model requires a twist with nonzero three-form
period, whereas an untwisted lift would make that twist exact.
This proves automatic Riemannianity. The nonflat branch is then
taut and has Berger transverse geometry; in the flat branch, the
developing image separates the axis models from the Euclidean
models. Elliptic equations determine the remaining lifting data,
and preservation of the ordered coframe determines the deck groups
and the finite-covering criteria.

\begin{maintheorem}\label{thm:main-local}
An untwisted odd generalized hypercomplex structure on a
four-manifold $M$ determines a unique line bundle $L\subset TM$
such that
\[
 \Span\{\pi_1,\pi_2,\pi_3\}=L\wedge TM,
 \qquad TM\cong L^{\oplus4}.\]
In suitable local coordinates $(t,x)\in I\times N^3$, and up to
a common closed real $B$-field transformation, the pure spinors
of the triple are
\begin{equation}\label{eq:normal-spinors}
 \rho_i
 =e^{-t*_h\theta}
  (\eta^j-\ii\eta^k)
  e^{\ii\,dt\wedge\eta^i},
 \qquad (i,j,k)\text{ cyclic}.
\end{equation}
Here $\eta=(\eta^1,\eta^2,\eta^3)$ is a coframe on $N$ satisfying
the GT equations~\eqref{eq:GT} for a real one-form $\theta$ and
a real-valued function $\kappa$, and
$
 h=\sum_{i=1}^{3}\eta^i\otimes\eta^i.$
The transverse forms are understood to be pulled back from $N$
to $I\times N$. Conversely, every such GT coframe defines an
untwisted pure type-one generalized hypercomplex structure.

Two such local models are equivalent if and only if
\begin{equation}\label{eq:main-local-equivalence}
 \psi^*\eta'=a^{-1}R\eta
\end{equation}
for a local diffeomorphism $\psi$, a nowhere-zero real function
$a$, and a constant matrix $R\in\operatorname{SO}(3)$. In this
case,
\[
 \psi^*\theta'=\theta-d\log|a|,
 \qquad
 \psi^*\kappa'=a\kappa.
\]
The transverse equivalence is realized by
$(t,x)\mapsto(a(x)t,\psi(x))$. The additional time translations
and their closed $B$-field compensators are described in
\cref{thm:classification}.
\end{maintheorem}

The line $L$ defines the \emph{canonical foliation} $\mathcal F$.
It is \emph{Riemannian} if $TM/L$ has a holonomy-invariant positive
metric, and \emph{taut} if some metric on $M$ makes its leaves
geodesic. A local generator $X$ is adapted if it preserves the
ordered Poisson tensors; writing $\pi_i=X\wedge E_i$ then defines
the basic coframe dual to $[E_i]\in TM/L$. For another adapted
generator $X'=fX$, where $Xf=0$, the GT equations give
\begin{equation}\label{eq:nonflat-intrinsic-generator}
\left\{
\begin{aligned}
&X'=fX,\\
&\eta'=f\eta,\\
&\theta'=\theta+d\log|f|,\\
&\kappa'=\kappa/f,\\
&K:=\kappa X=\kappa'X'\in\Gamma(L).
\end{aligned}
\right.
\end{equation}
In particular, $X\mapsto-X$ sends $(\eta,\kappa)$ to
$(-\eta,-\kappa)$, so $K$ is defined even when $L$ is nonorientable.
The global section $K$ is the \emph{GT invariant}; its scalar
coefficient $\kappa$ depends on the generator. The term \emph{flat}
means $K\equiv0$, equivalently $\kappa\equiv0$.
A finite \emph{ordered}
quotient is the quotient by a finite subgroup of Courant automorphisms
acting freely on $M$ and preserving the three structures individually,
allowing common closed $B$-field compensators.

\begin{maintheorem}\label{thm:intro-compact}
Let $M$ be a closed connected four-manifold with an untwisted odd
generalized hypercomplex structure. Its canonical foliation is
Riemannian, and its structure belongs, up to equivalence, to exactly
one of the following four families:
\[
 \renewcommand{\arraystretch}{1.25}
 \begin{array}{c|l|c}
 b_1(M)&\text{model family}&\text{$M$ is finitely covered by}\\\hline
 1&\text{Berger quotients (including round)}&S^3\times S^1\\
 2&\text{non-taut flat axis quotients}&N_C\times S^1\\
 3&\text{circle bundles over flat }T^3&\mathrm{KT}\\
 4&\text{constant tori}&T^4
 \end{array}
\]
Here $N_C$ is the mapping torus of $C\in \operatorname{SL}(2,\ZZ)$ with
$\operatorname{tr}C>2$, and $\mathrm{KT}$ is the product of $S^1$
with the Euler-one circle bundle over $T^2$. The circle bundles in
the third row have nonzero Euler class. Every family occurs.

In each of the four families, $M$ is parallelizable and
$b_2(M)=2b_1(M)-2$. Moreover, the intrinsic invariant $K$
vanishes identically if and only if $b_2(M)>0$, or equivalently,
if and only if $M$ admits a symplectic form.
If $K\not\equiv0$, it is nowhere zero and the canonical
foliation is oriented and taut. Whenever all canonical leaves
are closed, they form a principal circle bundle without
exceptional fibres.
\end{maintheorem}

\begin{remark}\label{rem:four-geometries}
As smooth manifolds, the finite covers in the table admit, respectively,
the product geometries
\[S^3\times\mathbb E^1,\qquad
 \operatorname{Sol}^3\times\mathbb E^1,\qquad
 \operatorname{Nil}^3\times\mathbb E^1,\qquad\mathbb E^4.\]
The structure metrics \eqref{eq:Berger-metric-explicit} are Berger
metrics for $a>0$ and round for $a=0$; both belong to the nonflat
GT branch. The axis models have transverse
developing image equal to a line complement.
Examples are obtained from the nonflat deck group
$\langle(1,T)\rangle$, $T>0$; the axis data
\[
 C=\begin{pmatrix}2&1\\1&1\end{pmatrix},\qquad
 v=k=0,\quad n=1\]
in \eqref{eq:axis-oriented-decks}; the circle bundle over
$\RR^3/\ZZ^3$ with $e=[dx^1\wedge dx^2]$ and $\mu=0$;
and a constant triple on $\RR^4/\ZZ^4$.
\end{remark}

The parameter equivalences for the four rows are
\eqref{eq:nonflat-deck-conjugacy}, \eqref{eq:axis-full-equivalence},
\eqref{eq:closed-flat-arithmetic-equivalence} and
\eqref{eq:euclidean-double-coset}, respectively.
\Cref{cor:nonflat-explicit-parameters} gives explicit nonflat parameters.
Finite-covering criteria are given in
\cref{cor:nonflat-finite-cover-criterion,cor:axis-all-finite-quotients,prop:closed-flat-finite-coverings,thm:euclidean-lattice-moduli};
normal covering subgroups give finite ordered quotients.
The last column alone is not an existence criterion for arbitrary
finite quotients. The moduli spaces are orbit sets; no smoothness or
Hausdorffness is asserted. The subgroup criteria retain the data of
a specified covering in addition to the equivalence class of its target.
The higher-dimensional intermediate-type problem is not addressed.

\paragraph{Organization.}
Section~\ref{sec:local} proves Theorem~A; Section~\ref{sec:compact-foliations}
proves automatic Riemannianity, develops compatible Killing generators
and classifies the four compact branches. Appendix~\ref{app:scalar-dynamics}
supplies the scalar-affine dynamics used in the Riemannianity argument.

\subsection*{Notation and conventions}
We use the natural pairing and Clifford action
\[
 \langle X+\xi,Y+\zeta\rangle
 =\tfrac12\bigl(\xi(Y)+\zeta(X)\bigr),
 \qquad
 (X+\xi)\cdot\rho=\iota_X\rho+\xi\wedge\rho.
\]
We use $\pi^\sharp\xi=\iota_\xi\pi$ and
$B^\flat X=\iota_XB$. Thus
$\iota_\xi(X\wedge Y)=\xi(X)Y-\xi(Y)X$, and the common
$B$-field transformation is
\[
 T_B(X+\xi)=X+\xi-\iota_XB,\qquad \rho\longmapsto e^B\rho.\]
For $d_H=d+H\wedge$ the sign convention is
\begin{equation}\label{eq:twist-conjugation}
 d_{H-dB}(e^B\rho)=e^B d_H\rho.
\end{equation}
The annihilator of a pure spinor is the $+\ii$ eigenspace.
Products of differential forms mean wedge products. We write
$\eta^{ij}=\eta^i\wedge\eta^j$ and $\eta^{123}=\vol_h$;
the transverse Hodge star always uses the ordered coframe orientation.
For a connection $D$, our curvature and Ricci conventions are
\[
 R^D(X,Y)=[D_X,D_Y]-D_{[X,Y]},\qquad
 \Ric^D(Y,Z)=\operatorname{tr}\bigl(X\mapsto R^D(X,Z)Y\bigr).\]
Our Laplacians have opposite sign conventions:
\[
 \Delta_h=\diver_h\operatorname{grad}_h,\qquad
 \Delta_H=d\delta+\delta d\geq0,\qquad \Delta_H|_{\Omega^0}=-\Delta_h.\]
The symbols $L$, $K$, $H$, $\eta$, $\theta$ and $\kappa$ denote,
respectively, the canonical line, the intrinsic GT section, the twist,
the transverse coframe, its Weyl form and its GT coefficient.
Local scalar coefficients and deck groups are specified within each model.

\section{Proof of Theorem A}\label{sec:local}
For the parity reduction, let $p:M\times S^2\to M$.
The images of $(\Id-\ii \mathcal J_a)/2$ form a smooth maximal
isotropic bundle. Their annihilators form a pure-spinor line
\[
 \mathcal K\subset p^*\bigl(\Lambda^{\mathrm{even}}T^*M\otimes\mathbb C\bigr)
       \oplus p^*\bigl(\Lambda^{\mathrm{odd}}T^*M\otimes\mathbb C\bigr).\]
Its parity equals type modulo two and is locally constant
\cite[Proposition~1.3]{Gualtieri2011}. On each connected component
of $M\times S^2$, the type lies in $\{0,1,2\}$; hence type one at
a single point forces type one throughout that component.

\begin{proposition}\label{thm:canonical-line}
The Poisson tensors of a pure type-one triple determine a unique
smooth line bundle $L\subset TM$ with
\[
 \Span\{\pi_1,\pi_2,\pi_3\}=L\wedge TM,
 \qquad TM\cong L^{\oplus4}.\]
It is unchanged by common $B$-fields and constant rotations. In
particular $w_1(TM)=w_2(TM)=p_1(TM)=0$. On a closed connected
manifold, $TM$ is trivial, $\chi=\sigma=0$ and $b_2=2b_1-2$.
\end{proposition}

\begin{proof}
Fix a point $p\in M$ and work in $T_pM$. For every nonzero
$c=(c_1,c_2,c_3)\in\RR^3$, the bivector
$\sum_{i=1}^3 c_i\pi_i(p)$
has rank two. In particular, $\pi_1(p),\pi_2(p),\pi_3(p)$ are
linearly independent. Moreover, every rank-two bivector has zero
square, so polarization gives
\begin{equation}\label{eq:local-poisson-pencil}
 W:=\Span\{\pi_1(p),\pi_2(p),\pi_3(p)\},\qquad
 \dim W=3,\qquad
 \pi_i(p)\wedge\pi_j(p)=0
 \quad(i,j=1,2,3).
\end{equation}
By bilinearity, $\pi\wedge\widetilde\pi=0$ for all
$\pi,\widetilde\pi\in W$.

Choose two linearly independent elements $\pi_1',\pi_2'\in W$.
Since every nonzero element of $W$ has rank two, both are
decomposable. The relation $\pi_1'\wedge\pi_2'=0$ shows that
their corresponding two-planes intersect nontrivially. These
planes are distinct, because $\pi_1'$ and $\pi_2'$ are linearly
independent, and hence their intersection is a line. We may
therefore write
\[
 \pi_1'=X\wedge Y,\qquad
 \pi_2'=X\wedge Z,\]
where $X,Y,Z$ are linearly independent. Extend them to a basis
$(X,Y,Z,T)$ of $T_pM$.
Choose $\pi\in W$ such that
$\{\pi_1',\pi_2',\pi\}$ is a basis of $W$, and write
\begin{equation}\label{details}
 \pi=aX\wedge Y+bX\wedge Z+cX\wedge T
     +dY\wedge Z+eY\wedge T+fZ\wedge T.
\end{equation}
By the last relation in \eqref{eq:local-poisson-pencil} and
bilinearity of the wedge product, we have
\begin{equation}\label{details02}
 \pi\wedge(X\wedge Y)=0,
 \qquad
 \pi\wedge(X\wedge Z)=0.
\end{equation}
Substituting \eqref{details} into \eqref{details02} gives
$f=0$ and $e=0$, respectively. Since $\pi$ is independent
of $X\wedge Y$ and $X\wedge Z$, we have $(c,d)\ne(0,0)$.
The remaining relation $\pi\wedge\pi=0$ then yields
$
 0=2cd\,X\wedge Y\wedge Z\wedge T,$
so $cd=0$.\\
If $c\ne0$, then $d=0$ and
\[
 W=\Span\{X\wedge Y,X\wedge Z,X\wedge T\}
  =\RR X\wedge T_pM.\]
If $d\ne0$, then $c=0$ and
\[
 W=\Span\{X\wedge Y,X\wedge Z,Y\wedge Z\}
  =\Lambda^2\Span\{X,Y,Z\}.\]
Thus these are the only two possibilities for $W$.

Write the generalized complex structures at $p$ in block form:
\begin{equation}\label{details03}
 \mathcal J_i=
 \begin{pmatrix}
  A_i & P_i\\
  S_i & -A_i^*
 \end{pmatrix},
 \end{equation}
where $P_i=\pi_i^\sharp.$ Suppose that the second possibility holds, so that
$ W=\Lambda^2U$ and $U:=\Span\{X,Y,Z\}\subset T_pM.$
Since $U$ has dimension three, there is a nonzero covector
$\xi\in T_p^*M$ satisfying
$
 \xi(X)=\xi(Y)=\xi(Z)=0.$
As $\pi_i(p)\in\Lambda^2U$, we have
\begin{equation}\label{details04}P_i\xi=0.\end{equation}
Identifying $T_p^*M$ with the subspace
$\{0\}\oplus T_p^*M$ of $T_pM\oplus T_p^*M$, set
$
 q=(0,\xi).$
Then by (\ref{details04}), we have
\[
 \mathcal J_iq=(0,-A_i^*\xi)\in T_p^*M
 \qquad(i=1,2,3).
\]
The quaternionic identities imply that $q,\ \mathcal J_1q,\ \mathcal J_2q,\ \mathcal J_3q$
are linearly independent over $\RR$. Indeed, if
\[
 (a_0\Id+a_1\mathcal J_1+a_2\mathcal J_2+a_3\mathcal J_3)q=0,\]
then applying
$a_0\Id-a_1\mathcal J_1-a_2\mathcal J_2-a_3\mathcal J_3$
gives
\[
 (a_0^2+a_1^2+a_2^2+a_3^2)q=0.\]
Since $q\ne0$, all coefficients vanish.
The four vectors above therefore form a basis of the
four-dimensional space $T_p^*M$. The quaternionic identities
also show that their span is preserved by each $\mathcal J_i$.
Consequently,
\[
 \mathcal J_i(T_p^*M)\subseteq T_p^*M
 \qquad(i=1,2,3).\]
By the block form of $\mathcal J_i$, this implies $P_i=0$ for
every $i$, contradicting $\rank P_i=2$. Thus the second
possibility is excluded.

It follows that $W=L_p\wedge T_pM$ for a unique line
$L_p\subset T_pM$, namely
\[
 L_p=\left\{
 v\in T_pM:
 v\wedge\pi_i(p)=0,\quad i=1,2,3
 \right\}.\]
As $p$ varies, these lines are the kernels of a smooth bundle
map of constant rank. They therefore form a smooth line bundle
$L\subset TM$. This characterization also shows that $L$ is
unchanged by common $B$-field transformations and constant
rotations of the triple.

The ordered Poisson tensors trivialize $L\otimes(TM/L)$, so
$
 L\otimes(TM/L)\cong\underline{\RR}^{\,3}.$
Choosing a splitting of $TM\to TM/L$ gives
\[
 TM\cong L\oplus(L^*)^{\oplus3}
    \cong L^{\oplus4}.\]
Set $x=w_1(L)$ and $u=c_1(L\otimes\mathbb C)$. Then
\[
 w(TM)=(1+x)^4,\qquad
 c(TM\otimes\mathbb C)=(1+u)^4,\qquad
 2u=0.\]
It follows that $w_1(TM)=w_2(TM)=0$ and
\[
 p_1(TM)=-c_2(TM\otimes\mathbb C)
        =-6u^2=0,
\]
because $6u^2=3(2u)u=0$.

The orientation cover of $L$ trivializes $L$ and is therefore
parallelizable. If $M$ is closed and connected, Euler
multiplicativity, the signature theorem and Poincar\'e duality
give
\[
 \chi(M)=0,\qquad
 \sigma(M)=\tfrac13\langle p_1(TM),[M]\rangle=0,\qquad
 b_2(M)=2b_1(M)-2.
\]

Finally, the isomorphism $TM\cong L^{\oplus4}\cong L\otimes\mathbb H$
gives global endomorphisms $I_1,I_2,I_3$ induced by left quaternionic
multiplication. These are well defined because quaternionic
multiplication commutes with the real transition functions of $L$.
Since $M$ is closed and oriented, $H^4(M;\ZZ)\cong\ZZ$; thus
$\chi(M)=0$ implies that the Euler class of $TM$ vanishes.
There is consequently a nowhere-zero vector field $V$, and
$(V,I_1V,I_2V,I_3V)$
is a global frame of $TM$. Hence $M$ is parallelizable.
\end{proof}

Fix an adapted frame $(e_0,e_1,e_2,e_3)$, with dual $(e^0,e^1,e^2,e^3)$,
and put
\[
 P_i=(e_0\wedge e_i)^\sharp,\quad
 S_i^0=(e^0\wedge e^i)^\flat,\quad
 C_ie_0=0,\quad C_ie_j=\varepsilon_{ijk}e_k.
\]
The standard quaternionic triple and its spinors are
\begin{equation}\label{eq:standard-matrix}
 \mathcal J_i^0=\begin{pmatrix}C_i&P_i\\S_i^0&-C_i^*\end{pmatrix},
 \qquad (e^j-\ii e^k)e^{\ii  e^0\wedge e^i},
 \quad (i,j,k)\text{ cyclic}.
\end{equation}
The cross-product identity gives
$\mathcal J_i^0\mathcal J_j^0=-\delta_{ij}\Id+
\varepsilon_{ijk}\mathcal J_k^0$.
Write $e=(e^1,e^2,e^3)^T$. For a constant
$R\in\operatorname{SO}(3)$, keeping $e^0$ fixed and replacing
the transverse coframe $e$ by $Re$ gives
$\mathcal J_i^0(R e)=\sum_jR_{ij}\mathcal J_j^0(e)$.
Thus each member of the sphere is the first member of a rotated
standard triple.

\begin{lemma}\label{prop:linear-B-normal}
Every orthogonal quaternionic triple with these ordered Poisson
blocks is $T_B\mathcal J_i^0T_B^{-1}$ for a unique real two-form
$B$, depending smoothly on the triple and frame.
\end{lemma}

\begin{proof}
Put $D_i=A_i-C_i$. The upper-right quaternionic identities give
$D_iP_j-P_iD_j^*=0$. Relative to
$\RR e_0\oplus\RR^3$, with $(q_1,q_2,q_3)$ the standard basis,
write
\[
 P_i=\begin{pmatrix}0&-q_i^T\\q_i&0\end{pmatrix},\qquad
 D_i=\begin{pmatrix}a_i&\beta_i\\u_i&M_i\end{pmatrix}.
\]
Equating blocks gives
\[
 \beta_iq_j+\beta_jq_i=0,\qquad
 M_iq_j=a_jq_i,\qquad
 u_iq_j^T+q_iu_j^T=0.
\]
Taking $i=j$ in the last equation gives $u_i=0$. Thus, if
$a=(a_1,a_2,a_3)^T$ and $Q$ has rows $\beta_i$, then
$Q^T=-Q$, $M_i=q_i a^T$, and
\begin{equation}\label{eq:local-algebraic-gauge}
 D_i=P_iB^\flat,\qquad
 B^\flat=\begin{pmatrix}0&a^T\\-a&-Q\end{pmatrix}.
\end{equation}
By \eqref{eq:local-algebraic-gauge}, transformation by $T_{-B}$
sets every upper-left block equal to $C_i$.
The upper-left quaternionic identities then give
$P_i(S_j-S_j^0)=0$ for all $i,j$, so $S_j=S_j^0$ because
$\bigcap_i\ker P_i=0$. This intersection also proves uniqueness;
\eqref{eq:local-algebraic-gauge} proves smooth dependence.
\end{proof}

\begin{proposition}\label{prop:GT-reduction}
Locally there are coordinates $(t,x)\in I\times N^3$ with
$\pi_i=\partial_t\wedge E_i(x)$, where $E_i$ is a transverse
frame. Write $\eta^i$ for its dual, $h=\sum_i(\eta^i)^2$ and
$\alpha_i=\eta^j-\ii\eta^k$. The standard spinors
$\alpha_i e^{\ii \,dt\wedge\eta^i}$ are integrable for a closed twist
$H$ precisely when
\begin{equation}\label{eq:H-GT}
 d\eta^i=\theta\wedge\eta^i-\kappa*\eta^i,
 \qquad H=-dt\wedge*\theta+\tau\vol_h.
\end{equation}
Here $\theta,\kappa$ are independent of $t$.
\end{proposition}

\begin{proof}
Write $\pi_i=X\wedge Y_i$. The bivector of each generalized complex
structure is Poisson \cite[Proposition~3.21]{Gualtieri2011}.
Applying this to every $\sum_i a_i\mathcal J_i$ and polarizing gives
$[\pi_i,\pi_j]_{\mathrm{Sch}}=0$. The diagonal equations imply
$[X,Y_i]=a_iX+b_iY_i$, and for $i\ne j$,
\[
 0=[\pi_i,\pi_j]_{\mathrm{Sch}}
   =(b_j-b_i)X\wedge Y_i\wedge Y_j.
\]
Hence all $b_i$ equal $b$. Solve $X(\log f)=b$ with $f>0$,
and put $W=fX$. Then
\begin{equation}
\left\{
\begin{aligned}
& [W,f^{-1}Y_i]=c_iW,\\
 & W(g_i)=-c_i,\\
 &E_i:=f^{-1}Y_i+g_iW.
\end{aligned}
\right.
\end{equation}
Consequently,
\[
 [W,E_i]=0,
 \qquad
 \pi_i=W\wedge E_i.
\]
Choose the initial values of $g_i$ to make $E_i$ tangent to a
transverse slice. Flowing this slice by $W$ gives
$W=\partial_t$ and $E_i=E_i(x)$.

For a nondegenerate type-one spinor $\alpha e^{\ii \omega}$, the
pure-spinor criterion \cite[Theorem~2.9]{Gualtieri2011} reduces to
\begin{equation}\label{eq:spinor-criterion}
 \alpha\wedge d\alpha=0,
 \qquad (H+\ii\,d\omega)\wedge\alpha=0.
\end{equation}
Indeed, the degree-zero, degree-two and degree-four terms in
$ e^{-\ii\omega}d_H(\alpha e^{\ii \omega})
   =d\alpha+(H+\ii\,d\omega)\wedge\alpha$ and
 $e^{-\ii\omega}(\iota_Z+\xi\wedge)(\alpha e^{\ii \omega})
   =\alpha(Z)+(\xi+\ii\,\iota_Z\omega)\wedge\alpha$
give the necessity of \eqref{eq:spinor-criterion}.
Conversely $\alpha\wedge d\alpha=0$ gives
$d\alpha=\nu\wedge\alpha$, and \eqref{eq:spinor-criterion}
then gives $d_H\rho=\nu\wedge\rho$.
Write the vertical part of $H$ as
$dt\wedge(h_{012}\eta^{12}+h_{013}\eta^{13}+h_{023}\eta^{23})$.
For the three cyclic indices, the real and imaginary parts of
\eqref{eq:spinor-criterion} are equivalent to
\begin{equation}
\left\{
\begin{aligned}
 d\eta^1&=-h_{013}\eta^{12}+h_{012}\eta^{13}+c\eta^{23},\\
 d\eta^2&=-h_{023}\eta^{12}-c\eta^{13}+h_{012}\eta^{23},\\
 d\eta^3&=c\eta^{12}-h_{023}\eta^{13}+h_{013}\eta^{23}.
\end{aligned}
\right.
\end{equation}
Setting $\theta=-h_{023}\eta^1+h_{013}\eta^2-h_{012}\eta^3$
and $\kappa=-c$ gives \eqref{eq:H-GT}. Its coefficients are
independent of $t$ because the coframe is. Conversely the GT equations
give
\begin{equation}\label{eq:star-identity}
 d\alpha_i=(\theta+\ii\kappa\eta^i)\wedge\alpha_i,
 \qquad \alpha_i\wedge(*u+\ii\,u\wedge\eta^i)=0
\end{equation}
for every real transverse one-form $u$. Substitution into
\eqref{eq:spinor-criterion} proves the converse, since
$\vol_h\wedge\alpha_i=0$.
\end{proof}

\begin{proposition}
\label{prop:untwisting-existence}
Every untwisted lift in the adapted Poisson form is locally equivalent, by a common closed $B$-field and a change of canonical time, to \eqref{eq:normal-spinors}. Conversely every GT coframe produces that untwisted quaternionic sphere.
\end{proposition}

\begin{proof}
The algebraic gauge argument writes the original spinors as
$e^B\alpha_i e^{\ii \,dt\wedge\eta^i}$, where $H=dB$ has the form
\eqref{eq:H-GT}. We use $\Delta_h=\diver_h\operatorname{grad}_h$,
so $d(*dv)=(\Delta_hv)\vol_h$, and set
\begin{equation}\label{eq:Ltheta}
 d(*\theta)=q\vol_h,
 \qquad \mathscr P_\theta v=\Delta_hv+\langle\theta,dv\rangle_h+qv.
\end{equation}
By \eqref{eq:H-GT}, closure of $H$ gives
\[
 dH=(\partial_t\tau+q)\,dt\wedge\vol_h=0,
 \qquad \tau=-tq+\tau_0(x).
\]
Set $s=t+v(x)$ and $\widetilde B=B+*dv$.
Equation \eqref{eq:star-identity} identifies the spinors with
$e^{\widetilde B}\alpha_i e^{\ii\,ds\wedge\eta^i}$, while
\eqref{eq:Ltheta} gives
\begin{equation}\label{eq:local-time-gauge}
 d\widetilde B=-ds\wedge*\theta+
        \bigl(-sq+\tau_0+\mathscr P_\theta v\bigr)\vol_h.
\end{equation}
On a small coordinate ball $B_r$, impose $v|_{\partial B_r}=0$.
Under $x=ry$, set $w(y)=v(ry)$. The rescaled operator is
\[
 \mathscr P_{\theta,r}
 =a^{ij}(ry)\partial_i\partial_j+r b^i(ry)\partial_i+r^2q(ry),
 \qquad (a^{ij}(0))>0.
\]
For any fixed $0<\alpha<1$, this is a small perturbation, as an
operator from
$\{w\in C^{2,\alpha}(\overline B_1):w|_{\partial B_1}=0\}$
to $C^{0,\alpha}(\overline B_1)$, of the invertible
constant-coefficient Dirichlet operator
$a^{ij}(0)\partial_i\partial_j$.
Thus $\mathscr P_{\theta,r}w=-r^2\tau_0(ry)$ has a solution with
zero boundary values for small $r$. Then $v(x)=w(x/r)$ solves
$\mathscr P_\theta v=-\tau_0$; interior elliptic regularity makes
it smooth on the smaller ball used for the local normal form.
Equation \eqref{eq:local-time-gauge} then becomes
\[
 d\widetilde B=-d(s*\theta),\qquad
 d(\widetilde B+s*\theta)=0.
\]
Removing the common closed form $\widetilde B+s*\theta$ gives
\eqref{eq:normal-spinors}.

Conversely \eqref{eq:star-identity} and
\eqref{eq:twist-conjugation} give directly
\[
 d\rho_i=(\theta+\ii\kappa\eta^i)\wedge\rho_i.
\]
The GT equations are invariant under constant $\operatorname{SO}(3)$ rotations;
the same calculation applies to the entire quaternionic sphere.
Every nonzero $\pi_a=\partial_t\wedge\sum a_iE_i$ has rank two.
\end{proof}

\begin{proposition}\label{thm:classification}
Let $(\eta,\theta,\kappa)$ on $N$ and
$(\widehat\eta,\widehat\theta,\widehat\kappa)$ on
$\widehat N$ define two normal forms. They are locally equivalent
precisely when there are a local diffeomorphism
$\psi:N\to\widehat N$, a nowhere-zero real function $a$ and
a constant $R\in \operatorname{SO}(3)$ such that
\begin{equation}\label{eq:weyl-gauge}
\left\{\begin{aligned}
 \psi^*\widehat\eta
 &=a^{-1}R\eta,\\
 \psi^*\widehat\theta
 &=\theta-d\log|a|,\\
 \psi^*\widehat\kappa
 &=a\kappa.
\end{aligned}\right.
\end{equation}
For such data, every equivalence has an underlying map of the form
\begin{equation}\label{eq:local-equivalence-map}
\begin{aligned}
 F(t,x)
 &=\bigl(a(x)t+v(x),\psi(x)\bigr),\\
 \mathscr P_\theta u
 &=0,
\end{aligned}
\end{equation}
where $u:=v/a$.
Let $\rho_i^R$ denote the normal spinors associated with the rotated
coframe $R\eta$, and define $B_u:=-*_h(du+u\theta)$. Then
\begin{equation}\label{eq:local-equivalence-compensator}
\begin{aligned}
 F^*\widehat\rho_i
 &=a^{-1}e^{B_u}\rho_i^R,\\
 dB_u
 &=-(\mathscr P_\theta u)\vol_h.
\end{aligned}
\end{equation}
Since $\mathscr P_\theta u=0$, the inverse compensator $-B_u$
is closed.
\end{proposition}

\begin{proof}
Preservation of $L$ makes an equivalence
$(t,x)\mapsto(\varphi(t,x),\psi(x))$. Comparison of the
Poisson tensors gives
\[
 (\partial_t\varphi)\,d\psi(E_i)
   =\sum_j R_{ji}\widehat E_j,
 \qquad\partial_t^2\varphi=0,
 \qquad\varphi=at+v.
\]
Dualizing gives the first relation in \eqref{eq:weyl-gauge};
substitution into \eqref{eq:GT} gives the other two.

Pull all transverse data back by $\psi$ and first take $R=\Id$.
The respective coframe orientations give
$\widehat h=a^{-2}h$ and $\vol_{\widehat h}=a^{-3}\vol_h$,
so $\widehat *=a^{-1}*$ on one-forms for either sign of $a$.
By \eqref{eq:star-identity}, the two-form
$\Omega_i(\xi)=*\xi+\ii\,\xi\wedge\eta^i$ satisfies
$\Omega_i(\xi)\wedge\rho_i=0$, hence $e^{s\Omega_i(\xi)}\rho_i=\rho_i$.
For $F_0(t,x)=(a(x)t,\psi(x))$, taking $\xi=d\log|a|$ gives
\begin{equation}\label{eq:local-scaled-spinors}
 F_0^*\widehat\rho_i
  =a^{-1}\exp\!\left(
       t\{*d\log|a|+\ii\,d\log|a|\wedge\eta^i\}\right)\rho_i
  =a^{-1}\rho_i.
\end{equation}
For constant $R$, replace $\rho_i$ by $\rho_i^R$.
With $u=v/a$ and $\Phi_u(t,x)=(t+u(x),x)$, we have
$F=F_0\circ\Phi_u$. Since the transverse data are
independent of $t$, \eqref{eq:star-identity}, applied to
the rotated coframe $R\eta$, gives
\begin{align*}
 \Phi_u^*\rho_i^R
 &=\exp\!\left(
      -u*_h\theta+\ii\,du\wedge(R\eta)^i
    \right)\wedge\rho_i^R\\
 &=\exp\!\left[-*_h(du+u\theta)\right]\wedge\rho_i^R.
\end{align*}
Moreover, using \eqref{eq:Ltheta} and
$du\wedge *_h\theta=\langle du,\theta\rangle_h\vol_h$,
we obtain
\begin{align*}
 d\bigl(*_h(du+u\theta)\bigr)
 &=d(*_hdu)+du\wedge *_h\theta+u\,d(*_h\theta)\\
 &=\bigl(\Delta_hu+\langle du,\theta\rangle_h+qu\bigr)\vol_h\\
 &=(\mathscr P_\theta u)\vol_h.
\end{align*}
Together with \eqref{eq:local-scaled-spinors} and \eqref{eq:Ltheta},
these identities give \eqref{eq:local-equivalence-compensator}.
Uniqueness in
\cref{prop:linear-B-normal} makes $\mathscr P_\theta u=0$ necessary and
sufficient for a closed compensator. This condition is covariant:
\[
 \widehat *\{d(au)+au\widehat\theta\}=*(du+u\theta).
\]
\end{proof}

The transverse data include a chosen hyperCR structure.
Passing to the underlying Einstein--Weyl geometry forgets this
choice and may identify distinct GT data; the oriented round
three-sphere admits two hyperCR structures
\cite[Section~3]{CalderbankPedersen2000}.

\begin{lemma}\label{prop:GTcompat}
Every GT coframe satisfies
\begin{equation}\label{eq:GTcompat-intro}
 *d\theta=d\kappa+\kappa\theta.
\end{equation}
The torsion-free Weyl connection $Dh=2\theta\otimes h$ has
\begin{equation}\label{eq:EWricci}
 \Ric^D_{ij}=\frac{\kappa^2}{2}h_{ij}-\frac32(d\theta)_{ij}.
\end{equation}
\end{lemma}

\begin{proof}
Equation \eqref{eq:GT} gives
\[
 d(*\eta^i)=2\theta\wedge*\eta^i,\qquad
 0=d^2\eta^i=d\theta\wedge\eta^i
                    -(d\kappa+\kappa\theta)\wedge*\eta^i,
\]
which is \eqref{eq:GTcompat-intro}. The connection
$D_XY=\nabla_X^hY-\theta(X)Y-\theta(Y)X+h(X,Y)\theta^\sharp$
has connection and curvature forms
\begin{align*}
 \Omega^i{}_j&=-\delta_{ij}\theta
              -\frac\kappa2\varepsilon_{ijk}\eta^k,\\
 \mathcal R^i{}_j&=-\delta_{ij}d\theta
 -\frac12\varepsilon_{ijk}(d\kappa+\kappa\theta)\wedge\eta^k
 +\frac{\kappa^2}{4}\eta^i\wedge\eta^j.
\end{align*}
Contraction and \eqref{eq:GTcompat-intro} yield
\begin{align*}
 \Ric^D_{ij}
 &=\frac{\kappa^2}{2}\delta_{ij}-(d\theta)_{ij}
   -\frac12\varepsilon_{ijk}(d\kappa+\kappa\theta)_k\\
 &=\frac{\kappa^2}{2}\delta_{ij}-\frac32(d\theta)_{ij},
\end{align*}
which proves \eqref{eq:EWricci}.
\end{proof}

\begin{proof}[Proof of Theorem~\ref{thm:main-local}]
\Cref{thm:canonical-line} gives $L$ and $TM\cong L^{\oplus4}$;
\cref{prop:GT-reduction,prop:untwisting-existence} give the normal
form and its converse. The equivalence criterion and the induced
transformations of $\theta,\kappa$ follow from \cref{thm:classification}.
For sufficiency, take $v=0$ there: the local diffeomorphism
$F_0(t,x)=(a(x)t,\psi(x))$ satisfies
$F_0^*\rho_i'=a^{-1}\rho_i^R$ by \eqref{eq:local-scaled-spinors}.
The nonzero scalar preserves the pure-spinor lines; composing with
the closed $B$-fields used for normalization proves the claim.
\end{proof}

\section{Proof of Theorem B}
\label{sec:compact-foliations}
We first establish automatic Riemannianity and compatible Killing
reconstruction. The argument then establishes the following branches;
``closed'' and ``nonclosed'' refer to the canonical leaves:
\begin{equation}\label{eq:topology-exhaustion}
 \begin{array}{c|c}
 K\not\equiv0&\text{Berger quotients}\\
 K\equiv0,\ \text{non-taut}&\text{affine axis quotients}\\
 K\equiv0,\ \text{taut},\ \text{nonclosed}&\text{constant tori}\\
 K\equiv0,\ \text{taut},\ \text{closed}
       &\text{flat circle bundles; }e=0\text{ gives tori}.
 \end{array}
\end{equation}
The subsequent subsections determine their parameters and finite
coverings. \Cref{prop:compact-exhaustion} completes the exhaustion
argument and proves the topological assertions.

\begin{lemma}
\label{lem:finite-Courant-descent}
For a finite group $G$ acting by ordered Courant automorphisms of an
untwisted triple, a common closed real $B$-field makes the triple
invariant under the underlying diffeomorphisms. If their action on
the manifold is free, the resulting untwisted triple descends to
the quotient.
\end{lemma}

\begin{proof}
Write the action of $G$ on spinors as
\[
 \mathcal A_g=e^{B_g}(f_g^{-1})^*,\qquad
 f_{gh}=f_g\circ f_h,\qquad dB_g=0.
\]
The group law and its average give
\begin{equation}\label{eq:finite-Courant-descent}
\left\{
\begin{aligned}
 &B_{gh}=B_g+(f_g^{-1})^*B_h,\\
 &C=\frac1{|G|}\sum_{h\in G}B_h,\\
 &B_g=C-(f_g^{-1})^*C,\\
 & e^{-C}\mathcal A_ge^C=(f_g^{-1})^*.
\end{aligned}
\right.
\end{equation}
Thus the common closed field $-C$ makes all three pure-spinor
lines invariant under the diffeomorphisms, and a free action gives
an untwisted triple on $M/G$.
\end{proof}

\subsection{Automatic Riemannianity}

Throughout this section manifolds are closed and connected. The labels
of the three structures are fixed; changes of charts therefore preserve
the ordered transverse projective frame, without an $\operatorname{SO}(3)$ rotation.

\begin{proposition}
\label{prop:transverse-Weyl}
The canonical foliation has a holonomy-invariant transverse conformal
class and torsion-free Weyl connection $D$. A holonomy germ fixing a
point $p$ has derivative $c\Id$ with $c>0$.
\end{proposition}

\begin{proof}
For ordered charts, \eqref{eq:weyl-gauge} gives
\begin{equation}\label{eq:riem-transverse-transition}
 \psi^*\eta'=a^{-1}\eta,\qquad
 \psi^*\theta'=\theta-d\log|a|,\qquad
 \psi^*h'=a^{-2}h.
\end{equation}
Indeed,
\[
 D(a^{-2}h)
 =2(\theta-d\log|a|)\otimes(a^{-2}h).
\]
Uniqueness of the torsion-free Weyl connection therefore makes $D$
holonomy invariant. At a fixed point, the first identity in
\eqref{eq:riem-transverse-transition} makes the holonomy derivative
scalar. To determine its sign, transport an adapted generator
continuously around the leafwise loop defining the holonomy germ.
The tangent line of the one-dimensional leaf is orientable, so its
terminal generator is a positive multiple of its initial generator.
The relation between generators and transverse coframes in
\eqref{eq:riem-transverse-transition} then makes the scalar normal
transition positive. Hence $c>0$, including for periodic leaves.
\end{proof}

\begin{proposition}\label{prop:global-scale}
Orient $L$, choose a positive generator $X$, write $\pi_i=X\wedge Y_i$,
and let $\eta^i$ be dual to $[Y_i]$ in $TM/L$. There is a common smooth
function $b$ with $[X,Y_i]\equiv bY_i\pmod L$. The foliation is
Riemannian if and only if a positive function $f$ solves
\begin{equation}\label{eq:global-scale-equation}
 X\log f=b.
\end{equation}
Equivalently, $fX$ preserves all three Poisson tensors; then $f\eta$
is a basic coframe.
\end{proposition}

\begin{proof}
The local Poisson normal form gives $b$, independently of the lifts
$Y_i$. Since $\eta^i(X)=0$, Cartan's formula gives
\begin{equation}\label{eq:riem-scale-identities}
\left\{
\begin{aligned}
 &\Lie_X\eta^i=-b\eta^i,\\
& \Lie_{fX}\pi_i=(fb-Xf)\pi_i,\\
 &\Lie_{fX}(f\eta^i)=f(Xf-fb)\eta^i.
\end{aligned}
\right.
\end{equation}
By \eqref{eq:riem-scale-identities}, a solution of
\eqref{eq:global-scale-equation} makes $f\eta$ basic and
$\sum_i(f\eta^i)^2$ invariant. Conversely, the transverse
volume of an invariant metric is $r\eta^{123}$, $r>0$, and
\[
 0=\Lie_X(r\eta^{123})=(Xr-3br)\eta^{123},
 \qquad X\log(r^{1/3})=b.
\]
Taking $f=r^{1/3}$ proves the converse.
\end{proof}

\paragraph{The transverse rigidity input.}
For a foliation $\mathcal F$ with normal bundle $Q$, an ordered
transverse projective frame is a holonomy-invariant reduction
\[
 E_T\subset\operatorname{Fr}(Q),\qquad
 \operatorname{structure\ group}(E_T)=\RR_{>0}\Id.
\]
Its points are adapted frames, with common positive rescaling;
$\mathcal F_T$ denotes the lift of $\mathcal F$ by linear holonomy.
For codimension at least two this is Rigal's transversely conformally
parallelizable (TCP) structure
\cite[Definition~1, p.~407; p.~409]{Rigal1998}. Indeed, the coordinate
lines and the line spanned by the sum of the frame vectors are
holonomy-invariant transverse line bundles. Their inverse images in
$TM$ are involutive and give the $q+1$ foliations in Rigal's definition.
\begin{proposition}
\label{prop:rigal-input}
Let $\mathcal F$ be a transversely oriented TCP foliation of
codimension $q\ge2$ on a closed connected orientable manifold.
For its adapted-frame bundle $E_T$ and lifted foliation
$\mathcal F_T$, the following results of Rigal hold:
\begin{enumerate}[label=(\roman*),nosep]
\item Either $\mathcal F$ is Riemannian, or the prescribed TCP
structure admits adapted affine coordinates with transitions
$x\mapsto ax+b$, $a>0$.
\item If a leaf of $\mathcal F_T$ is relatively compact in $E_T$,
$\mathcal F$ admits a global adapted transverse parallelism.
\end{enumerate}
\end{proposition}

These are \cite[Theorem, p.~408; Proposition~2, p.~410]{Rigal1998}.
In Rigal's affine construction, the coordinate coframes are adapted
sections of the given TCP reduction
\cite[\S II.3 and Proposition~3, p.~411]{Rigal1998}. Thus they
preserve the prescribed transverse projective frame: in our notation,
$dx^i=s\eta^i$ with one common positive function $s$.
In our application the codimension is $q=3$.
Here and below, \emph{positive scalar-affine} refers to the affine
atlas in~(i), including its compatibility with the prescribed frame.

\begin{proposition}\label{thm:transverse-dichotomy}
If the canonical foliation is not Riemannian, its intrinsic Weyl
connection is flat and $\kappa\equiv0$. On the orientation cover of
$L$, its given ordered transverse projective frame has an affine atlas
with changes $x\mapsto ax+b$, $a>0$.
\end{proposition}

\begin{proof}
Pass to the orientation cover of $L$. The Poisson tensors
$\pi_1,\pi_2,\pi_3,\pi_1+\pi_2+\pi_3$ determine four integrable
rank-two distributions containing $L$, hence the prescribed transverse
projective frame. The ordered normal frame and the orientation of
$L$ orient the cover. The rigidity input therefore applies.
A transverse invariant metric upstairs would descend after averaging
over the finite deck group. Thus the affine case of \cref{prop:rigal-input}(i) yields
$dy^i=s\eta^i$, $s>0$, and \eqref{eq:GT} gives
\[
 0=d(s\eta^i)\wedge s\eta^i
   =s^2d\eta^i\wedge\eta^i
   =-s^2\kappa\,\eta^{123}.
\]
Consequently $\kappa=0$ and, by \eqref{eq:GTcompat-intro},
$d\theta=0$. Locally $\theta=du$, whence
\[
 d(e^{-u}\eta^i)=0,\qquad D(e^{-2u}h)=0.
\]
Thus $D$ is the Euclidean Levi--Civita connection in this closed
coframe. Its flatness descends to $M$.
\end{proof}

\begin{proposition}
\label{prop:hopf-positive-scale}
On $M_a=(\RR^4\setminus0)/\langle a\Id\rangle$, $a>1$, with
coordinates $(t,x)\in\RR\times\RR^3$, let $f>0$ be smooth and
$f(at,ax)=a^2f(t,x)$. The tensors
$\pi_i=f\partial_t\wedge\partial_{x^i}$ admit no untwisted ordered
pure type-one lift.
\end{proposition}

\begin{proof}
Set $u:=f^{-1}$ and $\Omega:=dx^{123}$. The standard spinors
\begin{equation}\label{eq:riem-hopf-standard}
 \rho_i^0
 =f^{-1/2}(dx^j-\ii\,dx^k)
  \exp(\ii u\,dt\wedge dx^i)
\end{equation}
are invariant under the deck transformation
$(t,x)\mapsto(at,ax)$.

Suppose, for a contradiction, that an untwisted lift exists.
By \cref{prop:linear-B-normal}, comparison with the standard
spinors gives a global real two-form $B$ on $M_a$.
Equation \eqref{eq:twist-conjugation} then makes
\eqref{eq:riem-hopf-standard} integrable for the exact twist
$H=dB$. Their degree-one factors satisfy the Frobenius condition,
and \eqref{eq:spinor-criterion} gives
\begin{equation}\label{eq:hopf-positive-H}
 H=-dt\wedge *_3d_xu+h\Omega
\end{equation}
for a smooth real function $h$, where $*_3$ is the Euclidean
Hodge star in the $x$-variables. Since $H$ is closed,
\[
 h_t=-\Delta_xu.
\]
Moreover, deck invariance implies
$h(at,ax)=a^{-3}h(t,x)$.

The positive minimum and finite maximum of $f$ on a compact
fundamental annulus, together with its scaling relation, give
constants $c_1,c_2>0$ such that
\begin{equation}\label{eq:riem-hopf-estimates}
 \frac{c_1}{t^2+|x|^2}
 \le u(t,x)
 \le \frac{c_2}{t^2+|x|^2}.
\end{equation}
For each multi-index $\alpha$, the same scaling argument gives
$\partial_x^\alpha u=\operatorname{O}(|t|^{-2-|\alpha|})$
and $h=\operatorname{O}(|t|^{-3})$ as $|t|\to\infty$,
locally uniformly in $x$. Indeed, rescaling $|t|$ into $[1,a]$
reduces these estimates to smoothness near the punctured
$t$-axis.

For $x\ne0$, define
$U(x):=\int_{\RR}u(t,x)\,dt$.
The two-sided bound above makes $U$ finite and positive.
The derivative estimates through order two justify
differentiation under the integral on
$\RR^3\setminus\{0\}$. Furthermore,
$h_t=-\Delta_xu=\operatorname{O}(|t|^{-4})$ and $h\to0$
at both ends. Thus, using also a change of variables, we obtain
\begin{equation}\label{eq:riem-hopf-harmonic}
\begin{aligned}
 \Delta_xU(x)
 &=-\int_{\RR}h_t(t,x)\,dt=0,\\
 U(ax)&=a^{-1}U(x).
\end{aligned}
\end{equation}

It follows that $v(s,\omega):=e^sU(e^s\omega)$ is a smooth
positive function on
$Q=(\RR/(\log a)\ZZ)\times S^2$, satisfying
\[
 v_{ss}-v_s+\Delta_{S^2}v=0.
\]
Multiplying by $-v$ and integrating by parts over $Q$ gives
\[
 0=-\int_Qv(v_{ss}-v_s+\Delta_{S^2}v)
   =\int_Q\bigl(|v_s|^2+|\nabla_{S^2}v|^2\bigr).
\]
Hence $v=C>0$, and therefore $U(x)=C/|x|$.

Let $S^3\subset M_a$ be the image of the unit sphere in
$\RR_t\times\RR_x^3$, oriented as the boundary of the unit ball
with ambient orientation $dt\wedge dx^{123}$.
Under $M_a\cong S^3\times S^1$, it represents the corresponding
generator of $H_3(M_a;\ZZ)$.

In the punctured covering space, the boundary of
$[-T,T]\times B_R^3$ is homologous to this unit sphere.
For fixed $R$, the decay estimate for $h$ makes the integrals
over its two caps $\operatorname{O}(T^{-3})$.
The side has orientation $-dt\wedge\vol_{S_R^2}$.
Consequently,
\begin{equation}\label{eq:riem-hopf-period}
\begin{aligned}
 \int_{S^3}H
 &=\lim_{T\to\infty}
   \int_{S_R^2}*_3d_x
   \left(\int_{-T}^T u(t,x)\,dt\right)=\int_{S_R^2}*_3dU=-4\pi C\ne0.
\end{aligned}
\end{equation}
This nonzero period contradicts the exactness of $H=dB$.
\end{proof}

Appendix~\ref{app:scalar-dynamics} supplies the global scalar-affine
dynamics, preserving the ordered transverse projective frame.

\begin{corollary}\label{cor:compact-riemannian}
The canonical foliation of every untwisted ordered pure type-one
generalized hypercomplex triple on a closed connected four-manifold
is Riemannian.
\end{corollary}

\begin{proof}
Suppose otherwise and orient $L$, taking its connected orientation
double cover if necessary. The lifted foliation is still
non-Riemannian: averaging an invariant metric over the deck group
would produce one downstairs. By \cref{thm:transverse-dichotomy},
the prescribed transverse frame is positive scalar-affine; by
\cref{cor:scalar-affine-ergodic-periodic}, all ergodic probabilities
are supported on periodic orbits. Identity holonomy on such an orbit would give a compact leaf of
$\mathcal F_T$; \cref{prop:rigal-input}(ii) would then make
$\mathcal F$ Riemannian. Thus \cref{lem:scalar-source-sink,lem:scalar-hopf-gluing}
identify the cover with the Hopf model, preserving the ordered
transverse projective frame. On the punctured universal cover, the Poisson tensors take the form
\[
 \pi_i=f\,\partial_t\wedge\partial_{x^i},
 \qquad i=1,2,3,\]
where $f>0$ satisfies $f(at,ax)=a^2f(t,x)$, with $a>1$ the
fixed dilation factor introduced above.
The common positive factor follows from preservation of the ordered
projective frame, and the last identity from deck invariance.
This contradicts \cref{prop:hopf-positive-scale}.
\end{proof}

\begin{example}\label{ex:twisted-hopf}
For $a>1$, put $M_a=(\RR^4\setminus\{0\})/\langle a\Id\rangle$,
$r^2=t^2+|x|^2$ and $\mathcal R=t\partial_t+\sum_i x^i\partial_{x^i}$.
The forms
\begin{equation}\label{eq:riem-twisted-hopf}
 H=-2r^{-4}\iota_{\mathcal R}(dt\wedge dx^{123}),\qquad
 \rho_i=r^{-1}(dx^j-\ii\,dx^k)
                  \exp(\ii r^{-2}dt\wedge dx^i)
\end{equation}
descend to $M_a$ and satisfy
\[
 dH=0,\qquad (d+H\wedge)\rho_i=-d\log r\wedge\rho_i,
 \qquad\int_{\{r=1\}}H=-4\pi^2.
\]
The integrability identity is unchanged by a constant rotation of
the $x$-coframe in \eqref{eq:riem-twisted-hopf}; hence the whole
$H$-twisted sphere is integrable. Its Poisson tensors are
$\pi_i=r^2\partial_t\wedge\partial_{x^i}$. The two periodic leaves
have nonunit scalar holonomy, which cannot preserve a positive
transverse metric. Thus the canonical foliation is not Riemannian.

\end{example}

\subsection{Compatible Killing generators and transverse geometry}
\label{sec:killing}

A generator $X$ of the canonical line is \emph{compatible Killing} if
$\Lie_X\pi_i=0$ and there is a one-form $\lambda$ with
$\lambda(X)=1$ and $\iota_Xd\lambda=0$. The tensors
$\pi_i=X\wedge E_i$ determine a basic transverse coframe
$\eta=(\eta^1,\eta^2,\eta^3)$ dual to $[E_i]\in TM/L$.
A form $\omega$ is basic if $\iota_X\omega=\Lie_X\omega=0$;
its differential is denoted $d_B$. With
$h=\sum_i(\eta^i)^2$, the metric $g=\lambda^2+h$ makes $X$ unit
Killing. The transverse star and volume use the ordered coframe.

\begin{proposition}
\label{prop:taut-compatible}
For an untwisted ordered pure type-one triple on a closed connected
four-manifold whose canonical foliation is oriented and Riemannian,
geometric tautness is equivalent to the existence of a compatible
Killing generator.
\end{proposition}
\begin{proof}
A unit Killing field has geodesic integral curves, proving one
implication. Conversely, \cref{prop:global-scale} supplies a positive
Poisson-preserving generator $X$ and a basic GT coframe $\eta$.
Choose a metric $g_0$ for which the leaves are geodesic, and let
$Y=aX$, with $a>0$, be their positive unit tangent field.
Set $\lambda_0:=g_0(Y,\cdot)$, so that $\lambda_0(Y)=1$.
Since $Y$ is unit and geodesic,
\[
 d\lambda_0(Y,W)
 =g_0(\nabla^{g_0}_YY,W)
 =0.
\]
Cartan's formula gives $\Lie_Y\lambda_0=0$, while basicness of
$\eta$ gives $\Lie_Y\eta=0$. Thus $Y$ is unit Killing for
$g=\lambda_0^2+h$. Let $T$ be the closure of its flow in
$\operatorname{Isom}(M,g)$.

By \cref{prop:linear-B-normal}, the given spinors have the form
$e^B(\eta^j-\ii\eta^k)e^{\ii a\lambda_0\wedge\eta^i}$.
For horizontal forms, define
$d_\perp\omega:=d\omega-\lambda_0\wedge\Lie_Y\omega$.
The calculation in \cref{prop:GT-reduction} gives
\begin{equation}\label{eq:killing-variable-lift}
 H=dB
 =-\lambda_0\wedge *_h(d_\perp a+a\theta)+\tau\vol_h.
\end{equation}
Indeed, the additional term
$(d\lambda_0\wedge\eta^i)\wedge(\eta^j-\ii\eta^k)$
has horizontal degree four and therefore vanishes.

Define $q$ and $\mathscr D$ by
$d_B(*_h\theta)=q\,\vol_h$ and 
 $(\mathscr D u)\,\vol_h
=d_\perp\!\left(*_h(d_\perp u+u\theta)\right),$ respectively.
Differentiating \eqref{eq:killing-variable-lift} and using $dH=0$
gives
\[
 0=dH
 =(\mathscr D a+Y\tau)\lambda_0\wedge\vol_h.
\]
Set $\mathcal P:=\mathscr D+Y^2$. Since $Y$ is unit Killing,
the Laplacian decomposition and the preceding identity yield
\begin{equation}\label{eq:killing-elliptic-scale}
\begin{aligned}
 \mathcal P
 &=\Delta_g+\langle\theta,d\,\cdot\,\rangle+q,\\
 \mathcal P a
 &=Y(Ya-\tau).
\end{aligned}
\end{equation}

For a function $u$, write
$\bar u:=\int_T k^*u\,d\mu_T(k)$, where $\mu_T$ is normalized
Haar measure. The coefficients of $\mathcal P$ are $T$-invariant,
so $\mathcal P$ commutes with the $T$-action and, in particular,
with $Y$. Averaging \eqref{eq:killing-elliptic-scale} gives
\[
 \mathcal P\bar a=0,\qquad
 \bar a>0,\qquad
 Y\bar a=0.
\]
Define the ground-state transform
$\mathcal P_0u:=\bar a^{-1}\mathcal P(\bar a u)$, and set
$b:=a/\bar a$ and $c:=(Ya-\tau)/\bar a$.
The preceding identities imply
\begin{equation}\label{eq:killing-ground-state}
\left\{\begin{aligned}
 &\mathcal P_0u
 =\Delta_gu+\langle\theta+2d\log\bar a,du\rangle,\\
& [\mathcal P_0,Y]=0,\\
 &\mathcal P_0b=Yc.
\end{aligned}\right.
\end{equation}

On H\"older spaces,
$\mathcal P_0:C^{k+2,\alpha}\to C^{k,\alpha}$ has index zero
by elliptic homotopy to $\Delta_g$.
The maximum principle gives $\ker\mathcal P_0=\RR$ and excludes
$\mathcal P_0u=1$. Thus its range has codimension one and intersects
the constants only at zero. Elliptic regularity therefore gives
\[
 C^\infty(M)=\im\mathcal P_0\oplus\RR.
\]
We may consequently choose $v\in C^\infty(M)$ and $c_0\in\RR$
such that $\mathcal P_0v=c-c_0$. By
\eqref{eq:killing-ground-state},
\[
 \mathcal P_0(Yv-b)
 =Y(\mathcal P_0v)-Yc
 =0.
\]
Hence $Yv-b$ is constant. Since $\overline{Yv}=0$ and $\bar b=1$,
its average is $-1$, and therefore
\[
 Yv=b-1.
\]

Set $Z:=\bar aX=b^{-1}Y$ and $\lambda:=\lambda_0+dv$.
Using $\lambda_0(Y)=1$ and $\iota_Yd\lambda_0=0$, we obtain
\[
\begin{aligned}
 \lambda(Z)
 &=b^{-1}(1+Yv)=1,\\
 \iota_Zd\lambda
 &=b^{-1}\iota_Yd\lambda_0=0.
\end{aligned}
\]
Finally, $Y\bar a=0$ and $Y=aX$ imply $X\bar a=0$. Thus
\[
 \Lie_Z\pi_i
 =\bar a\Lie_X\pi_i-(X\bar a)\pi_i
 =0.
\]
Therefore $Z$ is a compatible Killing generator.
\end{proof}
\begin{proposition}\label{thm:killing-reconstruction}
Let $(\lambda,\eta)$ be a global coframe on a closed connected
four-manifold, with $\lambda(X)=1$, $\iota_Xd\lambda=0$ and $\eta$
basic. Here $X$ denotes the normalized generator determined by this
coframe. Its Poisson tensors $X\wedge E_i$ admit an untwisted ordered
pure type-one lift exactly when the GT equations hold and
\begin{equation}\label{eq:killing-GT-lift}
                 *_h\theta=d_B\beta
\end{equation}
for a basic one-form $\beta$. Every such lift, modulo a common
closed $B$-field, has the form
\begin{equation}\label{eq:killing-spinors}
 \rho_i=e^{\lambda\wedge\beta+\gamma}
       (\eta^j-\ii\eta^k)e^{\ii \lambda\wedge\eta^i},
 \qquad \gamma\in\Omega_B^2(M),
\end{equation}
where $(i,j,k)$ is cyclic. Conversely every such $\beta,\gamma$
gives a lift; no further condition is imposed on $F=d\lambda$.
\end{proposition}

\begin{proof}
The common algebraic two-form from \cref{prop:linear-B-normal}
gives $H=dB$. In \cref{prop:GT-reduction}, replacing $dt$ by
$\lambda$ adds $F\wedge\eta^i$; its wedge with
$\eta^j-\ii\eta^k$ vanishes by horizontal degree. Thus the GT
equations hold and
\begin{equation}\label{eq:killing-H}
 H=-\lambda\wedge *_h\theta+\tau\vol_h.
\end{equation}
The GT data $\eta$, $\theta$ and $\kappa$ are basic. Writing
$d_B(*_h\theta)=q\vol_h$, closedness of \eqref{eq:killing-H} gives
\[
 X\tau=-q,\qquad Xq=0,\qquad
 \tau(\Phi_t m)=\tau(m)-tq(m).
\]
Boundedness forces $q=0$ and $X\tau=0$. Let $T$ be the closure of
the flow of the present generator $X$ in
$\operatorname{Isom}(M,\lambda^2+h)$. Thus $H$ is $T$-invariant. Let
\[
 \bar B:=\int_T g^*B\,d\mu_T(g)
\]
be the normalized Haar average of $B$. Since $dB$ is $T$-invariant,
\[
 d\bar B=dB.
\]
Define
\[
 \beta:=\iota_X\bar B,
 \qquad
 \gamma:=\bar B-\lambda\wedge\beta.
\]
Because $\lambda(X)=1$, $\Lie_X\lambda=0$, and
$\Lie_X\bar B=0$, the forms $\beta$ and $\gamma$ are basic.
Moreover,
\[
 \bar B=\lambda\wedge\beta+\gamma.
\]
The change $\bar B-B$ is closed, and
\begin{equation}\label{eq:killing-basic-differential}
 d\bar B=-\lambda\wedge d_B\beta+F\wedge\beta+d_B\gamma.
\end{equation}
Comparing \eqref{eq:killing-basic-differential} with
\eqref{eq:killing-H} proves \eqref{eq:killing-GT-lift}.
Conversely, \eqref{eq:killing-GT-lift} gives the required vertical
part in \eqref{eq:killing-H}; horizontal three-forms impose no
condition in \eqref{eq:spinor-criterion}. Constant
$\operatorname{SO}(3)$ covariance gives integrability of the whole sphere.
\end{proof}

\begin{proposition}
\label{prop:killing-kappa-constant}
In \cref{thm:killing-reconstruction}, $\kappa$ is constant.
If $\kappa=0$, then $\theta=0$ and $d\eta^i=0$; otherwise
$\beta_0=\theta/\kappa$ is a basic primitive of $*_h\theta$.
\end{proposition}

\begin{proof}
Equations \eqref{eq:GTcompat-intro} and \eqref{eq:killing-GT-lift}
give
\[
\left\{\begin{aligned}
 &d_B\theta=*_h(d_B\kappa+\kappa\theta),\\
 &d_B(*_h\theta)=0,\\
 &d_B(*_hd_B\kappa)+d_B\kappa\wedge *_h\theta=0.
\end{aligned}\right.
\]
For basic $u$, divergence with respect to
$\vol_g=\lambda\wedge\vol_h$ gives
$d_B(*_hd_Bu)=(\Delta_gu)\vol_h$; the curvature term has
horizontal degree four. Thus $\kappa$ satisfies
\[
 \Delta_g\kappa+\langle\theta,d\kappa\rangle=0.
\]
By the maximum principle, $\kappa$ is constant. The GT compatibility
equation therefore reduces to
\begin{equation}\label{eq:killing-constant-kappa}
 d_B\theta=\kappa *_h\theta.
\end{equation}
If $\kappa=0$, \eqref{eq:killing-GT-lift} and Stokes' theorem give
\[
 d(\lambda\wedge\theta\wedge\beta)=|\theta|_h^2\vol_g,
 \qquad \int_M|\theta|_h^2\vol_g=0.
\]
Consequently $\theta=0$ and $d\eta=0$.
If $\kappa\ne0$, \eqref{eq:killing-constant-kappa} gives
$d_B(\theta/\kappa)=*_h\theta$.
\end{proof}

\begin{lemma}\label{lem:killing-parametrization}
Compatible Killing generators for fixed ordered Poisson tensors on
a closed connected manifold differ by a nonzero constant.
\end{lemma}

\begin{proof}
If $X'=fX$ preserves the tensors, then
$\Lie_{fX}\pi_i=-(Xf)\pi_i=0$, hence $Xf=0$.
A characteristic form $\lambda'$ for $X'$ satisfies
\[
 \Lie_X\lambda'=d(1/f),\qquad
 \Phi_t^*\lambda'=\lambda'+t\,d(1/f).
\]
The flow $\Phi_t$ is isometric for the adapted metric of $X$, so
\[
 |t|\,\|d(1/f)\|_\infty
 \leq\|\Phi_t^*\lambda'\|_\infty+\|\lambda'\|_\infty
 =2\|\lambda'\|_\infty.
\]
Letting $|t|\to\infty$ gives $df=0$.
\end{proof}

\paragraph{Regular circles and compact transverse coframes.}

\begin{proposition}\label{prop:reg-generator}
If the oriented canonical leaves are the fibres of a principal
circle bundle $p:M\to N$, there is a unique positive
Poisson-preserving generator of least period $2\pi$. Its induced
transverse coframe descends to $N$.
\end{proposition}

\begin{proof}
Let $U$ generate the initial action and write $\pi_i=U\wedge Y_i$
using an invariant connection. Along each fibre, regard
$m\mapsto(p_*Y_i)(m)$ as a map into the fixed vector space
$T_{p(m)}N$. Poisson compatibility gives
$U(p_*Y_i)=b\,p_*Y_i$ with the same $b$ for all $i$.
For a positive base volume $\nu$, set
\[
 f=|\nu(p_*Y_1,p_*Y_2,p_*Y_3)|^{1/3},\qquad
 U\log f=b,\qquad U(p_*Y_i/f)=0.
\]
Thus $fU$ preserves the tensors. Its fibrewise period $T$ is smooth;
$X=(T/2\pi)fU$ has least period $2\pi$. Any other positive
Poisson-preserving generator is $aX$ with $Xa=0$; its period forces
$a=1$.

Write $X=aU$. The normalization gives
\[
 \int_0^{2\pi}(a^{-1}-1)(\Phi_t^U m)\,dt=0.
\]
Fibre integration gives a unique zero-mean smooth solution of
$Uu=a^{-1}-1$. For $F(m)=\Phi_{u(m)}^U(m)$,
\[
 F_*X=a(1+Uu)U=U,
 \qquad 1+Uu=a^{-1}>0.
\]
The map $F$ has degree one on each fibre and is a bundle
diffeomorphism; the integral Euler class is unchanged.
\end{proof}

\begin{proposition}\label{thm:reg-reduction}
For this normalized action the transverse data are a global GT
coframe on $N$ satisfying $*_h\theta=d\beta$.
Every lift is given by \eqref{eq:killing-spinors}, with
$\beta,\gamma$ pulled back from $N$. Conversely these data work
on every principal circle bundle over $N$, for any connection.
\end{proposition}

\begin{proof}
An invariant connection makes the normalized generator compatible
Killing. Apply \cref{thm:killing-reconstruction}, using
$\Omega_B^*(M)=p^*\Omega^*(N)$.
\end{proof}

\begin{lemma}\label{lem:reg-compact-GT}
A closed connected GT three-manifold with $*_h\theta$ exact has
constant $\kappa$. If $\kappa=0$, its coframe is closed and it is
a three-torus. If $\kappa\ne0$, then $\theta^\sharp$ is Killing and
\begin{equation}\label{eq:reg-positive-Ricci}
 \Ric_h=
 (\kappa^2/2+|\theta|^2)h-\theta\otimes\theta
 \geq(\kappa^2/2)h.
\end{equation}
In particular its universal cover is compact and diffeomorphic to
$S^3$, and $\beta_0=\theta/\kappa$ is a lifting primitive.
\end{lemma}

\begin{proof}
The maximum-principle and Stokes arguments in
\cref{prop:killing-kappa-constant} apply on the base. If $\kappa=0$,
the complete commuting dual fields identify it with $\RR^3/\Lambda$
for a lattice $\Lambda$.

Suppose $\kappa\ne0$. With the nonnegative Hodge Laplacian,
\begin{equation}\label{eq:killing-theta-eigenform}
 d\theta=\kappa*_h\theta,\qquad
 \delta\theta=0,\qquad \Delta_H\theta=\kappa^2\theta.
\end{equation}
Put $V=\theta^\sharp$ and $S=\tfrac12(\nabla\theta+\nabla\theta^T)$.
For the Weyl connection
$D_XY=\nabla_X^hY-\theta(X)Y-\theta(Y)X+h(X,Y)V$,
contraction of the difference tensor gives, in dimension three,
\[
 \Ric^D=\Ric_h+S+\theta\otimes\theta-|\theta|^2h
             +(\diver_hV)h-\tfrac32d\theta.
\]
Using \eqref{eq:EWricci} and
$\delta\theta=-\diver_hV=0$ from \eqref{eq:killing-theta-eigenform}
gives
\begin{equation}\label{eq:killing-ricci-comparison}
 \Ric_h=\tfrac{\kappa^2}{2}h-S
                       -\theta\otimes\theta+|\theta|^2h.
\end{equation}
All integrals below use $\vol_h$. Since $\diver_hV=0$, we have
\[
 \int S(V,V)
 =\tfrac12\int V(|\theta|^2)
 =0.
\]
Using \eqref{eq:killing-theta-eigenform} and
\eqref{eq:killing-ricci-comparison}, the Weitzenb\"ock identity
therefore gives
\[
\begin{aligned}
 \kappa^2\int|\theta|^2
 &=\int|\nabla\theta|^2+\int\Ric_h(V,V)=\int|\nabla\theta|^2
   +\tfrac{\kappa^2}{2}\int|\theta|^2.
\end{aligned}
\]
On the other hand, the pointwise decomposition of $\nabla\theta$
into its symmetric and skew-symmetric parts yields
\[
 |\nabla\theta|^2
 =|S|^2+\tfrac12|d\theta|^2
 =|S|^2+\tfrac{\kappa^2}{2}|\theta|^2.
\]
Hence $\int|S|^2=0$, so $\Lie_Vh=2S=0$ and
\eqref{eq:reg-positive-Ricci} follows from
\eqref{eq:killing-ricci-comparison}. Hamilton's theorem
\cite[Main Theorem~1.1]{Hamilton1982} identifies the base with a
spherical space form; its universal cover is compact and
diffeomorphic to $S^3$.
\end{proof}

\begin{proposition}
\label{thm:closed-leaves-regular}
The canonical leaves of an untwisted ordered pure type-one triple on
a closed connected four-manifold are all closed if and only if $L$
is the vertical bundle of a principal circle bundle. In particular,
closedness of all leaves implies orientability of $L$.
\end{proposition}

\begin{proof}
A closed leaf is compact in its leaf topology
\cite[\S2.3--2.4]{Epstein1976}, hence is a circle. Its holonomy germ
$\varphi$ at $p$ preserves $D$. By \cref{prop:transverse-Weyl},
\[
 \varphi(\exp_p^D v)=\exp_p^D(cv),\qquad c>0.
\]
If $c\ne1$, replace $\varphi$ by its inverse so that $c<1$.
The compact leaf $L_p$ is embedded; choose a small transverse disk
$S$ with $S\cap L_p=\{p\}$. For $v\ne0$ small, the points
$q_n=\exp_p^D(c^nv)$ lie on the same leaf $L_{q_0}\ne L_p$,
whereas $q_n\to p$. This contradicts the closedness of $L_{q_0}$.
Thus $c=1$ and $\varphi=\Id$.
Reeb stability \cite[Theorems~4.1--4.3]{Epstein1976} gives a smooth
circle bundle over a closed Hausdorff base. If $L$ is oriented,
normalization by the smooth fibre-period function gives a free
circle action.

Otherwise pass to the connected orientation cover, with deck
involution $\tau$. For its lifted principal circle bundle,
\cref{prop:reg-generator} gives a unique positive Poisson-preserving
$X$ of period $2\pi$. Uniqueness and the GT equations give
\[
 \tau_*X=-X,\qquad \psi^*\eta=-\eta,
 \qquad \psi^*\theta=\theta,\qquad \psi^*\kappa=-\kappa,
\]
where $\psi$ is the induced base map. By
\cref{thm:reg-reduction,lem:reg-compact-GT}, $\kappa$ is constant;
hence $\kappa=0$, $\theta=0$ and the base is $\RR^3/\Lambda$.
In its affine coordinates,
\[
 \psi(x)=-x+c,\qquad \psi(x)=x\iff2x=c\pmod\Lambda.
\]
Multiplication by two on the torus is onto. Over a fixed point of
$\psi$, the involution $\tau$ reverses the fibre orientation and
therefore has a fixed point, contradicting freeness. Thus $L$ is
orientable. The converse follows from the circle-bundle definition.
\end{proof}

\paragraph{Nonclosed Killing leaves.}

\begin{proposition}\label{thm:killing-topology}
For compatible Killing data admitting a lift, the closure $T$ of
the canonical flow is a free torus action. Either all leaves are
closed or all are nonclosed. In the latter case:
\begin{enumerate}
\item if $\kappa=0$, the manifold is a four-torus;
\item if $\kappa\ne0$, $T$ is a two-torus, its orbit space is $S^2$,
and its Chern vector is nonzero. The manifold is consequently
$L(d,1)\times S^1$ for some $d\geq1$, up to orientation convention.
\end{enumerate}
\end{proposition}

\begin{proof}
The torus preserves $(\lambda,\eta)$. A coframe automorphism fixing
a point commutes with the dual-frame flows, hence fixes a
neighbourhood and therefore the connected manifold. Thus the
effective action is free. Set $r=\dim T$, $Q=M/T$, and let
$x\in\mathfrak t$ generate $X$. Every leaf is dense in its $T$-orbit;
nonclosedness is equivalent to
$r\geq2$ and $\RR x\cap\Lambda_T=0$ for the integral lattice.

Assume henceforth that the canonical leaves are nonclosed.
If $\kappa=0$, \cref{prop:killing-kappa-constant} gives $d\eta=0$.
Cartan's formula yields
\[
 d\bigl(\eta(Y_y)\bigr)
 =\Lie_{Y_y}\eta-\iota_{Y_y}d\eta=0,
 \qquad \ell:\mathfrak t\longrightarrow\RR^3,
 \quad \ell(y)=\eta(Y_y).
\]
For $y\in\ker\ell$, write $Y_y=gX$. Then
\[
 0=\Lie_{Y_y}\lambda
   =d\bigl(\lambda(Y_y)\bigr)+\iota_{Y_y}d\lambda=dg.
\]
Effectiveness gives $y=gx$, so $\ker\ell=\RR x$.
The projection of $\eta$ to $\RR^3/\im\ell$ descends to a closed
coframe on $Q$; hence $Q\cong T^{4-r}$. For a principal connection
$A$ write $\eta=\ell(A)+p^*\zeta$. For every
$[\Sigma]\in H_2(Q;\ZZ)$,
\begin{equation}\label{eq:killing-flat-chern}
 \ell(F_A)+d\zeta=0,
 \qquad \langle c,[\Sigma]\rangle\in
 \ker\ell\cap\Lambda_T=0,
 \qquad c\in H^2(Q;\Lambda_T).
\end{equation}
Since $H^2(Q;\ZZ)$ is torsion-free,
\eqref{eq:killing-flat-chern} gives $c=0$ and $M\cong T^4$.

If $\kappa\ne0$, the GT equations on commuting torus fields give
\[
 0=\theta(Y)\eta(Z)-\theta(Z)\eta(Y)
                        -\kappa\eta(Y)\times\eta(Z).
\]
Since the action is free and $\ker\eta=\RR X$, the restriction of
$\eta$ to each torus orbit has one-dimensional kernel.
Taking the scalar product with $\eta(Y)\times\eta(Z)$ yields
\[
 \kappa|\eta(Y)\times\eta(Z)|^2=0,
 \qquad r-1=\rank(\eta|_{\mathfrak t})\leq1,
 \qquad r=2.
\]
Choose $Y$ independent of $X$ and put
$A=\eta(Y)$, $s=|A|>0$, $v=A/s$, $\nu=v\cdot\eta$,
$b=\eta-v\nu$ and $\theta_v=\langle\theta,v\rangle$.
Cartan's formula gives
\begin{equation}\label{eq:GT-Cartan-evaluation}
 dA=\theta A-\theta(Y)\eta+\kappa A\times\eta.
\end{equation}
Taking the components of \eqref{eq:GT-Cartan-evaluation} along
$v$ and $v^\perp$, and differentiating $\nu=v\cdot\eta$, yields
\begin{equation}\label{eq:killing-Gauss}
\left\{\begin{aligned}
 &d\log s=\theta-\theta_v\nu,\\
 &dv=-\theta_v b+\kappa v\times b,\\
& d\nu=\theta\wedge\nu+\kappa*\nu.
\end{aligned}\right.
\end{equation}
The identities \eqref{eq:GT-Cartan-evaluation}--\eqref{eq:killing-Gauss}
hold for every local field $Y$ with $\Lie_Y\eta=0$ and
$\eta(Y)\ne0$.
The map $v$ descends to $Q$, and $b:TQ\to v^*TS^2$ is an
isomorphism. By \eqref{eq:killing-Gauss},
\[
 dv=(-\theta_v\Id+\kappa v\times)\circ b,
 \qquad \det(-\theta_v\Id+\kappa v\times)|_{v^\perp}=\theta_v^2+\kappa^2>0.
\]
Thus $v:Q\to S^2$ is a local diffeomorphism; compactness makes it
a covering, hence a diffeomorphism. Set $\alpha:=s^{-1}\nu$. By construction, $\alpha(X)=0$ and
$\alpha(Y)=1$. Equation \eqref{eq:killing-Gauss} then gives
\begin{equation}\label{eq:killing-nonzero-chern}
 d\alpha=\frac{\kappa}{s}*_h\nu.
\end{equation}
The curvature in \eqref{eq:killing-nonzero-chern} descends to an
area form of fixed sign on $Q$.
Its nonzero integral evaluates the torus Chern vector by the
functional $\ell(x)=0$, $\ell(y)=1$; thus $c\ne0$.
An integral change of basis sends $c$ to $(d,0)$, giving
$M\cong L(d,1)\times S^1$.
\end{proof}

\begin{corollary}
\label{cor:killing-flat-connection}
In the nonclosed case there is a closed invariant one-form
$\lambda_0$ with $\lambda_0(X)=1$. If $\kappa=0$, the coframe
$(\lambda_0,\eta)$ identifies the Poisson tensors with constant
tensors on a four-torus.
\end{corollary}

\begin{proof}
On the trivial torus bundle choose a flat principal connection $A$
and $\ell\in\mathfrak t^*$ with $\ell(x)=1$; then $\lambda_0=\ell(A)$.
For the bundle over $S^2$, its nonzero integral Chern vector $c$
is not proportional to $x$, since $\RR x\cap\Lambda_T=0$.
Choose $\ell(x)=1$ and $\ell(c)=0$. Then
\[
 [\ell(F_A)]=0,\qquad \ell(F_A)=d\xi,
 \qquad \lambda_0=\ell(A)-p^*\xi,
 \qquad d\lambda_0=0.
\]
In the flat case $d\eta=0$ as well. The complete commuting dual
fields of $(\lambda_0,\eta)$ give the asserted linear coordinates.
\end{proof}

\begin{lemma}
\label{lem:closed-characteristic-unique}
For compatible Killing data with a closed characteristic form,
all lifts with the same ordered Poisson tensors are equivalent by
a common closed $B$-field and a global basic orbit shift.
\end{lemma}

\begin{proof}
Use the closed characteristic form $\lambda$ in
\eqref{eq:killing-spinors}. Set $\beta_0:=0$ if $\kappa=0$, and
$\beta_0:=\theta/\kappa$ otherwise. Since
$d_B\beta=d_B\beta_0=*_h\theta$ and $d\lambda=0$, the two-form
$\lambda\wedge(\beta-\beta_0)$ is closed. Applying its negative
as a common $B$-field transformation, we may assume that
$\beta=\beta_0$.

Write $d_B\gamma=f\vol_h$. Stokes' theorem gives
\[
 \int_M f\lambda\wedge\vol_h
 =-\int_M d(\lambda\wedge\gamma)
 =0.
\]
Define $\Delta_\theta:=\Delta_g+\langle\theta,d\,\cdot\,\rangle$.
Since $d_B(*_h\theta)=0$, its formal adjoint with respect to
$\vol_g=\lambda\wedge\vol_h$ is
\[
 \Delta_\theta^*
 =\Delta_g-\langle\theta,d\,\cdot\,\rangle.
\]
The maximum principle gives
$\ker\Delta_\theta=\ker\Delta_\theta^*=\RR$.
The Fredholm alternative and the zero-mean condition on $f$
therefore yield a smooth solution
\begin{equation}\label{eq:killing-drift-solvability}
 \Delta_\theta v=f.
\end{equation}
Both $\Delta_\theta$ and $f$ are invariant under the torus closure
of the canonical flow. Averaging $v$ over this torus, we may
therefore assume that $v$ is basic.

The corresponding orbit shift $\Phi_v$ preserves the Poisson
tensors and satisfies $\Phi_v^*\lambda=\lambda+d_Bv$.
Define $Q_v:=d_Bv\wedge\beta_0-*_hd_Bv$. By
\eqref{eq:star-identity},
\begin{equation}\label{eq:killing-orbit-correction}
\begin{aligned}
 \Phi_v^*\rho_i
 &=e^{\lambda\wedge\beta_0+\gamma+Q_v}
   (\eta^j-\ii\eta^k)e^{\ii\lambda\wedge\eta^i},\\
 d_BQ_v
 &=-(\Delta_\theta v)\vol_h.
\end{aligned}
\end{equation}
Together with \eqref{eq:killing-drift-solvability}, this gives
\[
 d_B(\gamma+Q_v)
 =\bigl(f-\Delta_\theta v\bigr)\vol_h
 =0.
\]
Thus $\gamma+Q_v$ is a common closed two-form. Applying the
$B$-field transformation $-(\gamma+Q_v)$ reduces the lift to
the representative with $B=\lambda\wedge\beta_0$, proving
the claim.
\end{proof}

\begin{corollary}\label{cor:killing-nonclosed-unique}
For fixed compatible Killing Poisson tensors with nonclosed leaves,
the lift is unique under closed $B$-fields and basic orbit shifts.
If $\kappa=0$, the full triple is equivalent to a constant triple
on a four-torus.
\end{corollary}

\begin{proof}
Apply \cref{cor:killing-flat-connection,lem:closed-characteristic-unique}.
The constant representative in the flat case has $\beta=\gamma=0$.
\end{proof}

\subsection{Nonflat structures and their complete quotients}

\begin{proposition}\label{thm:nonflat-riemannian-taut}
On a closed connected four-manifold, an untwisted ordered pure
type-one triple whose GT invariant is not identically zero has
oriented, geometrically taut canonical foliation.  Its invariant
is nowhere zero, and it admits a compatible Killing generator.
\end{proposition}
\begin{proof}
Since $K\not\equiv0$, \cref{thm:transverse-dichotomy} rules out
the non-Riemannian alternative, in which $\kappa\equiv0$.
Thus $\mathcal F$ is Riemannian.

Assume first that $L$ is oriented. By \cref{prop:global-scale},
choose a Poisson-preserving generator $X$ and its global basic
transverse coframe $\eta$. Choose a one-form $\lambda$ with
$\lambda(X)=1$, and let $Y_i$ be the horizontal lifts of the
dual normal frame. Basicness gives $\eta([X,Y_i])=0$, so each
$Y_i$ is foliated. The fields $X,Y_1,Y_2,Y_3$ span $TM$ and
are complete because $M$ is compact. Hence $\mathcal F$ is
transversely complete in the sense of \cite[\S1.1]{Molino1977}.

Molino's theorem \cite[Theorem~1 and \S1.3]{Molino1977} therefore
applies directly to $(M,\mathcal F)$ and gives a locally trivial
fibration $p:M\to W$ whose fibres are the connected leaf closures,
each carrying a dense Lie flow. The global basic coframe and
complete foliated lifts verify the hypotheses on $M$ itself,
without passing to the transverse orthonormal frame bundle.

The Caron--Carri\`ere classification \cite{CaronCarriere1980}
identifies these fibres with tori carrying dense linear flows;
compare the proof of \cite[Theorem~4.2]{NozawaRoyoPrieto2014}.
Since Molino's local trivializations preserve the induced
foliations, this is a bundle with a fixed linearly foliated
torus as fibre.

Let $\mathscr C$ denote Molino's commuting local system of
$(M,\mathcal F)$, and put $r:=\rank\mathscr C$.
By \cite[Proposition~4.6]{NozawaRoyoPrieto2014}, the structure
group reduces to the affine automorphisms preserving the linear
foliation. The bundle has the form
\begin{equation}\label{eq:nonflat-closure-fibration}
 \mathbb T^{r+1}\longrightarrow M\xrightarrow{p}W,
\end{equation}
where $\dim W=3-r$.

Write $\RR v_0$ for the direction of the dense linear flow in
an affine fibre chart. In this chart, the fibre of $\mathscr C$
is identified with $\RR^{r+1}/\RR v_0$. Its parallel local
sections are represented by constant vertical vector fields
modulo the flow direction. The transition maps are induced by
\[
 A_{\alpha\beta}\in\operatorname{GL}(r+1,\ZZ),
 \qquad A_{\alpha\beta}(\RR v_0)=\RR v_0.
\]
Let $\bar\rho$ be the resulting representation of $\pi_1(W)$
on this quotient. By
\cite[Proposition~4.8]{NozawaRoyoPrieto2014},
\begin{equation}\label{eq:nonflat-commuting-holonomy}
 \operatorname{hol}_{\mathscr C}=\bar\rho\circ p_*.
\end{equation}

If $r=0$, all leaves are closed, and
\cref{thm:closed-leaves-regular} gives a principal circle action,
hence tautness. For the remainder of the tautness argument,
assume that $r>0$.

In an affine chart $\mathbb T^{r+1}_y\times U$, write
$V=\sum_j(v_0)_j\partial_{y^j}$. Since
$\iota_V\eta=\Lie_V\eta=0$ and every $V$-orbit is dense in its
fibre, the coefficients of $\eta$ are independent of $y$.
For each $w\in\RR^{r+1}$, define the constant vertical field
$Y_w:=\sum_jw_j\partial_{y^j}$. These fields commute pairwise. Since the coefficients of $\eta$
are independent of $y$, we have
\begin{equation}\label{eq:nonflat-structural-fields}
 \Lie_{Y_w}\eta=0.
\end{equation}
Moreover, $\ker\eta=\RR V$ implies that evaluation under $\eta$
induces an injective map from $\RR^{r+1}/\RR v_0$ to $\RR^3$
at each point.
Their classes modulo $\RR V$ represent the parallel local
sections of $\mathscr C$.

For two such structural fields $Y,Z$, Cartan's formula and
\eqref{eq:GT} give
$$
 \theta(Y)\eta(Z)-\theta(Z)\eta(Y)
 =\kappa\,\eta(Y)\times\eta(Z).$$
Taking the scalar product with $\eta(Y)\times\eta(Z)$ yields
\begin{equation}\label{eq:nonflat-structural-rank}
 \kappa|\eta(Y)\times\eta(Z)|^2=0.
\end{equation}
The evaluation map $[w]\mapsto\eta(Y_w)$ on $\mathscr C$ is
injective, because its kernel before taking the quotient is
exactly $\RR v_0$. At a point where $\kappa\ne0$,
\eqref{eq:nonflat-structural-rank} therefore forces $r\le1$.
Consequently $r=1$, and $W$ is a closed surface.

Since the holonomy of $\mathscr C$ factors through $\pi_1(W)$,
there is a connected cover $\widehat W\to W$ of degree at most
two such that the pulled-back rank-one local system is oriented.
Pull back the fibration to $\widehat p:\widehat M\to\widehat W$,
and denote the pulled-back local system by $\widehat{\mathscr C}$.
We retain the notation for the pulled-back geometric tensors.

Choose local structural fields $Y$ representing positive
parallel generators of $\widehat{\mathscr C}$. Their classes
differ by positive constants on overlaps, so the same is true
of $A:=\eta(Y)$. Set $s:=|A|>0$ and define
\[
 v:=A/s,\qquad
 \nu:=v\cdot\eta,\qquad
 b:=\eta-v\nu,\qquad
 \theta_v:=\langle\theta,v\rangle.
\]
The quantities $v,\nu,b,\theta_v$ agree on overlaps and are
therefore globally defined on $\widehat M$.

The GT equations determine $\theta$ and $\kappa$ uniquely from
$\eta$, so both are invariant under the local structural fields.
In a foliated torus chart, the functions $v,\theta_v,\kappa$ are
constant along the dense linear flow and hence on the entire
torus fibre. They therefore descend to $\widehat W$.
The form $b$ annihilates both the canonical and structural
directions and is invariant under their local flows. Thus it
also descends to $\widehat W$, where it defines a bundle
isomorphism $b:T\widehat W\longrightarrow v^*TS^2.$

Since each local structural field preserves $\eta$,
\eqref{eq:killing-Gauss} applies. Orient $\widehat W$ by $b$,
and denote the induced area form by $\omega_b$. The determinant identity gives
\[
 v^*\omega_{S^2}
 =(\theta_v^2+\kappa^2)\omega_b.
\]
Integrating over $\widehat W$ and using $\kappa\not\equiv0$,
we obtain
\[
 4\pi\deg v
 =\int_{\widehat W}(\theta_v^2+\kappa^2)\omega_b
 >0.
\]
The oriented bundle isomorphism
$T\widehat W\cong v^*TS^2$ then gives
\begin{equation}\label{eq:nonflat-Gauss-Euler}
 \chi(\widehat W)
 =\langle v^*e(TS^2),[\widehat W]\rangle
 =2\deg v
 >0.
\end{equation}
Since $\widehat W$ is a closed connected oriented surface,
it follows that $\widehat W\cong S^2$.

By \eqref{eq:nonflat-commuting-holonomy}, the local system
$\widehat{\mathscr C}$ is trivial. Thus the monodromy
$\rho:\pi_1(M)\to\RR^*$ of the original rank-one local system
$\mathscr C$ has finite image, and therefore
\[
 \rho(\pi_1(M))\subset\{\pm1\}.
\]
The foliation on $M$ is transversely oriented. Since $M$ is compact,
every bundle-like metric is geodesically complete, so $\mathcal F$
is transversally complete in the sense of
\cite[Definition~2.1]{NozawaRoyoPrieto2014}. This monodromy condition is
condition~(v) of
\cite[Corollary~1.10]{NozawaRoyoPrieto2014}, which implies
tautness.

We have therefore proved tautness in both cases.
By \cref{prop:taut-compatible,prop:killing-kappa-constant},
there is a compatible Killing gauge in which $\kappa$ is
constant. Since $K\not\equiv0$, this constant is nonzero.

Finally, remove the assumption that $L$ is oriented by applying
the preceding argument to its orientation cover. The intrinsic
field $K=\kappa X$ from \eqref{eq:nonflat-intrinsic-generator}
is nowhere zero on that cover, and hence nowhere zero on $M$.
It therefore orients $L$ downstairs. The oriented argument now
applies directly to $M$, completing the proof.
\end{proof}
\paragraph{Explicit geometry and quotient equivalence.}

The compatible Killing generator $K$ replaces $\eta$ by
$\kappa\eta$ and sets $\kappa=1$, fixing the scale of canonical time.
The compact Einstein--Weyl and hyperCR geometries are classical
\cite{Tod1992,GauduchonTod1998,CalderbankPedersen2000}; retaining the
ordered coframe determines the quotient equivalence.

\begin{proposition}\label{thm:compact-Berger-coframe}
Let $N$ be a closed connected three-manifold equipped with a
global GT coframe with $\kappa=1$. Then its universal cover is
diffeomorphic to $\operatorname{SU}(2)$, and, under a suitable
identification, the lifted coframe is $\eta_a$ for some $a\ge0$,
as defined below.

Identify $\mathfrak{su}(2)$ with the imaginary quaternions,
with bracket $[u,w]=2u\times w$, and write
$\sigma:=g^{-1}dg=\sum_{i=1}^3\sigma_i e_i$.
Set $c_a:=1+a^2$, $v(g):=ge_3g^{-1}$, and
$q:=-2c_a^{-1}\sigma_3$. The model coframe and its GT one-form are
\begin{equation}\label{eq:Berger-coframe-explicit}
\begin{aligned}
 \eta_a
 &=vq-c_a^{-1}(a\,dv+v\times dv),\\
 \theta_a&=aq.
\end{aligned}
\end{equation}
The induced metric $h_a:=\sum_{i=1}^3(\eta_a^i)^2$ is
\begin{equation}\label{eq:Berger-metric-explicit}
 h_a
 =\frac{4}{c_a}(\sigma_1^2+\sigma_2^2)
  +\frac{4}{c_a^2}\sigma_3^2.
\end{equation}

The diffeomorphisms preserving $\eta_a$ as an ordered coframe
are exactly the right translations $g\mapsto gc$, with
$c\in A_a$, where
\[
 A_a:=
 \begin{cases}
  \{\exp(u e_3):u\in\RR\}\cong S^1,& a>0,\\
  \operatorname{SU}(2),& a=0.
 \end{cases}
\]
More generally, every diffeomorphism preserving the coframe
up to a constant $\operatorname{SO}(3)$ rotation has the form
$g\mapsto bgc$, with $b\in\operatorname{SU}(2)$ and $c\in A_a$.
\end{proposition}

\begin{proof}
Equation \eqref{eq:GTcompat-intro} gives $d\theta=*\theta$.
Thus \cref{lem:reg-compact-GT} supplies a compact simply connected
cover and makes $V=\theta^\sharp$ Killing. We work on this universal
cover for the remainder of the proof, omitting pullback notation.
Consequently
\[
 \nabla_i\theta_j=\tfrac12\varepsilon_{ijk}\theta_k,
 \qquad d|\theta|^2=0.
\]
Set $a=|\theta|$.  The GT connection forms give
\begin{equation}\label{eq:nonflat-Berger-components}
 d\theta_i=\theta_i\theta-a^2\eta^i+(\theta\times\eta)_i,
 \qquad \Lie_V\eta^i=0.
\end{equation}
For $a=0$, the complete Maurer--Cartan coframe integrates to
$\eta=-2dg\,g^{-1}$, which is \eqref{eq:Berger-coframe-explicit}.
Suppose $a>0$.  By \eqref{eq:nonflat-Berger-components}, the compact
torus closure of the flow of $V$ preserves the coframe.  Its
effective action is free: an automorphism fixing a point commutes
with the dual-frame flows and hence is the identity.  Applying
\eqref{eq:nonflat-structural-rank} with $\kappa=1$ to two commuting
fundamental fields bounds the torus dimension by one.  Thus $V$
generates a free circle action.

Set $v_i:=\theta_i/a$, $q:=\theta/a$, and $\beta:=\eta-vq$.
Equation \eqref{eq:nonflat-Berger-components} gives
\[
 dv=-a\beta+v\times\beta,\qquad dq=*_hq.
\]
The restriction of $-a\Id+v\times$ to $v^\perp$ has determinant
$c_a=1+a^2>0$. Inverting this restriction and using $|q|_h=1$,
we obtain
\begin{equation}\label{eq:nonflat-Berger-reconstruction}
\begin{aligned}
 \eta
 &=vq-c_a^{-1}(a\,dv+v\times dv),\\
 h&=q^2+c_a^{-1}v^*g_{S^2},\\
 dq&=c_a^{-1}v^*\omega_{S^2}.
\end{aligned}
\end{equation}
The induced local diffeomorphism from the compact orbit surface
to $S^2$ is a covering, hence a diffeomorphism.
Orient the fibres by $V$, let $U$ have least period $2\pi$, and put
$p_0=q(U)>0$. For $\alpha=q/p_0$, our convention is
$e=[d\alpha/(2\pi)]$. Simple connectedness of the total space gives
Euler number $\pm1$; positivity of
$\int_{S^2}dq=c_a^{-1}\int_{S^2}\omega_{S^2}$ selects $+1$. Hence
\[
 p_0=\frac{2}{c_a},\qquad
 \alpha=\frac{c_a}{2}q,\qquad
 \int_{\text{fibre}}q=\frac{4\pi}{c_a}.
\]
Identify the bundle with the Hopf bundle so that the positive fibre
action is $g\mapsto g\exp(-u e_3)$, $u\in\RR/2\pi\ZZ$.
Its connection $-\sigma_3$ has curvature
$\tfrac12v^*\omega_{S^2}$, so $q$ and $-2c_a^{-1}\sigma_3$ have
the same curvature and positive fibre period. Their difference
is a closed base one-form, removed by a gauge transformation since
$H^1(S^2;\RR)=0$. Substitution in
\eqref{eq:nonflat-Berger-reconstruction} gives
\eqref{eq:Berger-coframe-explicit}.

Conversely, for the Hopf map in \eqref{eq:Berger-coframe-explicit},
\[
 v^*\omega_{S^2}=4\sigma_1\wedge\sigma_2,\qquad
 d(v\times dv)=2v\,v^*\omega_{S^2}.
\]
Writing $\beta=\eta_a-vq$, one obtains
\[
 *\eta_a=q\wedge(v\times\beta)+v*q,\qquad
 d\eta_a=-q\wedge dv-v*q=aq\wedge\eta_a-*\eta_a.
\]
Moreover,
\[
 \eta_a=2c_a^{-1}\operatorname{Ad}_g
       \bigl(a(e_3\times)-\Id\bigr)\sigma,\qquad
 \det\bigl(2c_a^{-1}\operatorname{Ad}_g
       (a(e_3\times)-\Id)\bigr)=-8c_a^{-2}\ne0.
\]
This proves the coframe condition and \eqref{eq:Berger-metric-explicit}.
For $a>0$, an ordered automorphism preserves $\theta_a$, hence $q$
and its coefficient vector $v$.  It is therefore a Hopf gauge
transformation preserving the connection, hence a constant circle
rotation.  For $a=0$, the right Maurer--Cartan coframe has precisely
the right translations as automorphisms.  Finally,
$L_b^*\eta_a=\operatorname{Ad}_b\eta_a$; absorbing a constant
rotation gives the asserted maps $g\mapsto bgc$.
\end{proof}

On $\operatorname{SU}(2)\times\RR$ consider the untwisted triple
\begin{equation}\label{eq:nonflat-universal-spinors}
 \rho_i=e^{dt\wedge\theta_a}
       (\eta_a^j-\ii\eta_a^k)e^{\ii \,dt\wedge\eta_a^i}.
\end{equation}
Here $K=\partial_t$; since $d\theta_a=*\theta_a$, integrability
follows from \eqref{eq:killing-GT-lift}--\eqref{eq:killing-spinors}
with $\lambda=dt$, $\beta=\theta_a$ and $\gamma=0$.

\begin{proposition}\label{thm:nonflat-universal-decks}
Every untwisted ordered pure type-one triple on a closed connected
four-manifold whose GT invariant is not identically zero is
equivalent to a quotient of the model
\eqref{eq:nonflat-universal-spinors}, for some $a\ge0$, by a
discrete cocompact subgroup $\Lambda\subset A_a\times\RR$.
The action is
$$
 (g,t)\cdot(c,s)=(gc,t+s).$$
Conversely, every such quotient carries the descended untwisted
ordered pure type-one triple.

Two such quotients, with parameters $a,a'$ and deck groups
$\Lambda,\Lambda'$, are equivalent under a diffeomorphism, a common
closed real $B$-field and a constant $\operatorname{SO}(3)$ rotation
if and only if $a=a'$ and
\begin{equation}\label{eq:nonflat-deck-conjugacy}
 \Lambda'=(c,0)^{-1}\Lambda(c,0)
 \qquad\text{for some }c\in A_a.
\end{equation}

Every such subgroup has time projection $T\ZZ$ for some $T>0$
and finite kernel $\Lambda_0\subset A_a$. It can be written as
\begin{equation}\label{eq:nonflat-deck-generators}
 \Lambda=\langle\Lambda_0\times\{0\},(h,T)\rangle,
\end{equation}
where $h\in N_{A_a}(\Lambda_0)$.
The canonical leaves are closed precisely when the class $[h]$
has finite order $m$ in $N_{A_a}(\Lambda_0)/\Lambda_0$.
In that case every leaf has least $K$-period $mT$.
Otherwise every canonical leaf is nonclosed.
\end{proposition}

\begin{proof}
Use the intrinsic generator $K$ from
\eqref{eq:nonflat-intrinsic-generator}, with the corresponding
GT gauge $\kappa=1$.

If the leaves are nonclosed, \cref{cor:killing-flat-connection}
provides a closed characteristic form $\lambda$, and
\cref{thm:killing-topology} gives $b_1(M)=1$.
If the leaves are closed, their circle quotient $N$ has finite
fundamental group by \cref{lem:reg-compact-GT}. Hence
$H^1(N;\RR)=0$, and Poincar\'e duality gives $H^2(N;\RR)=0$.
Choose an invariant characteristic form normalized on $K$.
Its curvature descends to an exact two-form on $N$; subtracting
a base primitive gives a closed characteristic form $\lambda$
with $\lambda(K)=1$.
The real Gysin sequence again gives $b_1(M)=1$.

Thus, in both cases, $d\lambda=0$, $\lambda(K)=1$, and $b_1(M)=1$.
The nowhere-zero form $\lambda$ cannot be exact on compact $M$.
Since the free part of $H_1(M;\ZZ)$ has rank one, its period
group is
\begin{equation}\label{eq:nonflat-characteristic-periods}
 \operatorname{Per}(\lambda)
 :=\left\{\int_\gamma\lambda:
          [\gamma]\in H_1(M;\ZZ)\right\}
 =T\ZZ
 \qquad\text{for some }T>0.
\end{equation}
Consequently, path integration defines a map
\[
 p:M\longrightarrow\RR/T\ZZ,\qquad
 p(m):=\int_{m_0}^{m}\lambda\pmod T.
\]
Since $dp=\lambda$, this is a proper submersion and hence a
fibre bundle. The definition of $T$ makes the induced map
$p_*:\pi_1(M)\to\pi_1(\RR/T\ZZ)\cong\ZZ$ surjective, so the
homotopy sequence shows that its fibres are connected.

A fibre $S$ is a closed connected three-manifold carrying the
restricted GT coframe with $\kappa=1$.
By \cref{thm:compact-Berger-coframe}, its universal-cover
coframe is $(\operatorname{SU}(2),\eta_a)$, with finite ordered
deck group $\Lambda_0\subset A_a$.
The flow of $K$ identifies $M$ with the mapping torus of its
time-$T$ return map on $S$. The inverse return map preserves
the coframe and lifts to a right translation $R_h$, with
$h\in N_{A_a}(\Lambda_0)$.
Thus the simply connected cover is
$\operatorname{SU}(2)\times\RR$, with deck group
\eqref{eq:nonflat-deck-generators}.

By \cref{lem:closed-characteristic-unique}, the original lift is
equivalent to the representative with common two-form
$\lambda\wedge\theta$. Its pullback to the universal cover is
\eqref{eq:nonflat-universal-spinors}, proving the quotient
description.

Conversely, the product action of $A_a\times\RR$ preserves
\eqref{eq:nonflat-universal-spinors}. It is free, since a
nonidentity element has either a nonzero time shift or a
nontrivial right translation. It is also proper, by compactness
of $A_a$ and properness of translations on $\RR$.
Every discrete cocompact subgroup therefore yields a closed
example.

We next determine equivalence between the quotients.
An equivalence lifts to the simply connected covers and
preserves the intrinsic generator $K$. Its underlying map
therefore has the form
\[
 \Phi(g,t)=(\psi(g),t+v(g)).
\]
Comparison of the Poisson tensors gives
$\psi^*\eta_{a'}=R\eta_a$ for a constant
$R\in\operatorname{SO}(3)$. The norm of the associated GT
one-form then gives $a'=a$, and
\cref{thm:compact-Berger-coframe} yields
$\psi(g)=bgc$, with $b\in\operatorname{SU}(2)$ and $c\in A_a$.

After absorbing the constant rotation, the common two-form
arising from the time shift is
$C_v:=dv\wedge\theta_a-*_{h_a}dv$. Its exterior derivative is
\begin{equation}\label{eq:nonflat-shift-compensator}
 dC_v
 =-\bigl(\Delta_{h_a}v+\langle\theta_a,dv\rangle\bigr)
   \vol_{h_a}.
\end{equation}
Uniqueness of the common compensator
(\cref{prop:linear-B-normal}) requires $dC_v=0$.
The maximum principle on compact $\operatorname{SU}(2)$
therefore makes $v$ constant, so
\[
 \Phi(g,t)=(bgc,t+s).
\]
Conjugation by this map sends a right deck transformation
$(d,\tau)$ to $(c^{-1}dc,\tau)$, giving
\eqref{eq:nonflat-deck-conjugacy}.
Conversely, that identity makes $(g,t)\mapsto(gc,t)$ descend
to an equivalence.

Finally, let $\Lambda\subset A_a\times\RR$ be any discrete
cocompact subgroup. Since $A_a$ is compact, each intersection
$\Lambda\cap(A_a\times[-C,C])$ is finite.
Its time projection is therefore discrete and, by cocompactness,
nonzero. Write it as $T\ZZ$, with $T>0$.
The kernel $\Lambda_0$ is finite, and a lift $(h,T)$ of the
positive generator normalizes $\Lambda_0$.
This proves \eqref{eq:nonflat-deck-generators}.

A return time of the $K$-flow must be a multiple $nT$.
Such a return occurs precisely when $h^n\in\Lambda_0$,
independently of the point $[g,t]$.
Thus, if $[h]$ has finite order $m$, every leaf has least
$K$-period $mT$; otherwise every leaf is nonclosed.
\end{proof}

\begin{corollary}
\label{cor:nonflat-explicit-parameters}
For $a>0$, the equivalence classes are parametrized by
\[
 (a,n,T,[\alpha]),\qquad
 n\ge1,\quad T>0,\quad
 [\alpha]\in\RR/(2\pi/n)\ZZ,
\]
through
\[
 \Lambda_0=\langle\exp(2\pi e_3/n)\rangle,\qquad
 \Lambda=\langle\Lambda_0\times\{0\},
                 (\exp(\alpha e_3),T)\rangle.
\]
The canonical leaves are closed if and only if
$n\alpha/(2\pi)\in\QQ$; if this rational number has
denominator $m$ in lowest terms, their least $K$-period is $mT$.
For every $[\alpha]$, the underlying manifold is
$L(n,1)\times S^1$.
For $a=0$, the parameters are a finite subgroup
$\Lambda_0\subset\operatorname{SU}(2)$, a number $T>0$, and a coset
$[h]\in N_{\operatorname{SU}(2)}(\Lambda_0)/\Lambda_0$, modulo
simultaneous $\operatorname{SU}(2)$ conjugacy, as in
\cref{thm:nonflat-universal-decks}. Nonclosed leaves force
$\Lambda_0$ to be cyclic.
In that case the underlying manifold is diffeomorphic to
$L(n,1)\times S^1$, with $n=|\Lambda_0|$ and $L(1,1)=S^3$.
\end{corollary}

\begin{proof}
For $a>0$, $A_a=S^1$ is abelian, so
\eqref{eq:nonflat-deck-conjugacy} is equality of subgroups.
Its finite subgroup of order $n$ is unique, and changing the
positive-time generator multiplies $\exp(\alpha e_3)$ by
$\Lambda_0$. This gives precisely the displayed parameters;
the period assertion follows from
\eqref{eq:nonflat-deck-generators}.
Right translations along a path from $1$ to $\exp(\alpha e_3)$
in $A_a$ induce an isotopy on $\operatorname{SU}(2)/\Lambda_0$,
so its mapping torus is $L(n,1)\times S^1$.

Now let $a=0$ and $[h]$ have infinite order in
$N_{\operatorname{SU}(2)}(\Lambda_0)/\Lambda_0$.
The normalizer then has positive-dimensional identity component,
which centralizes $\Lambda_0$ because
$\operatorname{Aut}(\Lambda_0)$ is finite. A circle in this
centralizer is contained in a maximal torus $T_0$, and
\[
 \Lambda_0\subset
 Z_{\operatorname{SU}(2)}(T_0)=T_0.
\]
Hence $\Lambda_0$ is cyclic. If $n>2$, its normalizer is $N(T_0)$,
and every element of $N(T_0)\setminus T_0$ squares to $-1$.
Thus an infinite-order class has $h\in T_0$.
For $n=1,2$, choose a maximal torus containing $h$.
In either case a path from $1$ to $h$ in $T_0$ induces an
isotopy of $\operatorname{SU}(2)/\Lambda_0$. Its mapping torus
is therefore $L(n,1)\times S^1$.
\end{proof}

\begin{corollary}
\label{cor:nonflat-finite-cover-criterion}
A structure-preserving finite covering between the nonflat models,
up to the stated equivalences, exists exactly when
$a=a'$ and $(c,0)^{-1}\Lambda(c,0)\subset\Lambda'$ has finite index
for some $c\in A_a$.  Its degree is that index; it is regular exactly
when this inclusion is normal.  Thus all finite ordered quotients
are precisely the normal finite-index inclusions, up to this conjugacy.
\end{corollary}

\begin{proof}
The lifted covering preserves the universal structures.  The
coframe calculation in \cref{thm:compact-Berger-coframe} and
\eqref{eq:nonflat-shift-compensator} give $a=a'$ and a lift $(g,t)\mapsto(bgc,t+s)$.  Equivariance is
\[
 (c,0)^{-1}\Lambda(c,0)\subset\Lambda'.
\]
The subgroup correspondence for connected coverings gives the
degree and regularity assertions; conversely any such finite-index
inclusion defines the required structure-preserving covering.
The averaging lemma \cref{lem:finite-Courant-descent} includes
finite ordered Courant actions with closed $B$-components.
\end{proof}

\subsection{Flat developing alternatives}\label{sec:beyond-taut}

Let $M$ be closed and connected, with $K\equiv0$, equivalently
$\kappa\equiv0$.
By \cref{cor:compact-riemannian,prop:global-scale}, the canonical
foliation is Riemannian and its orientation cover has a global
basic GT coframe satisfying
\begin{equation}\label{eq:flat-GT-compatibility}
\begin{gathered}
 d\eta^i=\theta\wedge\eta^i,\qquad d\theta=0,\\
 \theta=dv\ \Longrightarrow\ d(e^{-v}\eta^i)=0.
\end{gathered}
\end{equation}
This gives a transverse developing map with positive scalar linear
holonomy; we shall obtain its Euclidean and axis alternatives.

\paragraph{The exact-twist obstruction.}

The following identity requires neither a Killing generator nor
$\kappa=0$. Suppose $X$ preserves the ordered Poisson tensors
$\pi_i=X\wedge E_i$, with basic GT coframe $\eta$. For
$\lambda(X)=1$ and horizontal $\omega$, put
\begin{equation}\label{eq:flat-horizontal-decomposition}
 u=\iota_Xd\lambda,\qquad
 F=d\lambda-\lambda\wedge u,\qquad
 d_\perp\omega=d\omega-\lambda\wedge\Lie_X\omega.
\end{equation}
\begin{proposition}
\label{prop:mean-curvature-lift}
A three-form compatible with the simultaneous standard spinors has
\begin{equation}\label{eq:mean-curvature-H}
 H=\lambda\wedge *_h(u-\theta)+\tau\vol_h.
\end{equation}
An untwisted lift exists precisely when a compatible $H$ is exact,
$H=dB$, for a global real $B$. Writing $B=\lambda\wedge\beta+\gamma$ with horizontal
$\beta,\gamma$, this is equivalent to
\begin{equation}\label{eq:mean-curvature-lifting-system}
 \begin{aligned}
 u\wedge\beta-d_\perp\beta+\Lie_X\gamma&=*_h(u-\theta),\\
 F\wedge\beta+d_\perp\gamma&=\tau\vol_h.
 \end{aligned}
\end{equation}
No invariance of $\lambda,\beta,\gamma$ is required.
\end{proposition}

\begin{proof}
By \eqref{eq:flat-horizontal-decomposition},
\begin{equation}\label{eq:flat-lifting-differentials}
 \begin{aligned}
 d(\lambda\wedge\eta^i)
 &=\lambda\wedge(u-\theta)\wedge\eta^i
   +\kappa\lambda\wedge *_h\eta^i+F\wedge\eta^i,\\
 dB&=\lambda\wedge
 (u\wedge\beta-d_\perp\beta+\Lie_X\gamma)
       +F\wedge\beta+d_\perp\gamma.
 \end{aligned}
\end{equation}
The spinor equations of \cref{prop:GT-reduction} applied to the
first line give \eqref{eq:mean-curvature-H}; comparison with the
second gives \eqref{eq:mean-curvature-lifting-system}.
By \cref{prop:linear-B-normal}, an untwisted lift exists exactly
when a compatible $H$ is globally exact.
\end{proof}
For $g=\lambda^2+h$, the form $u=(\nabla_XX)^\flat$ is the
mean curvature of the unit field $X$.

\begin{proposition}
\label{prop:radial-developing-obstruction}
Suppose a Poisson-preserving generator $X$ has a global basic
coframe $\eta^i=dn_i+n_i\sigma$, where $n:M\to S^2$ and
$d\sigma=0$. For every $\lambda$ with $\lambda(X)=1$ and every
compatible $H$,
\begin{equation}\label{eq:radial-period-obstruction}
 \int_M H\wedge\sigma
       =\int_M\lambda\wedge\vol_h>0.
\end{equation}
Consequently these Poisson tensors have no untwisted lift.
\end{proposition}

\begin{proof}
From $\sum_i n_i^2=1$ we obtain
\begin{equation}\label{eq:flat-radial-identities}
\begin{aligned}
 \sigma&:=\sum_i n_i\eta^i,\qquad |\sigma|_h=1,\qquad\theta=-\sigma,\\
 b&:=*_h\sigma=n^*\vol_{S^2},\qquad db=0.
\end{aligned}
\end{equation}
Since $\vol_h=\sigma\wedge b$ and horizontal four-forms vanish,
\eqref{eq:mean-curvature-H} gives
\[
 H\wedge\sigma
 =\lambda\wedge(u+\sigma)\wedge b
 =d(\lambda\wedge b)+\lambda\wedge\vol_h.
\]
Integrating and applying Stokes' theorem proves
\eqref{eq:radial-period-obstruction}. If $H=dB$, however, then
\[
 \int_MH\wedge\sigma
 =\int_Md(B\wedge\sigma)=0,
\]
a contradiction.
For developing image $\RR^3\setminus\{0\}$, scalar holonomy fixes
the omitted point. On the transverse universal cover write
\[
 n=\frac{\operatorname{dev}}{|\operatorname{dev}|},\qquad
 \sigma=d\log|\operatorname{dev}|,\qquad
 d\operatorname{dev}=e^{-v}\eta,\qquad f=e^v|\operatorname{dev}|.
\]
The function $f$ descends to a positive global basic function, and
\[
 \widehat X=f^{-1}X,\qquad
 \widehat\eta=f^{-1}\eta
 =\frac{d\operatorname{dev}}{|\operatorname{dev}|}=dn+n\sigma.
\]
Since $Xf=0$, the scale identity \eqref{eq:riem-scale-identities}
gives $\Lie_{\widehat X}\pi_i=0$. Thus
\eqref{eq:mean-curvature-H} applies in this normalized gauge.
An untwisted triple would supply an exact compatible three-form,
contradicting \eqref{eq:radial-period-obstruction}.
\end{proof}

\begin{proposition}
\label{thm:flat-riemannian-topology}
On the orientation cover, the flat transverse developing image is
either $\RR^3$ or the complement of an affine line.
In the first case the flow is taut; if its leaves are nonclosed,
the triple has a constant torus normal form. In the second case
the leaf closures are the fibres of a smooth $T^2$-bundle over a
translation torus. The logarithmic radius $t$ and lifted angle
$\varphi$ have differentials descending to a closed coframe of
this base.
\end{proposition}

\begin{proof}
On the connected orientation cover of $L$, the flow is oriented.
The global basic coframe $\eta$ orients the normal bundle.  By
\cref{cor:compact-riemannian} the lifted flow is Riemannian, and
\eqref{eq:flat-GT-compatibility} gives a transverse similarity
atlas whose linear holonomy consists of positive scalar matrices.
The cover is closed and the codimension is three. Thus we have a
closed, oriented and transversely oriented Riemannian flow with a
transverse similarity structure, as required by
Asuke's developing theorem. The holonomy-invariant
transverse metric also provides a smooth positive holonomy-invariant
transverse density. That theorem \cite{Asuke1996}, recalled in
\cite[pp.~203--204]{Asuke1997}, gives a simply connected Hausdorff
leaf space $\widehat N$ and a factorization
$\operatorname{dev}=\operatorname{dev}_0\circ p$.
Here $p:\widetilde M\to\widehat N$ is a locally trivial fibration
whose fibres are the connected leaves of the lifted foliation,
and $\operatorname{dev}_0:\widehat N\to U$ is a covering.
The image $U$ is $\RR^3$ or a component of
$\RR^3\setminus\RR^{q_0}$, $0\le q_0<3$, with
$q_1=3-q_0$ in Asuke's notation. The case $q_0=0$ is excluded by
\cref{prop:mean-curvature-lift,prop:radial-developing-obstruction}.

Suppose now that $U$ is a proper complement.  Its holonomy preserves
the omitted affine subspace.  If all its scale factors were one,
then the positive function $|z_\perp|$ would be invariant and would
descend from the developing map to a smooth function on the closed
manifold.  Since the developing map is a submersion and
$d|z_\perp|$ is nowhere zero on $U$, this function would have no
critical point, a contradiction.  Thus the scale homomorphism is
nontrivial; this is the standing hypothesis needed in the cited
non-Euclidean parts of \cite{Asuke1997}.

For $q_0>0$, let $H_{\mathrm{hol}}$ be the closure of the holonomy group, using the
lift to the universal cover of $U$ in the case $q_1=2$ as in
\cite[\S3]{Asuke1997}.  Its identity component has no dilation part
\cite[Lemmas~2.1 and~3.1]{Asuke1997}.
Thus $|H_{\mathrm{hol},0}|=\{1\}$, and the construction in
\cite[proof of Proposition~4.1, pp.~220--221]{Asuke1997} gives
\begin{equation}\label{eq:flat-Asuke-monodromy}
 A\in\operatorname{SL}(q_0+d+1,\ZZ),\qquad
 \Phi A=P\Phi,\qquad
 P=\operatorname{diag}(rB_0,B_1),\quad r>1.
\end{equation}
Here $\Phi:\RR^{q_0+d+1}\to\RR^{q_0+d}$ is
surjective, $d$ is the angular orbit dimension, and
$B_0\in \operatorname{SO}(q_0)$, $B_1\in \operatorname{SO}(d)$.
Ordered scalar holonomy gives $B_0=I$. If $q_0=2$,
\cite[Proposition~4.1(2)]{Asuke1997} gives $d=0$, and
\eqref{eq:flat-Asuke-monodromy} becomes
\[
 A\in\operatorname{SL}(3,\ZZ),\qquad
 \Phi A=r\Phi,\qquad
 (T-r)^2\mid\chi_A(T).
\]
Let $m_r\in\QQ[T]$ be the minimal polynomial of $r$.
Characteristic zero makes $m_r$ separable. Every conjugate of
$r$ has the same multiplicity in $\chi_A$, whence
\[
 m_r^2\mid\chi_A,\qquad 2\deg m_r\le3,
 \qquad r\in\QQ.
\]
Since $A,A^{-1}$ are integral, $r$ and $r^{-1}$ are rational
algebraic integers, so $r\in\{\pm1\}$, contrary to $r>1$.

For $q_0=1$, hence $q_1=2$, \cite[Corollary~4.3]{Asuke1997}
gives a smooth $T^2$-bundle over $T^2$ whose fibres carry dense
irrational linear flows. To identify the base coordinates, use
$(x,t,\varphi)$ on the universal cover of the axis complement,
where $t=\log|z_\perp|$. Ordered positive scalar holonomy lifts to
\begin{equation}\label{eq:flat-axis-discrete-angle}
 (x,t,\varphi)\longmapsto
 (a x+b,\,t+\log a,\,\varphi+2\pi n),
 \qquad a>0,\quad b\in\RR,\quad n\in\ZZ.
\end{equation}
The angular translations lie in the discrete group $2\pi\ZZ$;
therefore the angular part of the identity component of the lifted
holonomy closure is trivial.  By \cite[Lemma~3.1 and Theorem~3.2]{Asuke1997},
the developing image of each connected component of a lifted leaf
closure has fixed $t$ and $\varphi$ and varies only in the axial
coordinate.  The invariant one-forms $dt,d\varphi$ are consequently
basic for the leaf-closure fibration.  Their descended forms on its
base are pointwise linearly independent, since the developing map is
a submersion, and closed.  Their dual vector fields commute and are
complete on the compact base; the resulting transitive $\RR^2$-action
identifies that base with $\RR^2/P$ for a lattice $P$.

In the Euclidean alternative $U=\RR^3$, the covering
$\operatorname{dev}_0$ is a diffeomorphism. Thus the developing fibres
are connected and the invariant transverse metric descends to $\RR^3$. For a
holonomy map $g(z)=az+b$, $a>0$, $a\ne1$, its fixed point
$z_0=b/(1-a)$ would satisfy
\[
 h_{z_0}(v,v)
 =h_{z_0}(Dg_{z_0}v,Dg_{z_0}v)
 =a^2h_{z_0}(v,v),
\]
a contradiction. Thus holonomy is translational, and $dx^i$
descend to a global closed transverse coframe $\zeta^i$.
Basic forms identify with smooth forms on $\RR^3$ invariant
under the translation holonomy group $\Delta$: local descent
through the developing submersion glues because its fibres are
connected. Put $Q=\RR^3/\overline\Delta$. It is a compact
abelian Lie group, since it is a continuous image of $M$.
If $\zeta^{123}$ were basic-exact, its primitive would descend to
\[
 b=b_1\,dx^{23}+b_2\,dx^{31}+b_3\,dx^{12},
 \qquad \partial_1b_1+\partial_2b_2+\partial_3b_3=1,
\]
with $b_i\in C^\infty(Q)$. For normalized Haar measure $\mu_Q$,
\[
 1=\int_Q1\,d\mu_Q
   =\sum_i\int_Q\partial_i b_i\,d\mu_Q=0.
\]
Consequently $H_B^3(M,\mathcal F)\ne0$. The Euclidean coframe
orients the normal bundle, and compactness gives transverse
completeness in the sense of
\cite[Definition~2.1]{NozawaRoyoPrieto2014} and
$H_{c,B}^3=H_B^3$. By
\cite[Corollary~1.10, (i)$\Leftrightarrow$(ii)]{NozawaRoyoPrieto2014},
including its zero-or-$\RR$ alternative, $\mathcal F$ is taut.
Now \cref{prop:taut-compatible} gives a compatible Killing
generator; \cref{cor:killing-nonclosed-unique} gives the constant
torus normal form when the leaves are nonclosed.
For nonclosed Euclidean structures, descent to the original manifold
is established in \cref{cor:euclidean-finite-quotients}.
\end{proof}

\subsection{Axis normal forms and full equivalence}\label{sec:axis-proof}

The axis alternative admits a common universal model that retains
angular winding, torus-bundle extension data and nonorientable
canonical lines. Put $\mathcal U=\RR^4_{x,y,t,\varphi}$ and
\begin{equation}\label{eq:axis-G-action}
 D(u,m)=(-1)^m\operatorname{diag}(e^u,e^{-u}),\qquad
 g_{w,u,m}(z,t,\varphi)
 =(D(u,m)z+w,t+u,\varphi+\pi m),
\end{equation}
where $z=(x,y)$, $w\in\RR^2$, $u\in\RR$, $m\in\ZZ$.
They form a group $\mathcal G$ with multiplication
\begin{equation}\label{eq:axis-group-law}
 (w,u,m)(w',u',m')
 =(w+D(u,m)w',u+u',m+m').
\end{equation}
Call $\Gamma<\mathcal G$ \emph{admissible} if it is discrete and
both its translation subgroup $\Lambda=\Gamma\cap\RR^2$
and its base-period group
$P=\{(u,\pi m):g_{w,u,m}\in\Gamma\}$ are full lattices in their
respective planes. In particular some radial period $u$ is nonzero.
Write $M_\Gamma=\mathcal U/\Gamma$.

Use the coframe
\begin{equation}\label{eq:axis-universal-coframe}
 \lambda=e^t dy,\quad \eta^1=e^{-t}dx,\quad
 \eta^2=\cos\varphi\,dt-\sin\varphi\,d\varphi,\quad
 \eta^3=\sin\varphi\,dt+\cos\varphi\,d\varphi.
\end{equation}
For $(a,b)\in\RR^2$, define
\begin{equation}\label{eq:axis-universal-B}
 B_{a,b}=(a\cos\varphi+b\sin\varphi)\eta^1\wedge d\varphi,
 \qquad
 \rho_i^{a,b}=e^{B_{a,b}}(\eta^j-\ii\eta^k)e^{\ii \lambda\wedge\eta^i}.
\end{equation}
Denote the resulting ordered triple on $M_\Gamma$ by
$\mathcal J_{\Gamma,a,b}$.

\begin{proposition}\label{prop:axis-group-models}
Every admissible group gives a compact smooth quotient carrying
\eqref{eq:axis-universal-B}. These triples are untwisted, and their
canonical foliations are Riemannian and non-taut, with leaves dense
in the torus fibres over $\RR^2/P$. The canonical line is
orientable exactly when all $m$ in $\Gamma$ are even; otherwise the
even-$m$ subgroup defines its orientation double cover.
\end{proposition}

\begin{proof}
A nonidentity element either translates the base or is a nonzero
fibre translation. Thus the action is free. Discreteness of $P$
and $\Lambda$ gives properness; compact fundamental sets in the
base and fibre give compactness. By \eqref{eq:axis-group-law},
\[
 g_{w,u,m}g_{v,0,0}g_{w,u,m}^{-1}=g_{D(u,m)v,0,0},
 \qquad D(u,m)\Lambda=\Lambda.
\]
If $u\ne0$ and an eigenline were rational,
its eigenvalue would satisfy
\[
 (-1)^m e^{\pm u}\in\QQ\cap
 \{\text{algebraic integer units}\}=\{\pm1\},
\]
which is impossible. Hence the $y$-lines are dense in the fibres.

Put $\tau_{a,b}=a\cos\varphi+b\sin\varphi$. By
\eqref{eq:axis-G-action}--\eqref{eq:axis-universal-B},
\begin{equation}\label{eq:axis-deck-covariance}
 g^*\lambda=(-1)^m\lambda,\quad
 g^*\eta^i=(-1)^m\eta^i,\quad
 g^*\tau_{a,b}=(-1)^m \tau_{a,b},\quad g^*B_{a,b}=B_{a,b}.
\end{equation}
By \eqref{eq:axis-deck-covariance}, the spinor lines descend. Set $\theta:=-dt$. Differentiating
\eqref{eq:axis-universal-coframe} gives
\[
 d\lambda=\lambda\wedge\theta,\qquad
 d\eta^i=\theta\wedge\eta^i,
\]
and hence $d(\lambda\wedge\eta^i)=0$.
Direct differentiation of the model two-form and spinors then yields
\begin{equation}\label{eq:axis-model-differentials}
\begin{aligned}
 dB_{a,b}
 &=\tau_{a,b}\vol_h,\\
 d\rho_i^{a,b}
 &=\theta\wedge\rho_i^{a,b}.
\end{aligned}
\end{equation}
Equation \eqref{eq:axis-model-differentials} is unchanged by a constant
$\operatorname{SO}(3)$ rotation, so the whole sphere is integrable. Finally
\[
 h=e^{-2t}dx^2+dt^2+d\varphi^2,\qquad
 X_0=e^{-t}\partial_y,\qquad
 g_*X_0=(-1)^m X_0
\]
give the Riemannian and orientation assertions.

On the orientation cover,
$\vol_h=d(\eta^1\wedge d\varphi)$. If a positive generator $Y$
were geodesible, its unit dual form $\lambda'$ would satisfy
$\iota_Yd\lambda'=0$. Horizontality gives
$d\lambda'\wedge\eta^1\wedge d\varphi=0$, and Stokes yields
\[
 0<\int\lambda'\wedge\vol_h
   =\int d\lambda'\wedge\eta^1\wedge d\varphi=0.
\]
The cover, and hence the original foliation, is non-taut. For
$g=\lambda^2+h$ the mean curvature is $u=\theta=-dt$.
\end{proof}

\begin{proposition}
\label{thm:axis-poisson-normalform}
Every compact triple with oriented canonical line and axis
developing image has, up to a diffeomorphism and constant rotation,
the standard ordered Poisson tensors of an admissible group whose
angular periods are even multiples of $\pi$.
\end{proposition}

\begin{proof}
Place the omitted line at the first coordinate axis. Scalar
holonomy is $(x,z_\perp)\mapsto(ax+b,az_\perp)$, $a>0$.
By \cref{thm:flat-riemannian-topology}, $dt,d\varphi$ descend to
the translation-torus base of the closure fibration. Fibre
holonomy consists solely of axis translations; its character
$\ell:\ZZ^2\to\RR$ is nonzero and irrational, since
otherwise the axis coordinate would descend to a function without
critical points on a compact fibre, or its kernel foliation would
have closed leaves. The oriented leaf direction and ordered axis
direction orient each torus fibre. Hence the monodromy around a radial period $s$ satisfies
\begin{equation}\label{eq:axis-radial-monodromy}
 A\in\operatorname{SL}(2,\ZZ),\qquad \ell A=e^s\ell.
\end{equation}
Irrationality means $\ker\ell\cap\QQ^2=\{0\}$, so
\[
 s=0\ \Longrightarrow\ (A-I)\QQ^2
 \subset\ker\ell\cap\QQ^2=\{0\}
 \Longrightarrow A=I.
\]
Some $s\ne0$, since otherwise $t$ would descend to a global
function with nowhere-zero differential on the compact base.
The monodromies commute and have positive eigenvalues. By
\eqref{eq:axis-radial-monodromy}, a fixed eigenbasis puts them in the
form $\operatorname{diag}(e^s,e^{-s})$. Bounded $s$ gives bounded
integral matrices in a lattice basis; thus
\[
 \{s:\operatorname{diag}(e^s,e^{-s})\text{ occurs}\}
      =T\ZZ,\qquad T>0.
\]
The radial character $P\to T\ZZ$ is surjective. Its kernel is
primitive of rank one, so choose a basis of $P$ with radial
periods $(T,0)$. The corresponding monodromies are $(C,I)$, where
$C\in \operatorname{SL}(2,\ZZ)$ is positive hyperbolic and
$\varrho=e^T>1$. The global circle-valued angle gives their periods
\[
 (T,2\pi k),\ (0,2\pi n),\qquad
 T=\log\varrho,\quad k\in\ZZ,\quad n\ge1.
\]

The foliated torus-bundle structure group reduces to affine maps
preserving the fixed irrational line
\cite[Proposition~4.6]{NozawaRoyoPrieto2014}. This reduction preserves
the dense linear fibre foliation determined by $\ker\ell$; an
equivariant shear below will match the developing differential.
Fix its linear monodromy $\rho=(C,I)$ and denote the associated lattice, vector
and torus bundles by $\ZZ^2_\rho$, $E_\rho$ and
$T_\rho=E_\rho/\ZZ^2_\rho$. Differences between affine
gluings with this monodromy are translation-valued one-cocycles.
The exact sequence of sheaves
\[
 0\longrightarrow\ZZ^2_\rho
 \longrightarrow\mathcal C^\infty(E_\rho)
 \longrightarrow\mathcal C^\infty(T_\rho)
 \longrightarrow0
\]
and the fineness of $\mathcal C^\infty(E_\rho)$ identify their
isomorphism classes with $H^2(T^2;\ZZ^2_\rho)$.
The final cellular coboundary for the standard cell structure is
\[
 (\ZZ^2)^2\longrightarrow\ZZ^2,
 \qquad (a,b)\longmapsto(C-I)b.
\]
Consequently
\begin{equation}\label{eq:axis-twisted-Euler-class}
 H^2(T^2;\ZZ^2_\rho)=\ZZ^2/(C-I)\ZZ^2.
\end{equation}
For a representative $\mathbf c\in\ZZ^2$, set
$v=(C-I)^{-1}\mathbf c$. On
$\RR^2_z\times\RR^2_{t,\varphi}$ define
\begin{equation}\label{eq:axis-oriented-decks}
 \begin{aligned}
  \gamma_j(z,t,\varphi)&=(z+j,t,\varphi),&&j\in\ZZ^2,\\
  \tau(z,t,\varphi)&=(Cz,t+T,\varphi+2\pi k),\\
  \sigma(z,t,\varphi)&=(z+v,t,\varphi+2\pi n).
 \end{aligned}
\end{equation}
Their relations are
\begin{equation}\label{eq:axis-oriented-deck-relations}
 \tau\gamma_j\tau^{-1}=\gamma_{Cj},\qquad
 \sigma\gamma_j\sigma^{-1}=\gamma_j,\qquad
 [\tau,\sigma]=\gamma_{\mathbf c}.
\end{equation}
Equation \eqref{eq:axis-oriented-deck-relations} realizes
$[\mathbf c]$ in \eqref{eq:axis-twisted-Euler-class}. The action is
free and properly discontinuous, with compact quotient: the base
translations form a lattice and the fibre translations are
$\ZZ^2$. Replacing $v$ by $v+j$ changes $\mathbf c$ by
$(C-I)j$. Thus these models realize every affine translation
class with the prescribed monodromy, including its torsion classes.

To match the transverse structure, let $\zeta$ be its axis
developing differential on the universal-base pullback.
The form $\zeta$ is closed and basic for the fixed dense linear
flow. Density makes its coefficients independent of $z$,
while closedness gives
\begin{equation}\label{eq:axis-developing-differential}
 \zeta=\ell\,dz+\beta(t,\varphi),
\end{equation}
where $\ell$ is a constant covector and $\beta$ is a closed
base one-form. Equivariance further gives
\[
 \tau_B^*\beta=\varrho\beta,\qquad
 \sigma_B^*\beta=\beta.
\]

On the simply connected universal base, choose $\psi_0$ with
$d\psi_0=\beta$, and define
\[
\begin{aligned}
 b_\tau&:=\psi_0\circ\tau_B-\varrho\psi_0,\\
 b_\sigma&:=\psi_0\circ\sigma_B-\psi_0.
\end{aligned}
\]
The equivariance of $\beta$ makes both defects constant.
Since $\tau_B$ and $\sigma_B$ commute,
\[
 (\varrho-1)b_\sigma=0.
\]
Thus $b_\sigma=0$, because $\varrho>1$.
Replacing $\psi_0$ by
$\psi:=\psi_0+b_\tau/(\varrho-1)$ eliminates the remaining
defect and gives
\begin{equation}\label{eq:axis-shear-equivariance}
\left\{\begin{aligned}
 &d\psi=\beta,\\
& \psi\circ\tau_B=\varrho\psi,\\
& \psi\circ\sigma_B=\psi.
\end{aligned}\right.
\end{equation}

Choose an eigenvector $Z$ satisfying $CZ=\varrho Z$ and
$\ell(Z)=1$, and define the shear
\[
 \mathcal S(z,t,\varphi)
 :=(z+\psi(t,\varphi)Z,t,\varphi).
\]
By \eqref{eq:axis-shear-equivariance}, this map commutes with
every deck transformation in \eqref{eq:axis-oriented-decks},
including the affine fibre translation in $\sigma$.
It therefore descends to a diffeomorphism of the quotient.
Moreover,
\[
 \mathcal S^*(\ell\,dz)
 =\ell\,dz+d\psi
 =\zeta,
\]
and $\mathcal S$ fixes $dt$ and $d\varphi$.
Hence it identifies the model ordered transverse projective
frame with the given one.

Choose a covector $\alpha C=\varrho^{-1}\alpha$ and put
$x=\ell z$, $y=\alpha z$, with sign chosen so $X_0$ is positive.
The radial normalization $|z_\perp|^{-1}d\operatorname{dev}$ is now precisely
\eqref{eq:axis-universal-coframe}. The Poisson tensors have the form
\[
 \pi_i=f^{-1}X_0\wedge E_i,\qquad f>0,
\]
where $(X_0,E_i)$ is dual to $(\lambda,\eta)$. We prove that their
speed is constant up to a global orbit shift.

Set $\lambda_f:=f\lambda$ and $X_f:=f^{-1}X_0$.
A direct computation gives
\[
 \iota_{X_f}d\lambda_f
 =\theta-d_\perp\log f.
\]
Substituting this identity into \eqref{eq:mean-curvature-H} yields
\begin{equation}\label{eq:axis-variable-speed-lift}
 H=-\lambda\wedge *_h d_\perp f+\tau\vol_h.
\end{equation}
Haar averaging over the affine torus fibres commutes with $d$ and
with all transitions in \eqref{eq:axis-oriented-decks}. Write
$F=\overline f$ and
\[
 \overline B=\lambda\wedge\beta_0+\gamma_0,
 \qquad \beta_0=b_1\eta^1+b_2dt+b_3d\varphi.
\]
Its coefficients are base functions. Comparing the vertical parts
of $d\overline B$ and \eqref{eq:axis-variable-speed-lift} gives
\[
 *_h dF=d\beta_0-\theta\wedge\beta_0
 \quad\Longrightarrow\quad
 F_t=(b_1)_\varphi,\qquad F_\varphi=-(b_1)_t.
\]
Thus $F$ is harmonic on the compact translation base:
\[
 F_{tt}+F_{\varphi\varphi}=0.
\]
Integration by parts gives
\[
 0=-\int F(F_{tt}+F_{\varphi\varphi})
   =\int(F_t^2+F_\varphi^2).
\]
Hence $F$ is constant. Since $F>0$, we conclude that
\begin{equation}\label{eq:axis-constant-mean-speed}
 F=c>0.
\end{equation}

The stable eigendirection of $C$ is quadratic irrational in lattice
coordinates. Its constant vector $V$ consequently satisfies
\begin{equation}\label{eq:axis-Diophantine-bound}
 |j\cdot V|\ge C_0(1+|j|)^{-1},\qquad
 0\ne j\in\ZZ^2.
\end{equation}
Indeed, write $V=c_*(1,\xi)$, $c_*\ne0$, with $\xi$ quadratic
irrational. If $a_2x^2+a_1x+a_0$ is its integral minimal polynomial
and $\xi'$ its conjugate, then for $(p,q)\in\ZZ^2\setminus\{0\}$,
\[
 1\leq\bigl|a_2(p+q\xi)(p+q\xi')\bigr|
 =|a_2p^2-a_1pq+a_0q^2|,
 \qquad |p+q\xi'|\leq |p|+|q|\,|\xi'|.
\]
These inequalities give \eqref{eq:axis-Diophantine-bound}.
In local fibre coordinates
$z\in\RR^2/\ZZ^2$, define
\[
 \widehat v_0(j;t,\varphi)=
 \frac{e^t\widehat f(j;t,\varphi)}{2\pi\ii\,j\cdot V}
 \quad(j\ne0),\qquad \widehat v_0(0)=0.
\]
On a relatively compact base chart, the Fourier coefficients of
$f$ and all their base derivatives decay faster than every power
of $|j|$. Equation \eqref{eq:axis-Diophantine-bound} loses only one
power, so the series and every differentiated series converge
locally uniformly; the same estimate applies after differentiating
with respect to base coordinates. The fibrewise kernel of $X_0$
consists of constants, by density of its orbits, so the zero-mean
solution is unique. In this oriented case every deck transformation
preserves $X_0$ and $f$ and maps normalized fibre Haar measure to
normalized fibre Haar measure. The pullback of the solution is
therefore another zero-mean solution of the same equation; uniqueness
makes it deck invariant. Using \eqref{eq:axis-constant-mean-speed},
we obtain
\begin{equation}\label{eq:axis-speed-cohomological-equation}
 v_0\in C^\infty(M),\qquad X_0v_0=f-c,\qquad\overline v_0=0.
\end{equation}
For
$\Phi(p)=\operatorname{Fl}_{X_0}^{v_0(p)/c}(p)$,
\eqref{eq:axis-speed-cohomological-equation} yields
\[
 \begin{aligned}
 D\Phi(X_0)&=\left(1+\frac{X_0v_0}{c}\right)X_0
             =\frac fcX_0,\\
 D\Phi|_{TM/L}&=I,\qquad
 \Phi_*\pi_i=\frac1cX_0\wedge E_i.
 \end{aligned}
\]
Since $f/c>0$, $\Phi$ is a local diffeomorphism homotopic to the identity;
compactness and degree one make it a diffeomorphism. The final
change $y\mapsto cy$ makes the speed one and transforms the
fibre lattice by $\operatorname{diag}(1,c)$. All deck
maps retain the form \eqref{eq:axis-G-action}.
\end{proof}

\paragraph{Lift parameters and full equivalence.}

The lifting calculation also applies to nonorientable canonical lines.
\begin{lemma}\label{lem:axis-twisted-marked}
For fixed standard ordered Poisson tensors on $M_\Gamma$, the
untwisted lifts modulo common closed $B$-fields and global basic
orbit shifts form $\RR^2$, represented by
\eqref{eq:axis-universal-B}.
\end{lemma}

\begin{proof}
On the orientation cover, $u=\theta$ by
\eqref{eq:axis-model-differentials}. Hence
\eqref{eq:mean-curvature-H} and $d^2=0$ give
\[
 dB=\tau\vol_h,\qquad
 0=d^2B=\lambda\wedge\Lie_{X_0}(\tau\vol_h)
 \quad\Longrightarrow\quad X_0\tau=0.
\]
By fibrewise density, $\tau$ is a base function. Deck invariance gives
\begin{equation}\label{eq:axis-twisted-functions}
 \tau(t+u,\varphi+\pi m)=(-1)^m\tau(t,\varphi),
 \qquad (u,\pi m)\in P.
\end{equation}
Every such function occurs. On the even-$m$ base cover, set
\[
 \widehat g(\xi)=\frac{\widehat\tau(\xi)}{1-\ii\xi_t},
 \qquad B=g\eta^1\wedge d\varphi;
 \qquad dB=(1-\partial_t)g\,\vol_h=\tau\vol_h.
\]
The inverse multiplier is smoothness preserving and commutes with
deck transformations, hence preserves the twisted subspace.
Moreover $\tau=\tau'$ if and only if $d(B-B')=0$.

A function $v$ satisfying \eqref{eq:axis-twisted-functions} gives
a global basic orbit shift with lift
\[
 \Phi_v(x,y,t,\varphi)
 =(x,y+e^{-t}v(t,\varphi),t,\varphi).
\]
The identity $\Phi_v^*\lambda=\lambda+dv+v\theta$ changes the
simultaneous algebraic normalizer by $-*_h(dv+v\theta)$. Since
$\theta=-dt$,
\[
 \begin{aligned}
 *_h(dv+v\theta)
 &=-(v_t-v)\eta^1\wedge d\varphi
       +v_\varphi\eta^1\wedge dt,\\
 d\bigl(*_h(dv+v\theta)\bigr)
 &=(v_{tt}-2v_t+v_{\varphi\varphi}+v)\vol_h.
 \end{aligned}
\]
Thus
\begin{equation}\label{eq:axis-marked-operator}
\begin{gathered}
 \mathscr P_{\mathrm{ax}}:=\partial_t^2-2\partial_t+
                           \partial_\varphi^2+1,\\
 \tau\longmapsto\tau-\mathscr P_{\mathrm{ax}}v.
\end{gathered}
\end{equation}
On the dual lattice of the even-$m$ base cover, the Fourier symbol is
\begin{equation}\label{eq:axis-Fourier-symbol}
 p(\xi)=1-|\xi|^2-2\ii\xi_t,\qquad
 p(\xi)=0
 \Longleftrightarrow (\xi_t,\xi_\varphi)=(0,\pm1).
\end{equation}
Both resonant modes satisfy \eqref{eq:axis-twisted-functions}.
Away from them, discreteness of the dual lattice and
$|p(\xi)|\ge |\xi|^2-1$ give
\[
 |p(\xi)|\ge C(1+|\xi|^2).
\]
Thus division by \eqref{eq:axis-Fourier-symbol} preserves smoothness
and commutes with projection to the twisted subspace. The adjoint
for the flat base measure has symbol $\overline{p(\xi)}$, so
\begin{equation}\label{eq:axis-resonant-lifting-space}
 \begin{gathered}
 \ker\mathscr P_{\mathrm{ax}}=\ker\mathscr P_{\mathrm{ax}}^*
   =\Span_{\RR}\{\cos\varphi,\sin\varphi\},\\
 \tau\equiv a\cos\varphi+b\sin\varphi
                   \pmod{\im\mathscr P_{\mathrm{ax}}}.
 \end{gathered}
\end{equation}
The representatives follow from \eqref{eq:axis-model-differentials}.
\end{proof}

Let $\mathcal N$ consist of
\begin{equation}\label{eq:axis-N-action}
 N_{s,\epsilon,w_0,\delta}(x,y,t,\varphi)
 =\bigl(s\epsilon x+x_0,s^{-1}y+y_0,
             t+\log|s|,\epsilon\varphi+\delta\bigr),
\end{equation}
where $s\ne0$, $\epsilon=\pm1$, $w_0=(x_0,y_0)\in\RR^2$ and
$\delta\in\RR$.
Writing $B_{s,\epsilon}=\operatorname{diag}(s\epsilon,s^{-1})$,
its action on the deck group is
\begin{equation}\label{eq:axis-deck-conjugation}
 N g_{w,u,m}N^{-1}
  =g_{B_{s,\epsilon}w+(I-D(u,m))w_0,\,u,\,\epsilon m}.
\end{equation}
In particular $\delta$ is independent of the deck-group conjugacy.

\begin{lemma}
\label{lem:axis-all-poisson-equivalences}
Every diffeomorphism of axis models preserving their Poisson
triples up to constant $\operatorname{SO}(3)$ rotation is an $N\in\mathcal N$
conjugating their deck groups, composed with a global twisted
basic orbit shift.
\end{lemma}

\begin{proof}
Use local developing coordinates
$z=(x,e^t\cos\varphi,e^t\sin\varphi)$, so
$\pi_i=\partial_y\wedge\partial_{z_i}$.
Write the rotation convention as
$F_*\pi_i=\sum_jR_{ji}\pi'_j$.
Preservation of the common line makes the target developing coordinate
independent of $y$, and comparison of the ordered triples gives
\begin{equation}\label{eq:axis-Poisson-Jacobian}
 z'=F_\perp(z),\qquad (\partial_y y')DF_\perp=R,\qquad
 DF_\perp=s(z)R,\qquad R\in\operatorname{SO}(3),
\end{equation}
where $s=(\partial_y y')^{-1}$ is nowhere zero and independent of $y$.
Commuting mixed
partials in \eqref{eq:axis-Poisson-Jacobian} and multiplying by
$R^{-1}$ gives
\[
 (\partial_j s)e_k=(\partial_k s)e_j\quad(j\ne k),
 \qquad ds=0.
\]
Therefore, on the connected universal cover,
\[
 z'=sRz+d_0,\qquad y'=s^{-1}y+v(x,t,\varphi).
\]
The model space $\mathcal U$ is simply connected, so the original
diffeomorphism and its inverse lift to inverse diffeomorphisms between
the universal covers. The developing map on each cover has image
$\RR^3$ with the corresponding axis removed. Hence the affine map
$z\mapsto sRz+d_0$ maps the source axis complement bijectively onto
the target axis complement. Since it is a bijection of $\RR^3$, it
maps the omitted source axis onto the omitted target axis. Thus
$Re_1=\epsilon e_1$, $d_0$ lies on that axis, and its polar lift gives
the $x,t,\varphi$ part of
\eqref{eq:axis-N-action}.

For $(a,b)\in\Lambda$ with image $(a',b')\in\Lambda'$, lattice
equivariance reads
\[
 v(x+a,t,\varphi)-v(x,t,\varphi)=b'-s^{-1}b.
\]
The dense $x$-projection of $\Lambda$ implies
$\partial_xv=h(t,\varphi)$; the same equation with $a\ne0$
makes $h$ constant. Intertwining a radial deck element then gives
\[
 h(e^u-e^{-u})=0,\qquad u\ne0,
 \qquad v=v_0(t,\varphi).
\]

Put $\chi(u,\pi m)=(-1)^m e^{-u}$. The remaining equivariance is
\[
 v_0(b+p)=\chi(p)v_0(b)+k_p,
 \qquad b=(t,\varphi),\quad p\in P.
\]
The diagonal fibre map already identifies the translation lattices,
so the defects factor through $P$, including for a nonsplit Euler
extension. Commutativity of $P$ gives
\[
 (1-\chi(q))k_p=(1-\chi(p))k_q.
\]
Choose a nonzero radial period $p_0$. Since $\chi(p_0)\ne1$, define
$c_0:=k_{p_0}/(1-\chi(p_0))$.
The preceding identity then gives
\begin{equation}\label{eq:axis-affine-defect-trivialization}
 k_p=(1-\chi(p))c_0
 \qquad\text{for all }p\in P.
\end{equation}
Writing $v_0=c_0+e^{-t}v_1$, the remaining equivariance condition
becomes
\[
 v_1(b+p)=(-1)^m v_1(b),
 \qquad p=(u,\pi m)\in P.
\]
By \eqref{eq:axis-affine-defect-trivialization}, the affine part
conjugates the deck groups and the remainder is a global twisted
basic orbit shift.
\end{proof}

\begin{proposition}
\label{thm:axis-full-classification}
Every untwisted ordered pure type-one generalized hypercomplex
structure in the compact axis branch is equivalent, under a
diffeomorphism, a constant $\operatorname{SO}(3)$ rotation and a common closed
$B$-field, to
\[
 (M_\Gamma,\mathcal J_{\Gamma,r,0}),
 \qquad \Gamma<\mathcal G\text{ admissible},\quad r\ge0.
\]
Two such models are equivalent exactly when
\begin{equation}\label{eq:axis-full-equivalence}
 r=r',\qquad N\Gamma N^{-1}=\Gamma'
 \quad\text{for some }N\in\mathcal N.
\end{equation}
This includes nonorientable canonical lines and comparisons
between different finite-cover presentations.
\end{proposition}

\begin{proof}
For oriented canonical line, existence follows from
\cref{thm:axis-poisson-normalform,lem:axis-twisted-marked}.
We prove descent from the orientation cover by controlling finite
ordered Courant actions, including their closed $B$-components.

For an admissible group, put
\[
 \mathcal P_\Lambda
 =\{(u,\pi m):D(u,m)\Lambda=\Lambda\}.
\]
In a lattice basis, bounded $u$ gives only finitely many integral
matrices $D(u,m)$, since $m$ enters only through its parity.
Thus $\mathcal P_\Lambda$ is discrete and contains the full lattice
$P$, so $[\mathcal P_\Lambda:P]<\infty$.
Only finitely many base components of
$Q_\Gamma=N_{\mathcal G}(\Gamma)/\Gamma$ remain.
Two normalizer elements with the same base class modulo $P$ can be
made to have the same base period by multiplication by an element of
$\Gamma$. Their quotient is then a pure translation $g_{p,0,0}$.
Fix
$D_0=D(u_0,m_0)$ from an element of $\Gamma$ with $u_0\ne0$.
Within one component, differences of normalizing translations satisfy
\[
 (I-D_0)p\in\Lambda,\qquad
 p\bmod\Lambda\in (I-D_0)^{-1}\Lambda/\Lambda,
\]
and
\[
 \#\bigl((I-D_0)^{-1}\Lambda/\Lambda\bigr)
       =|\det(I-D_0)|<\infty.
\]
Hence $Q_\Gamma$ is finite.

An ordered Poisson automorphism has $R=I$ in
\cref{lem:axis-all-poisson-equivalences}.
Its affine part belongs to $\mathcal G$: $\epsilon=1$ and
$\delta\equiv0$ or $\pi\pmod{2\pi}$ according as $s>0$ or $s<0$.
These affine maps fix every $\mathcal J_{a,b}$ exactly. By
\eqref{eq:axis-marked-operator} and \eqref{eq:axis-resonant-lifting-space},
the residual shift of an ordered Courant automorphism belongs to
\[
 V=\ker\mathscr P_{\mathrm{ax}}
   =\Span_{\RR}\{\cos\varphi,\sin\varphi\}.
\]
Its closed $B$-compensator is unique by \cref{prop:linear-B-normal}. The shift action is
faithful and meets $Q_\Gamma$ trivially: a lift of a common element
would be a constant fibre translation, whereas
$e^{-t}(A\cos\varphi+B\sin\varphi)$ is constant only when $A=B=0$.
For $g=g_{w,u,m}\in\mathcal G$ and $v\in V$, direct substitution gives
\[
 g\Phi_vg^{-1}=\Phi_{g\cdot v},\qquad
 (g\cdot v)(t,\varphi)
 =(-1)^m v(t-u,\varphi-\pi m)=v(t,\varphi).
\]
Uniqueness of the compensator gives the direct-product decomposition
\begin{equation}\label{eq:axis-ordered-automorphisms}
 \operatorname{Aut}_{\mathrm{Cour,ord}}
       (\mathcal J_{\Gamma,a,b})
 \cong V\times Q_\Gamma
 \cong\RR^2\times Q_\Gamma.
\end{equation}
Every finite subgroup has zero projection to $\RR^2$, hence lies
in $Q_\Gamma$ with zero $B$-compensators.

Apply this to the deck involution after identifying the canonical
orientation cover with a normal form. Its affine lifts generate
an admissible finite-index supergroup of $\Gamma$ in $\mathcal G$.
The normalizing Courant equivalence is equivariant and descends,
including its closed $B$-field. This proves existence also for
nonorientable canonical line.

Finally set $\zeta=\operatorname{sgn}s$. For
$N\in\mathcal N$, $N^*\lambda=\zeta\lambda$ and
$N^*\eta=\zeta R\eta$ for a constant $R\in \operatorname{SO}(3)$, while
\begin{equation}\label{eq:axis-amplitude-action}
 N^*B_{a',b'}=
 \zeta\bigl(a'\cos(\epsilon\varphi+\delta)
             +b'\sin(\epsilon\varphi+\delta)\bigr)
                    \eta^1\wedge d\varphi.
\end{equation}
Equation \eqref{eq:axis-amplitude-action} acts orthogonally on the
parameter plane. By \eqref{eq:axis-marked-operator} and
\eqref{eq:axis-resonant-lifting-space}, residual shifts do not change
the marked class. Thus equivalence implies \eqref{eq:axis-full-equivalence}.
Conversely, \eqref{eq:axis-deck-conjugation} is independent of $\delta$,
whereas \eqref{eq:axis-amplitude-action} varies transitively on each
parameter circle as $\delta$ varies. Equal radii therefore suffice.
\end{proof}

\begin{corollary}
\label{cor:axis-all-finite-quotients}
Every finite subgroup of the ordered Courant automorphism group
of an axis model lies in $Q_\Gamma$ and acts freely. Its
quotients correspond to normal finite-index supergroups
$\Gamma\triangleleft\Gamma'<\mathcal G$. More generally a
structure-preserving finite covering between models $(\Gamma,r)$ and $(\Gamma',r')$ exists
up to the stated equivalences exactly when $r=r'$ and
$N\Gamma N^{-1}\subset\Gamma'$ has finite index for some
$N\in\mathcal N$; that index is the degree, and normality is
exactly regularity.
\end{corollary}

\begin{proof}
A fixed point of an affine class modulo $\Gamma$ would give a
nonidentity element of $\mathcal G$ fixing a point of $\mathcal U$,
which is impossible. The normal-supergroup statement follows.
A structure-preserving finite covering lifts to a diffeomorphism of
universal covers. Applying \cref{lem:axis-all-poisson-equivalences}
to the corresponding finite-index target subgroup and using
\eqref{eq:axis-resonant-lifting-space} and
\eqref{eq:axis-amplitude-action} gives $r=r'$ and the stated
subgroup inclusion. Conversely such an inclusion induces the
covering, with degree equal to its index; it is regular exactly
when the subgroup is normal.
Descent of the untwisted triple is also a special case of
\cref{lem:finite-Courant-descent}.
Moreover, finite ordered subgroups up to ordered Courant conjugacy
are exactly the subgroup conjugacy classes of $Q_\Gamma$, by
\eqref{eq:axis-ordered-automorphisms}.
\end{proof}

Before normalization, an equivalence between speeds $c,c'>0$
has fibre linear part $\operatorname{diag}(s\epsilon,c/(c's))$, so
\[
 c\,\operatorname{covol}\Lambda
       =c'\,\operatorname{covol}\Lambda'.
\]
At speed one this is the lattice area. The group $\Gamma$ also
retains angular winding and the torsion class
\eqref{eq:axis-twisted-Euler-class}. The stated group orbits
classify the structures; no smooth orbit space or shorter numerical
enumeration is asserted.

For the finite-cover topology, use \eqref{eq:axis-oriented-decks}
on the orientation cover. Its monodromy is $(C,I)$, with $C$
positive hyperbolic and $v=(C-I)^{-1}\mathbf c\in\QQ^2$.
Choose $q\ge1$ with $qv\in\ZZ^2$. The angular generator on the
corresponding finite cover may be taken as
\[
 \gamma_{-qv}\sigma^q(z,t,\varphi)
   =(z,t,\varphi+2\pi qn).
\]
A linear change of base coordinates removes the oblique period
of $\tau$. The resulting cover is $N_C\times S^1$, where $N_C$
is the mapping torus of $C\in\operatorname{SL}(2,\ZZ)$ with
$\operatorname{tr}C>2$. On this cover the canonical foliation is
the product with $S^1$ of the classical hyperbolic torus-bundle flow
studied by Carri\`ere \cite[Example~I.D.6 and \S III.B]{Carriere1984};
see also \cite[Example~4.10]{NozawaRoyoPrieto2014}.
Conversely, $v=k=0$ and $n=1$ realize every such covering model in
\eqref{eq:axis-oriented-decks}.
\subsection{Torus and circle-bundle parameters}
\label{sec:flat-global-data}

For the two taut flat branches, parameter orbits classify structures;
a specified covering also retains its embedded fundamental subgroup.

\begin{proposition}\label{thm:euclidean-lattice-moduli}
On $V=\RR_t\oplus\RR_x^3$, equip each quotient by a
full lattice $\Gamma$ with the constant spinors
$(dx^j-\ii\,dx^k)e^{\ii\,dt\wedge dx^i}$.
Two such triples are equivalent if and only if $A\Gamma=\Gamma'$
for some
\begin{equation}\label{eq:euclidean-lattice-group}
 A\in G:=\left\{\begin{pmatrix}a&b\\0&a^{-1}R\end{pmatrix}:
 a\ne0,\ b\in(\RR^3)^*,\ R\in\operatorname{SO}(3)\right\}.
\end{equation}
Every diffeomorphism occurring in an equivalence is affine.
The moduli set is
\begin{equation}\label{eq:euclidean-double-coset}
 G\backslash\operatorname{GL}(4,\RR)/\operatorname{GL}(4,\ZZ).
\end{equation}
Nonclosed leaves correspond to $\Gamma\cap\RR\partial_t=0$.
A structure-preserving finite covering from $V/\Gamma_u$ to
$V/\Gamma_d$ exists precisely when
\[
 A\Gamma_u\subset\Gamma_d\qquad\text{for some }A\in G.
\]
Its degree is $[\Gamma_d:A\Gamma_u]$, and every such covering is
regular.
\end{proposition}

\begin{proof}
Put $X=\partial_t$. By \cref{lem:killing-parametrization}, a lift
$F:V\to V$ of a Poisson equivalence satisfies
\begin{equation}\label{eq:euclidean-lift-periods}
 F_*X=aX,\qquad F^*dx=a^{-1}Rdx,\qquad
 F(z)=Az+p(z),\qquad A\Gamma=\Gamma',
\end{equation}
where $a\ne0$ is constant and $p$ is $\Gamma$-periodic.
Averaging the differential in \eqref{eq:euclidean-lift-periods}
over $V/\Gamma$ gives $A\in G$; the transverse part of $p$ is constant.
Removing this affine part reduces $F$ to
$F_v(z)=z+v(z)X$, with $Xv=0$. The respective algebraic
compensators are
\begin{equation}\label{eq:euclidean-compensators}
 C_A=-a^{-1}*_h(b\,dx),\qquad C_v=-*_h dv.
\end{equation}
The first is closed. For the second,
\[
 dC_v=0
 \iff \sum_{i=1}^3\partial_{x_i}^2v=0
 \overset{Xv=0}{\iff}\Delta_{dt^2+|dx|^2}v=0
 \implies dv=0,
\]
by compactness. Thus every equivalence is affine, and conversely
$C_A$ in \eqref{eq:euclidean-compensators} realizes every affine
map with $A\in G$ and $A\Gamma=\Gamma'$.
Writing $\Gamma=F_0\ZZ^4$ gives
\eqref{eq:euclidean-double-coset}. An $X$-orbit closes exactly
when $\Gamma\cap\RR X\ne0$.

The same argument for a finite covering gives
$A\Gamma_u\subset\Gamma_d$ and degree $[\Gamma_d:A\Gamma_u]$;
regularity follows since $\Gamma_d$ is abelian.
\end{proof}

\begin{corollary}
\label{cor:euclidean-finite-quotients}
Every free finite ordered action on a constant torus triple,
allowing closed $B$-compensators, consists of translations. Its
quotient replaces $\Gamma$ by a finite-index superlattice.
Consequently the canonical line of every nonclosed compact
Euclidean example is orientable, and the original example itself
has constant torus normal form.
\end{corollary}

\begin{proof}
For an ordered affine map of finite order, $R=I$ in
\eqref{eq:euclidean-lattice-group}; hence
\[
 A^q=I\implies a=\pm1,\ b=0\implies A=\pm I.
\]
If $A=-I$, the equation $2z=c$ supplies a fixed point of
$z\mapsto-z+c$ on the torus. Freeness therefore forces $A=I$.
The lifts of the finite translation group $F$ form a lattice
$\widehat\Gamma$ with
\begin{equation}\label{eq:euclidean-superlattice}
 \Gamma\subset\widehat\Gamma,\qquad
 N=[\widehat\Gamma:\Gamma]<\infty,\qquad
 (V/\Gamma)/F=V/\widehat\Gamma.
\end{equation}
Translations fix all three spinors; uniqueness of the common
algebraic compensator then gives $B=0$.

If a nonclosed Euclidean example had nonorientable canonical line,
its orientation cover would be a constant torus by
\cref{thm:flat-riemannian-topology,cor:killing-nonclosed-unique}.
Its free ordered deck involution would reverse $X$, and hence
have $A=-I$, a contradiction. Finally,
\[
 0\ne v\in\widehat\Gamma\cap\RR X
 \implies 0\ne Nv\in\Gamma\cap\RR X.
\]
Together with \eqref{eq:euclidean-superlattice}, this shows that
finite quotients preserve nonclosedness.
The quotient triple descends as in
\cref{lem:finite-Courant-descent}.
\end{proof}

\begin{proposition}\label{thm:closed-flat-arithmetic-data}
The closed-leaf flat branch consists of principal circle bundles
over $\RR^3/\Lambda$, with period-$2\pi$ generator. For
nonzero Euler class its complete parameters are
\[
 (\Lambda,e,\mu),\qquad
 0\ne e\in\operatorname{Hom}(\Lambda\wedge\Lambda,\ZZ),
 \quad\mu\in\RR.
\]
Two tuples are equivalent if and only if some $Q\in\operatorname{O}(3)$ satisfies
\begin{equation}\label{eq:closed-flat-arithmetic-equivalence}
 Q\Lambda=\Lambda',\qquad
 e'(Qu,Qv)=(\det Q)e(u,v),\qquad \mu=(\det Q)\mu'.
\end{equation}
For $e=0$ the lifting parameter is absent and the structure belongs
to the constant-torus branch.
\end{proposition}

\begin{proof}
Write $N=\RR^3/\Lambda$. Circle normalization gives
$\eta=dx$ and $\theta=\kappa=0$.
For an invariant connection $\lambda$, let $F\in\Omega^2(N)$
denote its curvature, so that $d\lambda=p^*F$ and
$[F/(2\pi)]=e$.

Suppressing pullbacks from $N$, \eqref{eq:killing-spinors}
expresses the common two-form as
$B=\lambda\wedge\beta+\gamma$, with $d\beta=0$.
Set $\mathcal T:=F\wedge\beta+d\gamma$, so that
$dB=p^*\mathcal T$. Its integral is
\begin{equation}\label{eq:circle-lift-invariant}
 \mu:=\int_N\mathcal T
 =2\pi\langle e\smile[\beta],[N]\rangle.
\end{equation}

For fixed Poisson data and connection, two such algebraic
two-forms differ by a closed form precisely when their
corresponding three-forms $\mathcal T$ coincide.
Orbit shifts determined by $v\in C^\infty(N)$ and circle gauge
transformations with connection change $\nu$ act, respectively, by
\begin{equation}\label{eq:circle-gauge-action}
\begin{aligned}
 \mathcal T
 &\longmapsto\mathcal T-d(*_hdv)
  =\mathcal T-(\Delta_hv)\vol_h,\\
 \mathcal T
 &\longmapsto\mathcal T-d(*_h\nu).
\end{aligned}
\end{equation}
Here the gauge one-form $\nu$ satisfies
$d\nu=0$ and $[\nu/(2\pi)]\in H^1(N;\ZZ)$.
Both transformations preserve $\mu$.

Conversely, suppose that $\mathcal T_1$ and $\mathcal T_2$
have the same integral. Write
$\mathcal T_2-\mathcal T_1=f\vol_h$.
Then $\int_Nf\vol_h=0$, so elliptic solvability on the flat
torus provides a smooth function $v$ satisfying $\Delta_hv=f$.
Consequently,
\begin{equation}\label{eq:circle-zero-mean}
 \mathcal T_2-d(*_hdv)=\mathcal T_1.
\end{equation}
Thus, for fixed Poisson data and connection, the transformations
in \eqref{eq:circle-gauge-action}, together with common closed
$B$-fields, identify precisely the lifts with the same $\mu$.

We next show that $\mu$ is independent of the connection.
Consider an arbitrary change
$\widetilde\lambda=\lambda+\nu$, with $\nu\in\Omega^1(N)$,
and set $\widetilde B:=B+*_h\nu$.
Writing $\alpha_i:=\eta^j-\ii\eta^k$, the spinor identity gives
\begin{equation}\label{eq:circle-connection-change}
 e^{\widetilde B}\alpha_i
 e^{\ii\widetilde\lambda\wedge\eta^i}
 =e^B\alpha_i e^{\ii\lambda\wedge\eta^i}.
\end{equation}
Hence the same triple is represented using the new connection,
with
\[
 \widetilde{\mathcal T}
 =\mathcal T+d(*_h\nu).
\]
Stokes' theorem therefore gives
$\int_N\widetilde{\mathcal T}=\int_N\mathcal T$,
proving the claimed independence.

For $e\ne0$, Poincar\'e duality makes the linear map
\[
 [\beta]\longmapsto
 \langle e\smile[\beta],[N]\rangle
\]
surjective onto $\RR$.
Choose constant forms $F,\beta$ such that
$[F/(2\pi)]=e$ and $\int_NF\wedge\beta=\mu$,
and choose a connection of curvature $F$.
Taking $\gamma=0$ realizes the prescribed value of $\mu$
in \eqref{eq:circle-lift-invariant}.
If $e=0$, that equation gives $\mu=0$.
A product flat connection and the fixed-data classification
then give a constant-torus representative.

Finally, let $\Phi$ be an equivalence, and let $\psi$ be its
induced base diffeomorphism.
By \cref{lem:killing-parametrization} and the least-period
normalization,
$\Phi_*X=\epsilon X'$ for some $\epsilon\in\{\pm1\}$.
Comparison of the Poisson tensors gives
$\psi^*dx'=\epsilon Rdx$ for a constant
$R\in\operatorname{SO}(3)$.
Thus $Q:=\epsilon R$ lies in $\operatorname{O}(3)$,
$\det Q=\epsilon$, and $\psi$ is affine with
$Q\Lambda=\Lambda'$.

A fibrewise diffeomorphism of degree $\epsilon$ covering $\psi$
exists precisely when $\psi^*e'=\epsilon e$.
Write $\Phi^*\lambda'=\epsilon\lambda+\nu$, where $\nu$ is a
one-form on $N$. Spinor comparison, as in
\eqref{eq:circle-connection-change}, gives
\begin{equation}\label{eq:circle-equivalence-integral}
\begin{aligned}
 \mathcal T_{\mathrm{eff}}
 &=\psi^*\mathcal T'-\epsilon d(*_h\nu),\\
 \int_N\mathcal T_{\mathrm{eff}}
 &=(\deg\psi)\mu'=\epsilon\mu'.
\end{aligned}
\end{equation}
The fixed-data classification therefore makes equality of
lifting classes equivalent to $\mu=\epsilon\mu'$.
Together with the lattice and Euler-class conditions above,
this proves both necessity and sufficiency of
\eqref{eq:closed-flat-arithmetic-equivalence}.
\end{proof}
\begin{example}\label{ex:FG-parameters}
We first fix the lattice normalization used throughout this example.
On $G_{\mathrm{KT}}=\RR^4_{w,x,y,z}$, consider the group law
\[
 (w,x,y,z)(w',x',y',z')
 =(w+w',x+x',y+y',z+z'+xy').
\]
The integer points form a lattice subgroup
$\Gamma_{\mathrm{KT}}=\ZZ^4$.
The left-invariant coframe on
$\Gamma_{\mathrm{KT}}\backslash G_{\mathrm{KT}}$ is
\[
 e^0=dw,\qquad e^1=dx,\qquad
 e^2=dy,\qquad e^3=dz-x\,dy.
\]
Let $e_0,e_1,e_2,e_3$ denote its dual frame.
Then $de^0=de^1=de^2=0$ and $de^3=-e^1\wedge e^2$.
The base periods and the period of the central field
$e_3=\partial_z$ are one.
Orient the base by $e^{012}$, so that
$\int_{T^3}e^{012}=1$.
The equivalence statement below refers to this fixed lattice;
\cite{FinoGrantcharov2025} does not specify such a normalization.

Let $I_{\mathrm{FG}},J_{\mathrm{FG}}$ be the generator matrices
displayed in \cite[pp.~108--109]{FinoGrantcharov2025}, and write
$p:=b_1^{\mathrm{FG}}$ and $r:=b_2^{\mathrm{FG}}$.
The formulas below hold for all $p\in\RR$ and $r\ne0$,
including $p=0$.
To compare with our convention of matrices acting on columns
in the ordered vector--covector basis, we transpose the
displayed generators:
\[
 \mathcal I=I_{\mathrm{FG}}^T,\qquad
 \mathcal J=J_{\mathrm{FG}}^T,\qquad
 \mathcal K=\mathcal I\mathcal J
           =(J_{\mathrm{FG}}I_{\mathrm{FG}})^T.
\]

Take $X:=e_3/(2\pi)$ and $\lambda:=2\pi e^3$.
The transverse coframe is
\[
\begin{aligned}
 \eta^1&=\frac{r}{6\pi}(e^0+e^1)-\frac{e^2}{2\pi},\\
 \eta^2&=-\frac{r}{6\pi}e^0,\\
 \eta^3&=\frac{r}{6\pi}e^1.
\end{aligned}
\]
Its GT data satisfy $\theta=\kappa=0$.
The common two-form is $B=\lambda\wedge\beta+\gamma$, where
\[
 \beta=-\eta^2-\eta^3,\qquad
 \gamma=\frac{12\pi^2p}{r}\eta^2\wedge\eta^3.
\]
Equivalently, in the invariant coframe,
\[
 B=\frac r3 e^3\wedge(e^0-e^1)
   -\frac{pr}{3}e^0\wedge e^1.
\]
The dual horizontal fields are
\[
\begin{aligned}
 E_1&=-2\pi e_2,\\
 E_2&=-\frac{6\pi}{r}e_0-2\pi e_2,\\
 E_3&=\frac{6\pi}{r}e_1+2\pi e_2.
\end{aligned}
\]

For an explicit matrix comparison, substitute these data into
\eqref{eq:standard-matrix} and conjugate by $T_B$.
In the ordered vector basis $(e_0,e_1,e_2,e_3)$, the upper-left
and upper-right blocks of the second generator are
\[
 A_2=
 \begin{pmatrix}
  1&-1&0&0\\
  -1&-1&3/r&0\\
  0&-r&1&0\\
  0&p&0&1
 \end{pmatrix},
 \qquad
 P_2=
 \begin{pmatrix}
  0&0&0&-3/r\\
  0&0&0&0\\
  0&0&0&-1\\
  3/r&0&1&0
 \end{pmatrix}.
\]
These blocks, together with the remaining blocks, agree with
$J_{\mathrm{FG}}^T$.
The first generator similarly agrees with $I_{\mathrm{FG}}^T$,
and the third follows from $\mathcal K=\mathcal I\mathcal J$.

To describe the global parameters, set
\[
\begin{aligned}
 v_0&:=\frac{r}{6\pi}(1,-1,0),\\
 v_1&:=\frac{r}{6\pi}(1,0,1),\\
 v_2&:=-\frac1{2\pi}(1,0,0).
\end{aligned}
\]
The parameters of \cref{thm:closed-flat-arithmetic-data} are
\begin{equation}\label{eq:FG-parameters}
\begin{aligned}
 &\Lambda_r
 =\ZZ v_0+\ZZ v_1+\ZZ v_2,\\
 &e_r(v_1,v_2)=-1,\qquad
 e_r(v_0,v_1)=e_r(v_0,v_2)=0,\\
 &\mu_r
 =\int_{T^3}d\lambda\wedge\beta=-\frac r3.
\end{aligned}
\end{equation}
Indeed, $d\gamma=0$, and direct calculation gives
$\eta^{123}=\frac{r^2}{72\pi^3}e^{012}$ and $d\lambda\wedge\beta=-\frac r3 e^{012}.$
Thus varying $p$ changes only the common closed $B$-field.

Finally,
\[
 \Lambda_{-r}=\Lambda_r,\qquad
 e_{-r}=-e_r,\qquad
 \mu_{-r}=-\mu_r.
\]
Hence $Q=-\Id$ in
\eqref{eq:closed-flat-arithmetic-equivalence} identifies the
parameters $r$ and $-r$.
Conversely, that equivalence preserves $|\mu|=|r|/3$.
Therefore, on the fixed lattice $\Gamma_{\mathrm{KT}}$, two
members of this family are equivalent precisely when
$|r|=|r'|$.
\end{example}

\begin{remark}
\label{rem:type-one-calabi-yau}
Bailey, Cavalcanti and Gualtieri
\cite[Corollary~2.3]{BaileyCavalcantiGualtieri2017} show that a closed
connected four-manifold admits an untwisted type-one generalized
Calabi--Yau structure exactly when it is orientable and fibres over
$T^2$. Our constant-torus and closed-leaf flat representatives have
global closed spinors: $\theta=\kappa=0$ in
\eqref{eq:killing-spinors} gives $d\rho_i=0$.
The product $N_C\times S^1$ also fibres over $T^2$, whereas $b_1=1$
excludes the Berger family by
\cite[Corollary~1.11]{BaileyCavalcantiGualtieri2017}. The existence of one
such structure on an axis manifold does not identify it with a
specified member of our quaternionic sphere; local closed spinors
need not descend globally.
\end{remark}

\begin{proposition}\label{prop:closed-flat-finite-coverings}
A structure-preserving covering from $(\Lambda_u,e_u,\mu_u)$
to $(\Lambda_d,e_d,\mu_d)$, with $e_u,e_d\ne0$, exists precisely
when there are $m\in\ZZ\setminus\{0\}$ and
$R\in\operatorname{SO}(3)$ such that, setting
$A:=m^{-1}R$ and $\Lambda_0:=A\Lambda_u$, one has
\begin{equation}\label{eq:closed-flat-cover-data}
 \left\{
 \begin{aligned}
  \Lambda_0&\subset\Lambda_d,\\
  A^*e_d&=m e_u,\\
  \mu_u&=\delta\mu_d.
 \end{aligned}
 \right.
\end{equation}
Here
\[
 \delta:=
 \det(A)\,
 \frac{\operatorname{covol}(\Lambda_u)}
      {\operatorname{covol}(\Lambda_d)}
\]
is the signed degree of the induced base covering.
The total covering has $|m\delta|$ sheets.
It is regular, and hence a quotient by a finite group of deck
transformations, if and only if
\[
 e_d(\Lambda_d,\Lambda_0)\subset m\ZZ.
\]
\end{proposition}

\begin{proof}
Let $\Phi:M_u\to M_d$ be a structure-preserving covering, and
let $\psi:N_u\to N_d$ be its induced base covering, where
$N_u=\RR^3/\Lambda_u$ and $N_d=\RR^3/\Lambda_d$.
The normalized fibre generators satisfy
$\Phi_*X_u=mX_d$.
Applying \cref{lem:killing-parametrization} to the lifted
generator on $M_u$ shows that $m$ is constant.
The least-period normalization then gives
$m\in\ZZ\setminus\{0\}$, equal to the signed fibre degree.

Comparison of the Poisson tensors makes $\psi$ affine, with
linear part $A=m^{-1}R$ for some $R\in\operatorname{SO}(3)$.
Thus $\Lambda_0:=A\Lambda_u\subset\Lambda_d$, while the
circle-bundle lifting condition gives $A^*e_d=me_u$.
Writing $n:=[\Lambda_d:\Lambda_0]$, we obtain
\begin{equation}\label{eq:circle-cover-degrees}
\begin{aligned}
 \deg\psi
 &=\delta=\operatorname{sgn}(m)n,\\
 \#\Phi^{-1}(q)
 &=|m|n=|m\delta|.
\end{aligned}
\end{equation}
The total map preserves the coframe orientation, since
$\operatorname{sgn}(m)\operatorname{sgn}(\det A)=1$.

For invariant connections, write
$\Phi^*\lambda_d=m\lambda_u+\nu$, where $\nu$ is a one-form
on $N_u$ and pullbacks are suppressed.
The compensator calculation underlying
\eqref{eq:circle-equivalence-integral} gives
\begin{equation}\label{eq:circle-cover-integral}
\begin{aligned}
 \mathcal T_{\mathrm{eff}}
 &=\psi^*\mathcal T_d-m^{-1}d(*_h\nu),\\
 \int_{N_u}\mathcal T_{\mathrm{eff}}
 &=\delta\mu_d.
\end{aligned}
\end{equation}
Hence equality of lifting classes requires
$\mu_u=\delta\mu_d$.

Conversely, suppose that \eqref{eq:closed-flat-cover-data}
holds. The integral Euler-class identity $A^*e_d=me_u$
constructs a fibrewise degree-$m$ covering over the affine
base covering induced by $A$.
By \eqref{eq:circle-cover-integral}, its effective lifting
three-form has integral $\delta\mu_d=\mu_u$.
Equation \eqref{eq:circle-zero-mean} then identifies the
lifting data, proving sufficiency.

It remains to determine when the covering is regular.
The target fundamental group fits into the central extension
\begin{equation}\label{eq:circle-central-extension}
 1\longrightarrow\langle z\rangle
 \longrightarrow\Pi_{e_d}
 \longrightarrow\Lambda_d
 \longrightarrow1.
\end{equation}
Here $\langle z\rangle\cong\ZZ$ is the central infinite cyclic
subgroup represented by the circle fibre.
Choose its generator $z$ so that, for lifts
$\widetilde v,\widetilde w$ of $v,w\in\Lambda_d$,
\[
 [\widetilde v,\widetilde w]=z^{e_d(v,w)}.
\]
The covering subgroup
$\Pi_0:=\Phi_*\pi_1(M_u)$ projects onto $\Lambda_0$ and satisfies
\[
 \Pi_0\cap\langle z\rangle=\langle z^m\rangle.
\]
Since all commutators lie in the centre, $\Pi_0$ is normal in
$\Pi_{e_d}$ precisely when every value
$e_d(v,w)$, with $v\in\Lambda_d$ and $w\in\Lambda_0$,
is divisible by $m$. This is exactly
\[
 e_d(\Lambda_d,\Lambda_0)\subset m\ZZ.
\]

In the regular case, the deck group is $\Pi_{e_d}/\Pi_0$.
A specified covering retains the embedded subgroup $\Pi_0$,
including the central offsets of its generators; the abstract
deck group alone does not determine the covering.
\end{proof}

\begin{proposition}
\label{prop:compact-exhaustion}
The four families exhaust the compact structures. In the Berger,
axis, nontrivial circle-bundle and torus branches, respectively,
$(b_1,b_2)=(1,0),(2,2),(3,4),(4,6)$; moreover $\kappa\equiv0$
if and only if the underlying manifold admits a symplectic form.
\end{proposition}

\begin{proof}
Automatic Riemannianity and regularity of closed leaves follow from
\cref{cor:compact-riemannian,thm:closed-leaves-regular}. The preceding
classification establishes \eqref{eq:topology-exhaustion}.
The nonflat branch is oriented and taut
(\cref{thm:nonflat-riemannian-taut}) and classified by
\cref{thm:nonflat-universal-decks}. The flat alternatives follow from
\cref{thm:flat-riemannian-topology,prop:axis-group-models}, with full
classification in \cref{thm:axis-full-classification} for the axis branch,
\cref{thm:euclidean-lattice-moduli,cor:euclidean-finite-quotients} for
nonclosed taut leaves, and
\cref{thm:closed-flat-arithmetic-data,prop:closed-flat-finite-coverings}
for closed leaves.

For the finite-cover topology, an integral basis puts a nonzero
alternating form on a rank-three lattice in the form
\[
 e=d\,\alpha_1\wedge\alpha_2,\qquad d\ne0.
\]
The corresponding circle bundle is $P_d\times S^1$, where
$P_d\to T^2$ has Euler number $d$. A fibre map of signed degree
$d$ from $P_1$ to $P_d$ gives a $|d|$-sheeted
Kodaira--Thurston cover. For a Berger deck group,
\eqref{eq:nonflat-deck-generators} gives
\[
 \langle(h,T)\rangle\subset\Lambda,\qquad
 [\Lambda:\langle(h,T)\rangle]=|\Lambda_0|<\infty.
\]
This cover is the mapping torus of a right translation of
$\operatorname{SU}(2)$. A path from $1$ to $h$ gives an isotopy
to the identity, so the cover is $S^3\times S^1$.
The axis normal form gives the $N_C\times S^1$ cover at the end
of Section~\ref{sec:axis-proof}. Constant lattices, nonzero Euler
classes, $\langle(1,T)\rangle$, and axis data $v=k=0$, $n=1$
realize all four alternatives in \eqref{eq:topology-exhaustion}.

We compute the Betti numbers on the original manifolds.
For a Berger quotient, the finite kernel in
$1\to\Lambda_0\to\Lambda\to\ZZ\to1$ gives
$\operatorname{Hom}(\Lambda,\RR)\cong\RR$, hence $b_1=1$.
For a circle bundle with $e\ne0$, the Gysin sequence gives
\[
 0\longrightarrow H^1(T^3;\RR)\xrightarrow{p^*}H^1(M;\RR)
 \longrightarrow H^0(T^3;\RR)
 \xrightarrow{\smile e}H^2(T^3;\RR).
\]
The last arrow is injective, so $b_1=3$.
For an axis quotient, put $\Lambda=\Gamma\cap\RR^2$ and
$P=P_\Gamma$. The extension
\[
 1\longrightarrow\Lambda\longrightarrow\Gamma
 \longrightarrow P\longrightarrow1,\qquad
 \Lambda\cong P\cong\ZZ^2
\]
contains a conjugation action
$D=(-1)^m\operatorname{diag}(e^u,e^{-u})$, $u\ne0$.
If $\ell\in\operatorname{Hom}(\Gamma,\RR)$, then
\[
 \ell|_\Lambda\circ(D-I)=0,\qquad \det(D-I)\ne0
 \implies \ell|_\Lambda=0.
\]
Thus $H^1(M;\RR)\cong\operatorname{Hom}(P,\RR)\cong\RR^2$.
Finally $b_1(T^4)=4$. By \cref{thm:canonical-line}, all these manifolds are parallelizable.
Thus $\chi(M)=\sigma(M)=0$, and Poincar\'e duality gives
$b_2(M)=2b_1(M)-2$.
Combining this identity with the values of $b_1$ obtained above yields
\begin{equation}\label{eq:topology-betti}
 \bigl(b_1(M),b_2(M)\bigr)
 \in\{(1,0),(2,2),(3,4),(4,6)\}.
\end{equation}
It remains to prove
\begin{equation}\label{eq:topology-symplectic}
 \kappa\equiv0
 \iff b_2(M)>0
 \iff M\text{ admits a symplectic form}.
\end{equation}
The first equivalence follows from
\eqref{eq:topology-exhaustion} and \eqref{eq:topology-betti}.
For the second, a torus admits a constant symplectic form.
On a nontrivial circle bundle choose an oriented constant coframe
$\alpha$ and a connection of curvature
$F=c\,\alpha^1\wedge\alpha^2$, $c\ne0$. Define $\omega:=\lambda\wedge\alpha^1+\alpha^2\wedge\alpha^3$.
Then
\[
\begin{aligned}
 d\omega
 &=F\wedge\alpha^1=0,\\
 \omega^2
 &=2\lambda\wedge\alpha^1\wedge\alpha^2\wedge\alpha^3
 \ne0.
\end{aligned}
\]
Thus $\omega$ is a symplectic form.
On an axis quotient, \eqref{eq:axis-G-action} preserves
\[
\begin{aligned}
 \omega&:=dx\wedge dy-dt\wedge d\varphi,\\
 d\omega&=0,\qquad \omega^2=2\lambda\wedge\eta^{123}\ne0,
\end{aligned}
\]
where the last identity follows from
\eqref{eq:axis-universal-coframe}. Thus every flat branch is
symplectic. Conversely, on a closed four-manifold,
\[
 [\omega]=0\implies \omega=d\xi
 \implies \int_M\omega^2=\int_Md(\xi\wedge\omega)=0,
\]
contrary to nondegeneracy. Hence $b_2>0$, completing
\eqref{eq:topology-symplectic}.
\end{proof}

\begin{proof}[Proof of Theorem~\ref{thm:intro-compact}]
The classification and topological assertions follow from
\cref{prop:compact-exhaustion}; \cref{thm:closed-leaves-regular}
gives the final assertion on closed leaves.
\end{proof}

\appendix
\section{Positive scalar-affine dynamics}\label{app:scalar-dynamics}
The scalar-affine flow classification, including the vertical Hopf
model, is described in \cite[p.~408]{Rigal1998}. We give the
source--sink argument and smooth gluing needed here, keeping track
of the prescribed ordered transverse projective frame.

\begin{lemma}
\label{lem:nonatomic-scalar-holonomy}
An oriented scalar-affine flow of codimension $q\ge2$ on a closed
connected manifold, with positive scalar changes, is Riemannian if it admits a nonzero locally
finite holonomy-invariant transverse measure without atoms.
\end{lemma}

\begin{proof}
We prove the positive scalar-affine analogue of
\cite[Theorem~2.8, p.~131]{Ghys1991}, adapting the ball and
uniform-extension arguments of \cite[Lemmas~2.1--2.2, pp.~126--127]{Ghys1991}.
The leaf orientation and positive transverse changes orient $M$. Choose nested finite flow boxes and
let $T$ be the disjoint union of their outer affine transversals.
Shrinking the inner boxes gives $T_0\Subset T$ meeting every leaf
and finitely many generators
$g_i:D_i\to R_i$, $D_i,R_i\subset T_0$, of the restricted
pseudogroup on $T_0$, including inverses. Choose these as restrictions
of holonomy maps $\widetilde g_i$ defined on neighbourhoods of
$\overline{D_i}$ and taking values in the appropriate component of $T$.
Such a compactly extendible generating family follows by subdividing
leafwise paths using nested flow boxes; see the uniform-extension
construction in \cite[Lemmas~2.1--2.2]{Ghys1991}. In the affine
coordinates each $\widetilde g_i$ is a restriction of one map
$x\mapsto a_i x+b_i$, $a_i>0$.

Put $C_0=\overline{T_0}$. Because there are finitely many generators
and $C_0$ and all $\overline{D_i}$ are compact, we may choose
$\varepsilon>0$ such that every $\varepsilon$-ball centred on $C_0$
lies in its component of $T$ and, for each $z\in D_i$, the entire
ball $B(z,\varepsilon)$ is in the domain of $\widetilde g_i$ and its
image is in the target component of $T$.

Let $S=C_0\cap\operatorname{supp}\mu$. Holonomy invariance and the
fact that $T_0$ meets every leaf give $S\ne\varnothing$. Compactness
and the absence of atoms imply
\begin{equation}\label{eq:riem-measure-ball-bounds}
 m:=\inf_{z\in S}\mu(B(z,\varepsilon))>0,
 \qquad
 \lim_{r\downarrow0}\sup_{z\in C_0}\mu(B(z,r))=0.
\end{equation}
For the first assertion, take finitely many positive-mass balls of
radius $\varepsilon/4$ around points of $S$ whose concentric
$\varepsilon/2$-balls cover $S$. For the second, a contrary sequence
$z_n\to z$, $r_n\downarrow0$, with masses bounded below would give
an atom at $z$ by continuity from above in a relatively compact
neighbourhood. Choose $0<\delta<\varepsilon$ with
$\sup_{z\in C_0}\mu(B(z,\delta))<m$.

Fix $x\in T_0\cap\operatorname{supp}\mu$. We extend every admissible
word at $x$ to $B(x,\delta)$ by induction on its length. The induction
hypothesis is that a prefix is a holonomy map on the entire ball,
with image $B(x_j,r_j)$ contained in a component of $T$, and that
$r_j<\varepsilon$. This holds for the empty word. If the next letter
is $g_i$, then $x_j\in D_i$ and the choice of $\varepsilon$ makes
$\widetilde g_i$ a holonomy map on the whole ball $B(x_j,r_j)$, with
image in the target component of $T$. Its image is the entire ball
$B(x_{j+1},r_{j+1})$, since the map is a positive homothety.
Holonomy invariance also gives $x_{j+1}\in S$ and
\[
 \mu(B(x_{j+1},r_{j+1}))=\mu(B(x,\delta))<m.
\]
If $r_{j+1}\ge\varepsilon$, the image ball contains
$B(x_{j+1},\varepsilon)$, whose measure is at least $m$ by
\eqref{eq:riem-measure-ball-bounds}, a contradiction. This proves
the induction step without using the extension beyond its domain.
Consequently a word $h$ sends $B(x,\delta)$ to a ball of radius
$a(h)\delta<\varepsilon$.
Applying the same argument to $h^{-1}$ at $h(x)$ gives
\begin{equation}\label{eq:riem-holonomy-scale-bounds}
 \delta/\varepsilon\le a(h)\le\varepsilon/\delta.
\end{equation}
By \eqref{eq:riem-holonomy-scale-bounds}, the lifted leaf through an
adapted frame over $x$ has scale bounded above and away from zero
over $T_0$. Transport to $T_0$
from finitely many relatively compact leafwise neighbourhoods
multiplies these bounds by smooth positive functions with finite
maxima and positive minima. Relative to a smooth section of $E_T$,
the whole lifted leaf is therefore contained in
\[
 M\times[c,C]\subset E_T,\qquad 0<c<C<\infty.
\]
It is relatively compact, so \cref{prop:rigal-input}(ii) gives a
global adapted transverse parallelism and hence an invariant transverse metric.
\end{proof}

\begin{corollary}\label{cor:scalar-affine-ergodic-periodic}
In the non-Riemannian case of \cref{lem:nonatomic-scalar-holonomy}, every ergodic
invariant probability measure of a positive generator is supported
on a periodic orbit.
\end{corollary}

\begin{proof}
In flow coordinates an invariant probability disintegrates as
$dt\,d\mu$, with holonomy-invariant locally finite transverse measures.
If $\mu(\{x\})=m>0$, every holonomy translate $x_j$ of $x$ has mass
$m$. A nonperiodic leaf has infinitely many distinct intersections
$x_j$ in a relatively compact transversal near an accumulation point,
contradicting local finiteness since $\mu(\{x_1,\ldots,x_N\})=Nm$.
Hence each atom lies on a periodic
orbit, which has positive flow measure and therefore full measure
by ergodicity. If there is no atom, \cref{lem:nonatomic-scalar-holonomy}
instead implies Riemannianity.
\end{proof}

\begin{lemma}\label{lem:scalar-source-sink}
For an oriented positive scalar-affine flow on a closed connected
manifold of dimension at least three, suppose every ergodic invariant
probability is supported on a periodic orbit and no periodic orbit
has identity holonomy. Then there is exactly one
source and one sink, and all other orbits run between them.
\end{lemma}

\begin{proof}
A periodic return map is $x\mapsto ax$, $a>0$, $a\ne1$. Its orbit
is therefore a source or sink, with an isolating tube whose boundary
the flow crosses strictly outward or inward. No other periodic
orbit meets this tube. Every nonempty compact invariant set supports
an ergodic probability and hence contains a periodic orbit. If there
were infinitely many distinct periodic orbits $C_n$, the set
\[
 \Omega_\infty:=\bigcap_{N\ge1}\overline{\bigcup_{n\ge N}C_n}
\]
would be nonempty, compact and invariant. A periodic orbit in
$\Omega_\infty$ has an isolating tube disjoint from all but at most one
$C_n$, contradicting the definition of $\Omega_\infty$. Thus the periodic
orbits are finite in number.

For a nonperiodic orbit $\gamma$, the compact invariant set
$\omega(\gamma)$ contains a periodic orbit. This orbit cannot be a
source: a repelling tube cannot be reentered after $\gamma$ leaves it.
Meeting a sink instead puts $\gamma$ in its trapping tube, so
$\gamma$ converges to that sink. Reversing time yields
\[
 \alpha(\gamma)=C_-,\qquad \omega(\gamma)=C_+,
 \qquad C_-\text{ a source},\quad C_+\text{ a sink}.
\]
Removing the finitely many periodic circles leaves a connected
manifold, because their codimension is at least two. The complement
is a disjoint union of open sets, one for each source--sink pair.
Only one is nonempty. Every periodic orbit has nonperiodic neighbours,
so there is exactly one source and one sink.
\end{proof}

\begin{lemma}\label{lem:scalar-hopf-gluing}
An oriented positive scalar-affine flow on a closed connected
$(q+1)$-manifold, $q\ge3$, with one source, one sink and all other
orbits running between them is smoothly equivalent, as an oriented
one-dimensional foliation, to the vertical foliation on
$(\RR^{q+1}\setminus0)/\langle a\Id\rangle$, $a>1$.
The equivalence preserves the ordered transverse projective frame.
\end{lemma}

\begin{proof}
A periodic return map is $x\mapsto b^{-1}x$, $b>0$, $b\ne1$.
Choose a circle-valued phase near the periodic orbit. Its derivative
along a positive generator is positive after shrinking the tube;
dividing the generator by this derivative gives unit phase speed
and unit return time. The tube is thus a neighbourhood of the zero
section in
\begin{equation}\label{eq:riem-periodic-tube}
 E_b=(\RR_s\times\RR_x^q)/((s,x)\sim(s+1,bx)),
 \qquad\mathcal F=\RR\partial_s.
\end{equation}
The identification is smooth and preserves the transverse affine
charts. The function
\[
 r=e^{-s\log b}|x|,\qquad \partial_s r=-(\log b)r,
\]
is well defined on $E_b$. Thus $U_b=\{r\le\varepsilon\}$ is a
compact tube with transverse boundary, and
\[
 s=\frac{\log(|x|/\varepsilon)}{\log b},\qquad
 \partial U_b\cong
 (\RR^q\setminus\{0\})/\langle b\Id\rangle.
\]
Since $q\ge3$, the punctured affine space is the universal cover
of $\partial U_b$.

Choose disjoint source and sink tubes. Every other orbit meets each
boundary exactly once. The transit map $\mathscr T$ is a smooth
diffeomorphism and locally a holonomy map. Since both punctured
affine spaces are simply connected, $\mathscr T$ lifts to a diffeomorphism
$F:\RR^q\setminus\{0\}\to\RR^q\setminus\{0\}$. Around every point,
$F$ agrees with a scalar-affine transformation: the transit map is
locally holonomy, and changes of the lifted affine charts are scalar
affine. The coefficients of two such expressions agree wherever
their domains overlap, because both expressions represent the same
lift $F$ there. Connectedness of the punctured space therefore gives
a single global scalar-affine expression. Consequently
\[
 F(x)=cx+d,\qquad c>0,\qquad
 F(\RR^q\setminus\{0\})=\RR^q\setminus\{0\},
\]
so $d=0$. Put $a_\pm=\max(b_\pm,b_\pm^{-1})>1$.
Equivariance under the generators of the two infinite cyclic deck
groups gives
\begin{equation}\label{eq:riem-transit-dilation}
\begin{gathered}
 ca_-x=a_+^{\epsilon}cx,\qquad\epsilon\in\{1,-1\},\\
 \epsilon=1,\qquad a_-=a_+=a.
\end{gathered}
\end{equation}

The vertical Hopf model has the two tube models
\eqref{eq:riem-periodic-tube} with $b=a,a^{-1}$. For $k=\log a$,
their embeddings are
\[
 E_a\longrightarrow M_a,\quad(s,x)\longmapsto[e^{ks},x],
 \qquad
 E_{a^{-1}}\longrightarrow M_a,\quad(s,x)\longmapsto[-e^{-ks},x].
\]
They preserve leaf orientation and identify $r\le\varepsilon$ with
$|x|\le\varepsilon|t|$. The Hopf transit map is the identity in $x$.
Multiplying one tube's transverse coordinate by $c$ matches $F$ in
\eqref{eq:riem-transit-dilation} and extends smoothly over the zero
section. The tube radii may differ; a prescribed boundary is preserved
by combining this scaling with $s\mapsto s+\log c/\log b$.

Let $\Sigma_-$ be the source boundary and $\ell:\Sigma_-\to(0,\infty)$ its smooth
hitting time at the sink boundary. The intervening compact region is
\[
 \Sigma_-\times[0,1]\longrightarrow M,
 \qquad (z,u)\longmapsto\Phi^{u\ell(z)}(z).
\]
The tube maps prescribe increasing interval maps on endpoint collars,
smooth in $z$ and taking the endpoints to $0,1$. To extend them,
choose uniform small collars and cutoffs $\chi_0,\chi_1$, equal to
one on smaller collars and supported in the respective endpoint
collars. Let $p_0(z,u),p_1(z,u)>0$ denote the $u$-derivatives of
the prescribed maps on those collars, and choose a fixed nonnegative $\chi_*$
supported in the middle with $\chi_0+\chi_1+\chi_*>0$.
By compactness of $\Sigma_-$, shrinking the collars gives
\[
 A(z):=\int_0^1(\chi_0p_0+\chi_1p_1)\,du<1.
\]
Define
\begin{equation}\label{eq:riem-collar-interpolation}
\begin{aligned}
 p(z,u)&:=\chi_0p_0+\chi_1p_1
       +\frac{1-A(z)}{\int_0^1\chi_*\,du}\chi_*,\\
 F_z(u)&:=\int_0^u p(z,v)\,dv.
\end{aligned}
\end{equation}
Then $p>0$, $F_z(0)=0$ and $F_z(1)=1$. On each smaller endpoint
collar, \eqref{eq:riem-collar-interpolation} has the prescribed
derivative and endpoint value, hence equals the prescribed map.
Thus $(z,u)\mapsto(z,F_z(u))$ glues smoothly with all endpoint
derivatives matched and preserves the transverse coordinate, hence
the ordered transverse projective frame.
\end{proof}

\end{document}